\documentclass[11pt,a4paper]{article}
\usepackage{framed,tikz,pstricks,amsmath,amsfonts,amssymb,mathrsfs,graphicx,xcolor,cite,soul,hyperref,mathtools,wasysym,url,mathscinet,bigints,verbatim}
\usepackage[english]{babel}
\usepackage[T1]{fontenc}
\usepackage{amsthm}
\usepackage{graphicx}
\usepackage{xfrac}
\usepackage{subfigure}
\usepackage{pstricks}
\usepackage{pst-node}
\usepackage{tikz}
\usetikzlibrary{decorations.pathmorphing}
\usepackage{pgfplots}
\pgfplotsset{compat=1.16}

\usepackage{amsmath}
\usetikzlibrary{decorations.pathreplacing,shadows.blur}

\definecolor{brickred}{rgb}{0.8, 0.25, 0.33}
\definecolor{darkpowderblue}{rgb}{0.0, 0.2, 0.6}
\definecolor{intermedio3}{rgb}{0.25, 0.47, 0.44}
\definecolor{intermedio2}{rgb}{0.21, 0.49, 0.34}
\definecolor{intermedio1}{rgb}{0.17, 0.52, 0.24}
\definecolor{forestgreen(web)}{rgb}{0.13, 0.55, 0.13}

\mathtoolsset{showonlyrefs=true}
\definecolor{lightred}{rgb}{0.85,0,0}
\newrgbcolor{blu}{.1 0 .7}
\newtheorem{theorem}{\bf Theorem}[section]
\newtheorem{lemma}[theorem]{\bf Lemma}

\newtheorem{definition}[theorem]{\bf Definition}
\newtheorem{remark}[theorem]{\bf Remark}
\newtheorem{proposition}[theorem]{\bf Proposition}

\numberwithin{equation}{section}

\newenvironment{giurev}{\color{blue}}{\color{black}}
\newcommand{\bgr}{\begin{giurev}}
\newcommand{\egr}{\end{giurev}}

\newcommand{\denoterow}[1]{\rlap{\hspace{1em}$\leftarrow{}$ $j^{th}$}}

\newcommand{\R}{\mathbb{R}}
\newcommand{\N}{\mathbb{N}}

\def \rn {{\mathbb {R}}^{N}}   
\def \rnm {{\mathbb {R}}^{\tilde N}} 
\def \rnn {{\mathbb {R}}^{N+1}}  
\def \r {\varrho}
\def \tt {\tilde}
\def \s {\sigma}
\def \A {\mathscr{A}}

\def \GG {{\mathbb{G}}}
\def \div {{\text{\rm div}}}

\def \loc {{\text{\rm loc}}}
\def \p {\partial}

\def \ti {\times}

\def \rank {{\text{\rm rank }}}
\def \span {{\text{\rm span}}}
\def \h {\mathscr{H}}
\def \H {\mathscr{H}_\GG}
\def \HG {\mathscr{H}_{\tilde\GG}}
\def \HH{\mathscr{H}_\GG^*}
\def \a {{\alpha}}

\def \G {{\Gamma}}
\def \d {{\delta}}
\def \e {{\varepsilon}}
\def \l {{\lambda}}
\def \t {{\tau}}
\def \m {{\mu}}
\def \n {{\nu}}
\def \x {{\xi}}
\def \y {{\eta}}
\def \z {{\zeta}}
\def \phi {{\varphi}}
\def \O {{\Omega}}
\def \AA {{ \mathscr{A}_{\z} }}

\usepackage{enumitem}

\title{Fundamental solution and Harnack inequality for subelliptic evolution operators on Carnot groups}
\author{{\sc{G. Pecorella}
		\thanks{Dipartimento di Scienze Fisiche, Informatiche e Matematiche, Universit\`{a} di Modena e Reggio Emilia, Via
			Campi 213/b, 41125 Modena (Italy). E-mail: giulio.pecorella@unimore.it} \qquad 
		\sc{A. Rebucci}
		\thanks{Max-Planck Institute for Mathematics in the Sciences. E-mail: annalaura.rebucci@mis.mpg.de} 
}}

\begin{document}

\maketitle

\begin{abstract}
In this paper we study regularity properties of a class of subelliptic evolution operators. We first prove the existence of the fundamental solution by means of Levi’s parametrix method, establishing also several key properties. We then employ these results, together with some mean value formulas, to prove a maximum principle and an invariant Harnack inequality for the classical solutions to the equations under study. Our analysis critically relies on the Carnot–Carathéodory geometry naturally induced by the vector fields defining these operators.
\end{abstract}

\bigskip
\normalsize
\tableofcontents

\section{Introduction}
In this paper, by constructing the fundamental solution, we prove a strong maximum principle and an invariant Harnack inequality for classical solutions to the second order subelliptic equation
\begin{equation}\label{HG1}
	\H u(z) =\sum_{i,j=1}^{m_1} X_i\big(a_{ij}(z)X_j\big)u(z) + \sum_{i=1}^{m_1} b_i(z)X_iu(z)+ c(z)u(z)-\p_t u(z)=f(z),\qquad z\in\O,
\end{equation}
in some open set $\O\subset \rnn$. Throughout the rest of the paper, $z=(x,t)$ and $\z=(\x,\t)$ denote points in $\rn \times \R$, and $X_i$'s are smooth vector fields on $\rn$, namely
\begin{equation}\label{smooth-vector-field}
	X_i(x)=\sum_{j=1}^N \phi_j^i(x)\p_{x_j},\qquad i=1,\ldots,m_1,
\end{equation}
with $\phi_j^i\in C^\infty\left(\rn\right)$. Moreover, $A(z)=\big(a_{i,j}(z)\big)^{m_1}_{i,j=1}$ is a $m_1\times m_1$ symmetric matrix with continuous and bounded real entries that satisfies the usual \emph{uniform ellipticity condition}. We also assume that the low order coefficients $b$ and $c$ are continuous and bounded. Moreover, we suppose that $m_1<N$, and therefore $\H$ is strongly degenerate. However, we recover some regularity properties for $\H$ by assuming that the vector fields $\{X_i\}^{m_1}_{i=1}$  satisfy H\"ormander's rank condition.

In the following, we employ the notation introduced in the monograph \cite{biagi2019introduction} concerning differential operators on stratified Lie groups. Let $\mathfrak{g}$ be the Lie algebra generated by $\!\{X_i\}^{m_1}_{i=1}$ (the smallest Lie subalgebra of the Lie algebra
of smooth vector fields on $\rn\!$ containing $\!\{X_i\}^{m_1}_{i=1}$). We identify a smooth vector field $X(x)\!=\!\sum_{j=1}^N \phi_j(x)\p_{x_j}$  with the vector valued function $\big(\phi_1(x),\ldots,\phi_N(x)\big)\in \rn$, and for every $x\in\R^N$ we set $\mathfrak{g} (x):=\big\{X(x), X\in  \mathfrak{g} \big\}$. We assume that
\\\\\medskip
\noindent
{{\bf(H1)}}  \ {
\it $X_1,\ldots,X_{m_1}$ are linearly independent smooth vector fields, and 
	\vspace{-0.3cm}
\begin{equation}
	\rank \mathfrak{g}=N.
		\end{equation}}{{{\bf(H2)}} \ {\it The  Hörmander rank condition (see \cite{hormander1967hypoelliptic}) holds
\begin{equation}
	\rank \mathfrak{g}(x)=N,\qquad \forall x \in \rn,
\end{equation}
and for every $i=1,\ldots ,m_1$ we have $X_i(0)=\p_{x_i}$.}
\\\\{{\bf(H3)}}\ {
\it There exists a family of (non-isotropic) dilations $\{\d_r\}_{r>0}$ on $\rn$ of the kind
\begin{equation}\label{eq:dilH3}
	\d_r(x)=\Big(r^{\sigma_1}x_1, r^{\sigma_2}x_2,\ldots,r^{\sigma_N} x_N\Big),
\end{equation}
where $1 = \sigma_1 \leq \sigma_2 \leq \ldots \leq \sigma_N$ are integers, such that the vector fields $X_1,\ldots,X_{m_1}$ are $\d_r$-homogeneous of degree $1$, i.e.
\begin{equation*}
	X_i \left( f \big(\delta_r(x)\big) \right)=r\big(X_i f \big)  \big(\delta_r(x)\big), \quad \forall r>0,\quad x \in\rn, \quad f \in C^\infty\big(\R^N\big),\quad i=1, \ldots,m_1.
\end{equation*}}

\smallskip 

We remark that the assumptions \textbf{(H1)}, \textbf{(H2)} are independent. Indeed, the condition that $\rank \mathfrak{g}(x) = \rank \mathfrak{g}$ does not necessarily hold for every $x \in \rn$: consider for instance the vector fields $X_1=\p_{x_1}$ and $X_2=x_1\p_{x_2}+x_2\p_{x_3}$ on $\R^3$; one then has $\rank \mathfrak{g}(0)=3$ and $\rank \mathfrak{g}=4$. 

\medskip

We recall, by Theorem 1.7 in \cite{bonfiglioli2020hormander}, that the assumptions \textbf{(H1)}, \textbf{(H2)} and \textbf{(H3)} imply the existence of a Carnot group $\GG=\left(\rn,\circ,\d_{r}\right)$ (where the dilations $\{\d_r\}_{r>0}$ are the ones given by \textbf{(H3)}) such that $\{X_i\}^{m_1}_{i=1}$ are left invariant w.r.t. $\circ$, and the operator
\begin{equation}\label{sublaplacian}
	\Delta_{\GG}=\sum_{i=1}^{m_1}X_i^2,
\end{equation}
is the \emph{canonical sublaplacian} on $\GG$. Despite being degenerate, $\Delta_{\GG}$ is a \emph{hypoelliptic} operator, namely every distributional solution $u$ to $\Delta_\GG u=f$ is smooth whenever $f$ is smooth (see \cite{hormander1967hypoelliptic} for more detail).

\medskip

We associate to the smooth vector fields some function spaces, that are standard to study classical solutions to equation $\H u= f$.
\begin{definition}[Lie derivative]\label{Lie-derivative}
	Let $X=\sum_{j=1}^N \phi_j\p_{x_j}$ be a smooth vector field on $\rn$. For every $x\in \rn$ the \emph{integral path of} $X$ \emph{starting at} $x$ is the solution $\exp(sX)(x) $ to the following Cauchy problem
	\begin{equation}
		\begin{cases}
			\frac{d}{ds}\exp(sX)(x) = X \left( \exp(sX)(x)\right),\\
			\exp(sX)(x)|_{s=0}=x.
		\end{cases}
	\end{equation}
	We say that a function $u$ is Lie differentiable at $(x,t)\in\rnn$ w.r.t. $X$ if the following limit
	\begin{equation}
		Xu(x,t)=\lim_{s\rightarrow 0} \frac{u\big(\exp(sX)(x) ,t\big)-u(x,t)}{s}
	\end{equation}
	 exists and is finite.
\end{definition} 
We observe that if $\Omega$ is an open subset of $\R^{N+1}$ and $\p_{x_j}u \in C\left(\O\right)$ for every $j=1,\ldots,N$, then $u$ is Lie differentiable w.r.t.$\,X$ and its Lie derivative is given by $Xu=\sum_{j=1}^N \phi_j\p_{x_j}u$. On the other hand, the converse is not necessarily true.

We introduce the function space
\begin{equation}\label{C2}
	C^2_\GG\left(\O\right):= \Big\{ u\in C(\O):X_i u, X_iX_j u,\p_tu \in C(\O),\quad  \text{for every}\;\; i,j=1,\ldots,m_1\Big\},
\end{equation}
where the derivatives $X_iu$ and $X_iX_ju$ are meant as first and second order Lie derivatives.

We are now ready to give the definition of classical solutions to \eqref{HG1}. 
\begin{definition}[Classical solutions for $\H$]\label{Def-classical}
	Let $\Omega$ be an open subset of $\R^{N+1}$ and let $f\in C(\O)$. We say that a function $u$ is a classical solution to
	\begin{equation}
		\H u(z) =\sum_{i,j=1}^{m_1} X_i\big(a_{ij}(z)X_j\big)u(z) + \sum_{i=1}^{m_1} b_i(z)X_iu(z)+ c(z)u(z)-\p_t u(z)=f(z),\qquad z \in \O,
	\end{equation}
	if $u \in C^2_{\GG}(\O)$ and the above equation is satisfied at every point $z \in \O$.
\end{definition}
Our goal is to construct the fundamental solution to $\H$, and to prove a strong maximum principle and the invariant Harnack inequality. To this end, we improve upon a mean value formula recently derived by Pallara and Polidoro in \cite{pallara2023meanvalueformulasclassical}, which relies on the \emph{fundamental solution} $\G^*$ of the adjoint operator $\HH$.  Therefore, our first step is to prove the existence of $\G^*$ (together with regularity estimates) by adapting the classical \emph{Levi's parametrix method} (see Section \ref{Fundamental} below). To apply this method, we need to assume that the coefficients belong to a H\"older space suited to $\H$, that is defined w.r.t.$\,$a parabolic distance $d$  related to the Carnot group distance $d_X$ (see \eqref{GG-distance} and Definition \ref{H\"older continuous function} in Section \ref{Carnot} below). We therefore rely on the following assumptions on the coefficients. 
\\\\
\noindent
{{\bf(H4)}}\ {
\it For every $i,j=1,\ldots,m_1$ we have $a_{ij},b_i,c,X_ja_{ij},X_ib_i \in C\big(\rnn\big)$, and there exist $M_1,M_2>0$ and $\a\in(0,1]$ such that
\begin{equation}\label{non-divergence-form-condition}
	\begin{aligned}
		&|a_{ij}(z)|,|b_{i}(z)|,|c(z)|\leq M_1,\qquad &\forall z \in \rnn,\\
		&|a_{ij}(z)-a_{ij}(\z)|,|b_{i}(z)-b_i(\z)|,|c(z)-c(\z)|\leq M_2\,d(z,\z)^\a, \qquad &\forall z,\z\in \rnn,
	\end{aligned}
\end{equation}
\begin{equation}\label{divergence-form-condition}
	\begin{aligned}
		&\!\!\!\!\!\!\!\!\!\!\!\!|X_ja_{ij}(z)|,|X_ib_{i}(z)|\leq M_1,\qquad &\qquad\forall z \in \rnn,\\
		&\!\!\!\!\!\!\!\!\!\!\!\!|X_ja_{ij}(z)-X_ja_{ij}(\z)|,|X_ib_{i}(z)-X_ib_i(\z)|\leq M_2\,d(z,\z)^\a, \qquad &\forall z,\z\in \rnn,
	\end{aligned}
\end{equation}
where $d(z,\zeta)$ denotes the $\GG$-parabolic distance between $z$ and $\zeta$ as defined in \eqref{GG-distance}. Moreover, there exists $\l\geq 1$ such that $A$ belongs to the uniform ellipticity class 
\begin{equation}\label{def-ellipticity-class}
	\textbf{M}_\l:=\left\{M\in \text{Sym}_{m_1}(\R): \;\;\frac{1}{\l}|\xi|^2\leq \langle M(z)\x,\x\rangle \leq \l |\x|^2,\quad \forall z \in \rnn,\, \forall \x \in \R^{m_1}\right\}.
\end{equation}}
Throughout the paper, we assume \textbf{(H1)}, \textbf{(H2)}, \textbf{(H3)} and \textbf{(H4)}.

\subsection{Main results}
One of the main contributions of this paper is the existence of the global fundamental solution to \eqref{HG1}, the construction of which proceeds by adapting the classical Levi parametrix method. Along with the existence, we show Gaussian upper and lower bounds on the fundamental solution and its Lie derivatives, as well as many interesting properties (see Theorem \ref{Th-Fundamental}). Moreover, we establish the corresponding dual properties for the fundamental solution to the adjoint operator. These results on the fundamental solution are then employed to derive new regularity results for these operators.

\medskip
	
The first regularity result is a strong maximum principle for \eqref{HG1}. In order to state this result, we need to introduce the definition of \emph{attainable set}.

We say that a curve $\gamma(\t)=\big(\gamma_1(\t),\ldots,\gamma_{N+1}(\t)\big):[0,1]\to \rnn$ is \textit{$\H$-admissible} if it is absolutely continuous and solves the following differential equation
\begin{equation}
	\dot{\gamma}(\t)=\sum_{i=1}^{m_1}\omega_i(\t)X_i\big(\gamma(s)\big)+ \big(0,\ldots,0,-1\big),
\end{equation}
for a.e. $\t\in [0,1]$, where $\omega_1,\ldots,\omega_{m_1}\in L^\infty\big([0,1]\big)$.
\begin{definition}[$\H$-attainable set]
	Let $\O$ be an open subset of $\rnn$, and let $\z\in\O$. The $\H$-attainable set of $\z$ in $\O$ is
	\begin{equation}
		\AA(\O):=\Big\{z\in\O :\; \emph{there exists an $\H$-admissible curve $\gamma$ s.t. $\gamma(0)=\z$ and $\gamma(1)=z$}\Big\}.
	\end{equation}
\end{definition}
\begin{theorem}\label{Th-maximum}
	Assume that \textbf{(H1)},\textbf{(H2)},\textbf{(H3)} and \textbf{(H4)} hold, and let  $u$ be a classical solution to $\H u=f$ in some open set $\O\subset \rnn$. Suppose that $c\leq 0$, $\div_\GG b-c\geq 0$   and $f\geq 0$. If there exists $\z\in\O$ such that $u(\z)=\max_\O u\geq 0$, then
	\begin{equation}
		u(z)=u(\z)\quad \text{and}\quad f(z)=u(\z)c(z),\qquad \text{for every } z\in \overline{\AA(\O)}. 
	\end{equation}
	An analogous result holds if $u(\z)=\min_\O u\leq 0$ and $f\leq0$.  Finally, the assumption on the sign of $u(\z)$ can be dropped if $c=0$. 
\end{theorem}
\begin{remark}
	We also point out that, due to the generality of \eqref{HG1}, it is not straightforward to get rid of  assumption $\div_\GG b-c\geq 0$. However, this additional hypothesis seems to be natural in this setting, see \cite{pallara2023meanvalueformulasclassical}. In particular, Theorem \ref{Th-maximum} improves upon Theorem 2.5 in \cite{pallara2023meanvalueformulasclassical} since it includes the endpoint case $\div_\GG b-c=0$.  
\end{remark}

\medskip

The second regularity result is an invariant Harnack inequality for non-negative solutions to $\H u =0$, where $\H$ is the operator in \eqref{HG1}. In order to state the invariant Harnack inequality, we first need to introduce the future and past cylinder where we carry out our analysis. Here and in the following, we denote by $B_r(x_0)$ the \textit{$\GG$-ball of radius $r$ and center $x_0$} (defined w.r.t. the CC-distance $d_X$ introduced in \eqref{d_X} below). Correspondingly, we define the \emph{$\GG$-cylinder of radius $r$ and center $z_0=(x_0,t_0)$} as the set
\begin{equation}
	Q_r(z_0):=B_r(x_0)\times \Big(t_0-r^2,t_0\Big).
\end{equation}
Moreover, for given constants $\nu, \y, \mu, \vartheta$ with $0 < \nu < \y < \mu < 1$ and $0 < \vartheta < 1$, 
we set 
\begin{equation}
	\label{unit-box+-}
	 Q_r^+(z_0) := B_{\vartheta r}(x_0) \times \Big(t_0-\nu r^2,t_0\Big), \qquad 
	 Q_r^-(z_0) := B_{\vartheta r}(x_0) \times \Big(t_0-\mu r^2,t_0-\y r^2\Big),
\end{equation}
see Figure \ref{figure-harnack}. We use the notation $Q^+$ and $Q^-$ whenever $z_0=0$ and $r=1$. Given the definition above, we are in a position to state the following result.
\begin{theorem}\label{Th-Harnack}
Let $\Omega$ be an open subset of $\R^{N+1}$ and let $u$ be a non-negative solution to $\H u =0$ in $\Omega $ under assumptions \textbf{(H1)},\textbf{(H2)},\textbf{(H3)} and \textbf{(H4)}.  Moreover, we consider three constants $\nu, \y, \mu$ such that $0 < \nu < \eta < \mu < 1$. Then, there exist $\vartheta_0\in (0,1)$, only depending on operator $\H$, and two positive constants $r_0$, $C_H$, depending also on $\nu$, $\eta$, $\mu$, such that
\begin{equation}\label{eq:thm-harnack}
\sup_{Q^{-}_{r}(z_0)}u \leq C_H \inf_{Q^{+}_{r}(z_0)}u,
\end{equation}
for every $z_0 \in \Omega$, for every $\vartheta\leq \vartheta_0$, and for every $r >0$ such that $r \leq r_0$ and $Q_r(z_0) \subset \Omega$ .
\end{theorem}
}
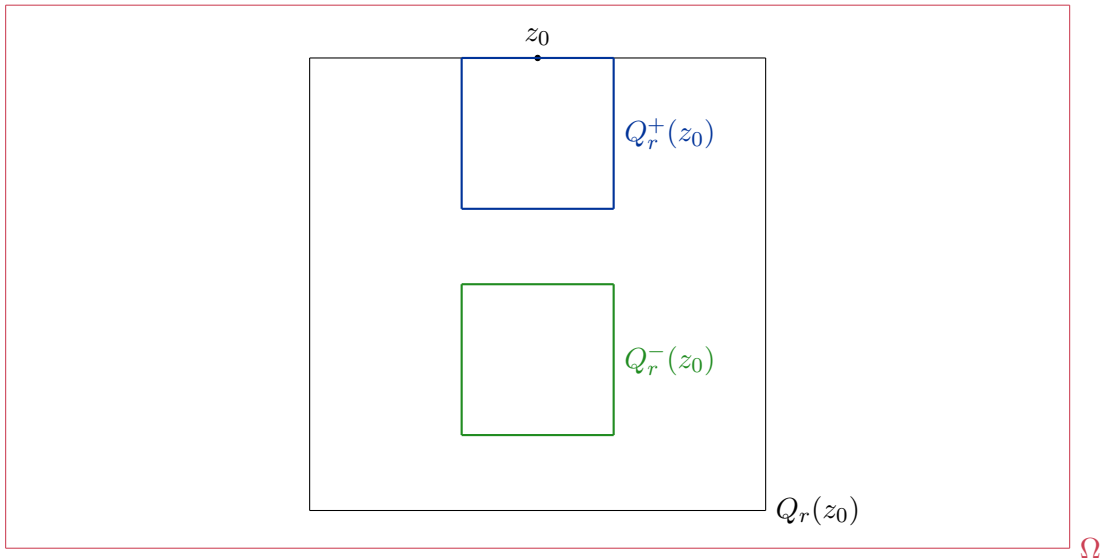
\begin{figure}\label{figure-harnack}
  \centering

  \begin{tikzpicture}[scale =1]
    
    \draw [brickred] (-7,0.7) -- (7,0.7);
    \draw [brickred] (-7,-6.5) -- (7,-6.5);
    \draw [brickred] (-7,0.7) -- (-7,-6.5);
    \draw [brickred] (7,0.7) -- (7,-6.5) node[anchor=west, scale=1]
    {$\O$};
    
    \draw [black] (-3,0) -- (3,0);
    \draw [black] (-3,-6) -- (3,-6);
    \draw [black] (-3,0) -- (-3,-6);
    \draw [black] (3,0) -- (3,-6) node[anchor=west, scale=1]
    {$Q_r(z_0)$};
    
    \filldraw (0,0) circle (1pt) node[above=1pt, scale=1] {$z_0$};
    
    \draw [darkpowderblue, thick](-1,0) -- (1, 0);
    \draw [darkpowderblue, thick](-1,-2) -- (1, -2);
    \draw [darkpowderblue, thick](-1,0) -- (-1, -2);
    \draw [darkpowderblue, thick](1,-0) -- node[anchor=west, scale=1]
    {$Q^+_r(z_0)$}(1, -2);
    
    \draw[forestgreen(web), thick] (-1,-3) -- (1, -3);
    \draw [forestgreen(web), thick](-1,-5) -- (1, -5);
    \draw [forestgreen(web), thick](-1,-3) -- (-1, -5);
    \draw [forestgreen(web), thick](1,-3)  -- node[anchor=west, scale=1]
    {$Q^-_r(z_0)$} (1, -5);

  \end{tikzpicture}

  \caption{Geometric setting of the Harnack inequality (the time variable is represented vertically, and upwardly increasing).}
\end{figure}
\begin{remark}
	Theorem \ref{Th-Harnack} also applies to classical solutions of the stationary operator
	\begin{equation}
		\mathscr{L}_\GG  =\sum_{i,j=1}^{m_1} X_i\big(a_{ij}X_j\big) + \sum_{i=1}^{m_1} b_iX_i+ c.
	\end{equation}
	Hence, the above theorem improves Theorem 1.3 in \cite{pecorella} in that the technical assumption $\div_\GG b - c\geq 0$ is removed.  Such an assumption is often included when applying the mean value formulas, but can be removed by a standard change of variables in the time variable. 
\end{remark}
\subsection{Motivation and strategy of the proof}
Before outlining the structure of the paper, we provide an overview of known results on the maximum principle and the Harnack inequality and we briefly discuss their theoretical interest. We furthermore describe the strategy of the proof of our main results.

A maximum principle for \eqref{HG1} with $C^\infty$ smooth coefficients was first obtained by Bony in \cite{bony1969principe} with a barrier argument, and then extended by Amano in \cite{amano1979maximum} for  $C^1$ coefficients. We also mention \cite{bonfiglioli2006maximum} by Bonfiglioli and Uguzzoni, where the authors proved a maximum principle for \eqref{HG1} with continuous coefficients under the additional assumption that a Picone function exists. We point out that Theorem \ref{Th-maximum} is stronger than \cite{amano1979maximum}, and does not require the existence of a Picone function as in \cite{bonfiglioli2006maximum}.

On the other hand, since Bony's work \cite{bony1969principe} several results for the Harnack inequality are known for operators with smooth coefficients. Concerning operators with H\"older continuous coefficients, we mention Bramanti, Brandolini, Lanconelli and Uguzzoni \cite{bramanti2010non}: here, the authors prove a \emph{parabolic} Harnack inequality for positive solutions to \eqref{HG1} with $c=0$, adapting the method introduced by Krylov and Safonov in \cite{krylov1981certain}. 
In this paper we also prove a parabolic Harnack inequality (see Proposition \ref{Parabolic-Harnack} below) by extending the result of \cite{bramanti2010non} to operators with $c\neq 0$. The main contribution of this paper is the derivation of an \emph{invariant} Harnack inequality expressed in terms of future and past cylinders, in the same spirit of the one obtained in \cite{malagoli} for uniformly parabolic operators. To the best of the authors’ knowledge, this result is new when dealing with operator \eqref{HG1}. This is obtained by adapting the optimal control strategy introduced in \cite{boscain}, and strongly relies on the sub-Riemannian geometric properties induced by the vector fields $\{X_i\}_{i=1}^{m_1}$.

\medskip

The proofs of the maximum principle and the parabolic Harnack inequality rely on a technique involving mean value formulas originally developed in the context of harmonic functions, though it is not straightforward to adapt this method to our case. We recall that mean value formulas for \eqref{HG1} have been recently established by Pallara and Polidoro in \cite{pallara2023meanvalueformulasclassical}, by employing the fundamental solution $\G^*$ of the adjoint operator of \eqref{HG1}. In this work the authors managed to extend to the case of H\"older coefficients the results previously known only for smooth coefficients. The main difficulty one encounters when dealing with non-smooth coefficients is due to the loss of regularity of the fundamental solution: while H\"ormander's condition theorem ensures smoothness when the coefficients are smooth, the fundamental solution fails to be even $C^1$ when the coefficients are merely H\"older continuous (as we will see in Definition \ref{Def-fundamental}). To overcome this issue, in \cite{pallara2023meanvalueformulasclassical} the authors used some of the tools developed in the last years in the field of Geometric Measure Theory on Carnot groups (see \cite{cassano2016some} for an overview of this topic). In order to employ the mean value formula in \cite{pallara2023meanvalueformulasclassical}, we first construct the fundamental solution of \eqref{HG1} and of its adjoint by means of Levi's parametrix method \cite{levi1907sulle,levi1909problemi} (see Theorem \ref{Th-Fundamental} below). We also derive several key properties of the fundamental solution (most notably, Gaussian estimates), which are then employed to prove the mean value formula in our setting.

The main difficulty encountered in extending this “mean value formula technique” to our subelliptic evolutive setting (in particular, in the proof of the parabolic Harnack inequality) lies in the behavior of the kernel appearing in the formula. More precisely, the kernel is not bounded from above and below by some strictly positive constants (an issue that already arises in the case of the heat operator on $\R^2$). This prevents us from taking advantage of the classical elementary arguments developed for harmonic functions to establish the maximum principle and the parabolic Harnack inequality. To overcome this problem, we employ the technique introduced by Kuptsov  in \cite{Kupcov}, based on Hadamard's descent method, to construct an improved bounded kernel. The resulting kernel is strictly positive (up to a set of zero measure), thereby allowing us to make use of the classical theory developed for harmonic functions. A parabolic Harnack inequality obtained via Hadamard’s descent method was first established by Garofalo and Lanconelli in \cite{GarofaloLanconelli-1989} and later by Malagoli, Pallara and Polidoro in \cite{malagoli} for uniformly parabolic operators, and by Polidoro \cite{polidoro-parametrix} and by one of the authors \cite{rebucci24} for Kolmogorov operators. Here we follow the approach introduced by Polidoro in \cite{polidoro-parametrix}, although several technical difficulties need to be addressed in order to obtain the desired bounds on the kernel. This is, in particular, due to the fact that there is no explicit expression for the fundamental solutions $\Gamma$ of the class of operators studied here. In order to tackle the problem, we employ the equivalence of sublaplacians on \emph{free Carnot groups} derived by Bonfiglioli and Uguzzoni in \cite{bonfiglioli2004families}. To make use of this equivalence on a general Carnot group $\GG$ (\emph{i.e.} whenever $\GG$ is not free), we apply the Rothschild and Stein lifting technique \cite{rothschild1976hypoelliptic}, similar to the construction by Bonfiglioli, Lanconelli and Uguzzoni in \cite{bonfiglioli2002uniform}. Thanks to this lifting procedure, we can derive the desired estimates and the parabolic Harnack inequality for the lifted operator, which in turn implies the parabolic Harnack inequality for the class of operators under study. The proof of the invariant Harnack inequality presented in Theorem \ref{Th-Harnack} requires repeated application of the parabolic Harnack inequality to construct a Harnack chain. We finally remark that the Carnot–Carathéodory distance associated to the vector fields $\{X_i\}_{i=1}^{m_1}$ plays a key role in this last step.

\subsection{Outline of the paper}
This paper is organized as follows: in Section \ref{Carnot} we  recall some known facts on Carnot groups and the Carnot-Caratheodory distance; in Section \ref{Fundamental} we prove Theorem \ref{Th-Fundamental} concerning the existence of the fundamental solution of operator $\H$, together with Gaussian upper and lower bounds and other properties; in Section \ref{Mean-Value-Formula} we introduce the mean value formula and a suitable improvement, that will allow us to prove Theorem \ref{Th-maximum} in Section \ref{Maximum}; in Section \ref{Parabolic Harnack inequality} we make use of the lifting tecnique to prove a “parabolic-type” Harnack inequality, that will be employed in Section \ref{Harnack} to prove the invariant one stated in Theorem \ref{Th-Harnack}; finally, in Appendix \ref{Comparison} we prove a comparison principle for solutions to \eqref{HG1} which is used in Section \ref{Fundamental}.

\medskip  We give some notational remarks. In the following we consistently refrain from restating the phrase  'for $i,j=1,\ldots,m_1$', as it is always implied. Moreover, $\mathcal{H}^{d}$ will denote the $d$-dimensional Hausdorff measure, $\nu$ the Euclidean inward normal unit vector and $\langle \cdot,\cdot\rangle$ the usual inner product.

\section{Preliminaries on Carnot groups}\label{Carnot}
In this section we describe some properties concerning the Carnot group that defines the subriemannian setting of \eqref{HG1}. For a more detailed exposition of this topic we refer to the monograph \cite{bonfiglioli2007stratified}. We subsequently introduce the Carnot-Caratheodory distance associated to the vector fields $X = \{ X_1, \ldots, X_{m_1} \}$ in \eqref{HG1}. We eventually describe Rothschild and Stein lifting technique \cite{rothschild1976hypoelliptic} to lift $\GG$ to a free group which preserves the homogeneous structure of $\GG$.

\medskip

Let $\GG=\big(\rn,\circ,\d_r\big)$ be the Carnot group induced by the smooth vector fields $\{X_i\}_{i=1}^{m_1}$ (whose existence is given by Theorem 1.7 in \cite{bonfiglioli2020hormander}, thanks to \textbf{(H1)},\textbf{(H2)} and \textbf{(H3)}). We first recall that a Carnot group is in particular \textit{homogeneous Lie group}, namely a group with a smooth composition law $\circ$ and a dilation law $\{\d_r\}_{r>0}$ that is an automorphism of the group, i.e.
\begin{equation}
	\d_r(x\circ y)=\big(\d_r(x)\big)\circ \big(\d_r(y)\big),\qquad \text{for every}\,\;x,y \in \rn,\;r>0.
\end{equation}
By \textbf{(H1)}, \textbf{(H2)} and \textbf{(H3)}, $\mathfrak{g}$ is the Lie algebra of the group $\GG$. Moreover, \textbf{(H3)} induces a \emph{stratification} on $\mathfrak{g}$
\begin{equation}\label{stratification}
	\mathfrak{g} = V_1 \oplus \ldots \oplus V_\kappa,
\end{equation}
where $V_1=\span \{X_i\}^{m_1}_{i=1}$, $V_{k+1}=\span \big\{[X,Y], X \in V_1, Y \in V_k\big\}$, for $k=1,\ldots, \kappa-1$, and $\big\{[X,Y], X \in V_1, Y \in V_\kappa\big\}=\emptyset$, with $[X,Y]=XY-YX$.

We observe that the stratification \eqref{stratification} implies that the dilations  $\{\d_r\}_{r>0}$ take the following form (see Remark 1.4.2 in \cite{bonfiglioli2007stratified})
\begin{equation}\label{dilations-prop}
	\d_r(x)=\left(r x^{(1)}, r^2 x ^{(2)},\ldots,r^\kappa x^{(\kappa)}\right), \qquad x^{(k)}\in \R^{m_k}, \; k=1,\ldots,\kappa,
\end{equation}
where $m_k$ denotes the dimension of $V_k$, for every $k=1,\ldots,\kappa$. We will refer to the number
\begin{equation}
	Q:=\sum_{k=1}^{\kappa} k\,m_k, 
\end{equation}
as \emph{homogeneous dimension} of $\GG$. By H\"ormander's rank condition we notice that $\sum_{k=1}^\kappa m_k = N$, inferring that $Q\geq N$. In particular, the equality holds if and only if $\GG$ is the Euclidean group. A group $\GG$ with the properties above is often referred to as Carnot group of step $\kappa$ and $m_1$ generators.

As it will be useful in Section \ref{Fundamental}, we recall that Proposition 1.3.5 in \cite{bonfiglioli2007stratified} provides us with the following characterization of smooth $\delta_r$-homogeneous vector fields: let $\delta_r$ be as in \eqref{eq:dilH3} and let $X=\sum_{j=1}^N \phi_j\p_{x_j}$ be a smooth vector field on $\rn$. Then $X$ is $\delta_r$-homogeneous of degree $n \in \R$ if and only if $\phi_j$ is a polynomial function $\delta_r$-homogeneous of degree $\sigma_j - n$ (unless $\phi_j \equiv 0$). In particular, by \textbf{(H3)} we deduce that the coefficients $\phi_j^i$ of the smooth vector fields $\{X_i\}_{i=1}^{m_1}$ are polynomial functions.

\medskip

We now introduce a homogeneous distance on $\GG$, \emph{i.e.} the \textit{$X$-control distance} (also known as the \textit{Carnot–Carathéodory distance}) associated with the family of vector fields $X = \{ X_1, \ldots, X_{m_1} \}$. This distance will play a crucial role in the proof of the maximum propagation and the Harnack inequality stated in Theorem \ref{Th-maximum} and in Theorem \ref{Th-Harnack}, respectively.
\begin{definition}[Carnot-Caratheodory distance]\label{CC-distance}
	The Carnot-Carathéodory distance (CC, in short) $d_X$ associated to the set of H\"{o}rmander vector fields $X = \lbrace X_1, \ldots, X_{m_1} \rbrace$ is defined as follows. Let $\mathcal{C}$ denote the set of the absolutely continuous maps $\gamma:[0,1] \rightarrow \R^N$ satisfying 
	\begin{equation*}
		\dot{\gamma}(\t)=\sum_{i=1}^{m_1}\alpha_i(\t)X_i\big(\gamma(\t)\big),
	\end{equation*}
	for a.e. $\t\in[0,1]$, with $\alpha=(\alpha_1,\ldots,\alpha_{m_1})$ real-valued function on $[0,1]$ such that
	\begin{equation}\label{condition-alphaj}
		 \|\a\|_{\infty}:= \sup_{\t\in[0,1]}\left(\sum_{j=1}^{m_1}|\alpha_j(\t)|^2\right)^{\frac{1}{2}}<+\infty.
	\end{equation}	
	The CC-distance is then defined as 
	\begin{equation*}\label{d_X}
		d_X(x,y):=\inf\Bigg\lbrace \int_0^1 |\a(\t)|d\t \, : \, \gamma \in \mathcal{C},\,\gamma(0)=x,\; \gamma(1)=y\Bigg\rbrace.
	\end{equation*}
	In the following we will refer to $\mathcal{C}$ as the set of \emph{$X$-trajectories on $\R^N$}, and the function $\t\mapsto\a(\t)$ will be referred to as \emph{control} of the X-trajectory $\gamma$. 
\end{definition}
We now list some properties of the CC-distance that will be useful in the sequel. 
\begin{itemize}
	\item The CC-distance can equivalently be defined in terms of \emph{X-subunit curves}, that is, by requiring $\|\a\|_{\infty}\leq 1$ (see, \emph{e.g.}, Proposition 7.12 in \cite{biagi2019introduction}).
	\item Another equivalent definition of the CC-distance is the following
	\begin{equation}\label{p=2}
		d^{(p)}_X(x,y) = \begin{cases}
			\displaystyle\inf_{\gamma \in \mathcal{C}}\left(\int_0^1 |\a(\t)|^p\, d\t \right)^{1/p}		
			& 
			\text{if} \hspace{1mm} p < +\infty, \\
			\displaystyle\inf_{\gamma \in \mathcal{C}}\sup_{\t\in[0,1]}|\a(\t)|, & \text{if} \hspace{1mm} p = + \infty.
		\end{cases}
	\end{equation}
	It is clear that $d_X \equiv d_X^{(1)}$. Moreover, the following property holds
	\begin{equation}\label{eq:d1-dp}
		d^{(p)}_X=d_X, \qquad \forall p \in [1,+\infty].
	\end{equation}
	see for instance Proposition 7.13 in \cite{biagi2019introduction}.
	\item The CC-distance associated to a system of H\"{o}rmander vector fields is a metric (see Theorem 6.22 in \cite{biagi2019introduction}). 
	
	\item There exists $k_1\geq 1$ such that 
	\begin{equation}\label{k1}
		d_X(x,\x)\leq k_1\big(d_X(x,y) + d_X(y,\x)\big),\qquad \text{for every} \;x,\x,y \in \rn.
	\end{equation}  
	If $k_1=1$, this is the classical triangular inequality; for such reason this property is usual referred to as \emph{pseudo-triangular inequality}.
	\item The CC-distance is $\d_r$-homogeneous of degree $1$, hence the Euclidean metric is continuos w.r.t. the metric induced by the CC-distance (see Proposition 5.15.1 in \cite{bonfiglioli2004fundamental}).	
\end{itemize}
In order to study the evolution operator \eqref{HG1} on the Carnot group $\GG$ (that is relevant to the space variable $x\in\rn$), we introduce the  $\GG$-parabolic distance on $\rnn$ 
\begin{equation}\label{GG-distance}
	d(z,\z)=d\big((x,t),(\x,\t)\big):=d_X(x,\x)+ |t-\t|^\frac{1}{2},\qquad \forall z=(x,t),\z=(\x,\t)\in\rnn.
\end{equation}
We are now in a position to define the Hölder spaces naturally induced by operator $\H$.
\begin{definition}[H\"older continuous function]\label{H\"older continuous function}
	Let $\a \in (0,1]$. We say that  $u:\rnn\rightarrow \R$ is an $\a$-H\"older continuous function with exponent $\a$ in $\rnn$ if there exists a positive constant $c>0$ such that 
	\begin{equation}
		\left|u(z)-u(\z)\right|\leq c\,d(z,\z)^\a,\qquad \forall z,\z \in \rnn.
	\end{equation}
\end{definition}
This definition of H\"older space plays a central role in our analysis: in Section \ref{Fundamental} we prove that, under the assumption that the coefficients belong to this H\"older space (namely \textbf{(H4)}), we are able to apply the parametrix method. Hence, we are able to prove the existence of the fundamental solution $\G$ to $\H$ as well as the fundamental solution $\G^*$ to $\HH$, along with upper and lower bounds. This, in turn, enables us to use the mean value formula developed in \cite{pallara2023meanvalueformulasclassical}, which we then refine in Section \ref{Mean-Value-Formula} and exploit in the proof of Theorem \ref{Th-maximum} and Theorem \ref{Th-Harnack}. Mean value formulas hold for divergence form operators: accordingly, we introduce the differential operators $\div _{\GG}$ and $\nabla _{\GG}$, that are the natural extension of the divergence and gradient operator in the non-Euclidean setting of $\GG$. In order to define these operators, we agree to identify a \emph{horizontal section} $F =\sum_{i=1}^{m_1} F_i X_i$ with its canonical coordinates $F =(F_1,\ldots,F_{m_1})$; moreover, we define as $C^1_{\GG}$ the space of functions with continuous first order Lie derivatives. Then, for every $F \in C^1_{\GG}\big(\rnn,\R^{m_1}\big)$  and for every real valued $f \in C^1_{\GG}\big(\rnn\big)$ we set 
\begin{equation}\label{divG}
	\div_{\GG}F:= \sum_{i=1}^{m_1} X_i F_i,\qquad \nabla_{\GG}f:=\sum_{i=1}^{m_1} \big(X_i f\big)X_i.
\end{equation}
Employing these operators, we can write \eqref{HG1} in $\GG$-divergence form as follows
\begin{equation}\label{HG1-div}
	\H u =\div_\GG \big(A\nabla_\GG u\big)+ \langle b,\nabla_\GG u\rangle + c\,u-\p_t u=f.
\end{equation}

\medskip

We conclude this section by introducing the lifting technique that will be employed in Section \ref{Parabolic Harnack inequality} to establish a parabolic-type Harnack inequality.  We first recall that a Carnot group $\GG$ is a \textit{free Carnot group} if the Lie algebra of $\GG$ is isomorphic to a free nilpotent Lie algebra with $r$ generators of step $s$, for some $r \geq 2$ and $s \geq 1$. We then let $\{X_i\}_{i=1}^{m_1}$ be the system of vector fields generating the Carnot group $\GG$ of  homogeneous dimension $Q$  and step $\kappa$. Following Theorem 8.3 in \cite{bonfiglioli2002uniform}, $\GG$ can be lifted to a free Carnot group $\tt \GG$ in the sense of Rothschild and Stein \cite{rothschild1976hypoelliptic}. More precisely, there exists $\tilde{N}\in \N$ and a free Carnot group $\tilde{\GG}=\big(\rnm\times \rn,\tt \circ,\tt \d_r\big)$ with homogeneous dimension $\tt Q$ and Jacobian basis $\big\{\tilde X_i\big\}_{i=1}^{m_1}$, such that the following lifting property holds: let $\pi : \mathbb{R}^{\tt N} \times \mathbb{R}^{{N}} \to \mathbb{R}^N$ denote the canonical projection onto the last $N$ coordinates, then 
\begin{equation}\label{lifting}
	\tt X_i u\big(\pi(\tt x, x)\big) = X_i u(x),\qquad (\tt x, x) \in \rnm \times \rn,
\end{equation}
for every $u \in C^\infty\big(\mathbb{R}^N\big)$ and for every $i=1.\ldots,m_1$. We now choose a suitable homogeneous distance on  $\tilde \GG$. Since the Lie algebra of $\tt\GG$ is stratified as in \eqref{stratification}, we have the following decomposition 
\begin{equation}
	\rnm\times \rn = \R^{\tt m_1}\ti\R^{\tt m_2}\ti\ldots\ti \R^{\tt m_\kappa}\ti\R^{m_1}\ti\R^{m_2}\ti\ldots\ti \R^{m_\kappa},
\end{equation}
for some $\tt m_1,\ldots,\tt m_\kappa \in \N$, where $ m_1,\ldots, m_\kappa$ are as in \eqref{dilations-prop}. Accordingly, we decompose any $(\tt x, x)\in\rnm \times\rn$ as
\begin{equation}
	\left(	\tt x^{(1)},\tt x ^{(2)},\ldots,\tt x^{(\kappa)}, x^{(1)},x ^{(2)},\ldots, x^{(\kappa)}\right) \qquad \tt x^{(k)}\in \R^{\tt m_k}, x^{(k)}\in \R^{m_k}, \; k=1,\ldots,\kappa,
\end{equation}
and we set
\begin{equation}\label{tilde-distance}
	d_{\tt \GG}(\tt x,x) := \left(\sum_{i=1}^\kappa \sum_{j=1}^{\tt m_i} |\tt x_j^{(i)}|^{\frac{2}{i}}\right)^{\frac{1}{2}} \!\!+ d_X(x,0), 
\end{equation} 
where $d_X$ is the CC-distance on $\GG$ introduced in Definition \ref{CC-distance}. 

If $\GG$ is a free Carnot group, the following remarkable property holds: there exists a diffeomorphism $T_A : \R^N \rightarrow \R^N$ such that 
\begin{equation*}
	\G_A\big((x,t),(\xi,\tau)\big)=J_A\, \G_0\Big(T_A\big(\x^{-1}\circ x	\big),t-\t\Big),
\end{equation*}
where $\G_0$ is the fundamental solution of the canonical heat operator $\Delta_{\GG}-\partial_t$, $\G_A$ is the fundamental solution of of the constant coefficients operator 
\begin{equation}\label{marina di acate}
	\mathscr{H}_A=\sum_{i,j=1}^{m_1}a_{ij}X_iX_j-\p_t,
\end{equation}
and $J_A$ is the Jacobian determinant of $T_A$. It is clear that $T_A$ acts as a classical change of basis in our subelliptic setting. We remark that it is possible to derive ad hoc uniform estimates for the diffeomorphism $T_A$, see for instance \cite{bonfiglioli2004families}. Hence, lifting $\GG$ to  $\tt \GG$ allows us to take advantage of the previous correspondence for the lifted fundamental solution $\tilde{\G}_A$ (see equation \eqref{mazzareddi}) to derive precise estimates for $\tilde{\G}_A$.

\section{Fundamental solution}\label{Fundamental}
In this section we prove the existence and some useful properties of the fundamental solution $\G$ associated to operator $\H$. As previously mentioned, the main tool that allows us to employ the mean value formula in \cite{pallara2023meanvalueformulasclassical} and to prove our main results is the fundamental solution $\G^*$ of the adjoint operator 
\begin{equation}\label{HG1star}  
	\HH u =\sum_{i,j=1}^{m_1} X_i\big(a_{ij}X_j\big)u - \sum_{i=1}^{m_1} b_iX_iu+ \left(c-\sum_{i=1}^{m_1}X_ib_i\right)u+\p_t u.
\end{equation}
We give now the definition of fundamental solution for \eqref{HG1} and for \eqref{HG1star}.
\begin{definition}[Fundamental solution]\label{Def-fundamental}
	We say that the real valued function $\G$ defined on $\big(\rnn\times\rnn\big)\setminus \big\{(z,z)\in \rnn\times \rnn\big\}$ is a fundamental solution for \eqref{HG1} if:
	\begin{enumerate}
		\item for every $\z\in\rnn$ the function $ \G(\cdot,\z)$ is a classical solution to $ \H\G(\cdot,\z)\!=\!0$ in $\rnn \setminus \{\z\}$, and belongs to $L^1_\loc \big(\rnn \setminus \{\z\}\big)$;
		\item for every $\phi \in C_c\big(\rn\big)$ the function
		\begin{equation}\label{delta}
			u(x,t):=\int_{\rn} \G(x,t,\xi,\t)\phi(\xi)d\xi,\qquad (x,t) \in \rn,
		\end{equation}
		is a classical solution to $\H u=0$ in $\rn\times (\t,+\infty)$, and satisfies 
		\begin{equation}
			\lim_{(x,t)\to(\x,\t)} u(x,t) = \phi(\x),\qquad\text{for every}\;\;\x\in\rn. 
		\end{equation} 
	\end{enumerate}
	We say that $\G^*$ is a fundamental solution for \eqref{HG1star} if it satisfies the dual statement.
\end{definition}
We now focus our attention on proving the existence of $\G$ and $\G^*$ under the assumption that the coefficients are H\"older continuous. In addition, we establish several fundamental properties that will play a crucial role in our analysis.
\begin{theorem}[Fundamental solution]\label{Th-Fundamental}
	Assume that \textbf{(H1)},\textbf{(H2)},\textbf{(H3)} and \textbf{(H4)} hold. Then
	\begin{enumerate}
		\item There exists a fundamental solution $\G$ to \eqref{HG1}.\label{item1}
		\item The following Gaussian upper bounds hold for every $T>0$ and $0<t-\tau\leq T$ 
		\begin{equation}\label{Gaussian-estimates}
			\G(x,t,\x,\t)\leq  \frac{c_T}{\big(t-\t\big)^{\frac{Q}{2}}} \;\exp{\left(-c_u\frac{d_X(x,\x)^2}{(t-\t)}\right)},
		\end{equation}
		\begin{equation}\label{Gaussian-estimates-1}
			|X_i\G(x,t,\x,\t)|\leq\frac{ c_T }{\big(t-\t\big)^{\frac{Q+1}{2}}}\exp{\left(-c_u\frac{d_X(x,\x)^2}{(t-\t)}\right)},
		\end{equation}
		\begin{equation}\label{Gaussian-estimates-2}
			|X_iX_j\G(x,t,\x,\t)|,|\p_t\G(x,t,\x,\t)|\leq \frac{ c_T }{\big(t-\t\big)^{\frac{Q+2}{2}}}\exp{\left(-c_u\frac{d_X(x,\x)^2}{(t-\t)}\right)},
		\end{equation}
		for some $c_u>0$ (depending only on structural assumptions) and $c_T>0$ (depending also on $T$). \label{item2}
		\item For any $(x,t),(\x,\t) \in \rnn$ with $t\leq \t$, we have that $\G(x,t,\x,\t)=0$. \label{item3}
		\item Let $f$ be a continuous function on $\rn \times (T_1,T_2)$ such that for every $x,\x$ belonging to a compact set $\mathcal K \subset  \rn$, and for every $t \in (T_1,T_2)$, there exists $c_\mathcal K=c_\mathcal K(\mathcal K)>0$ and $0<\y\leq 1$ satisfying
		\begin{equation}
			|f(x,t)-f(\x,t)|\leq c_\mathcal K\, d_X(x,\x)^\y.
		\end{equation}
		Moreover, suppose that $f$ has an exponential growth uniform in $t$, namely that there exist $h_1,h_2>0$ such that
		\begin{equation}\label{f-growth-exp}
			|f(x,t)|\leq h_2\,\exp\Big(h_1 d_X(x,0)^2\Big).
		\end{equation}
		Let $g$ be a continuous function on $\rn$ with the same exponential growth
		\begin{equation}\label{g-growth-exp}
			|g(x)|\leq h_2 \, \exp\Big(h_1 d_X(x,0)^2\Big).
		\end{equation}
		Then, there exists $T_1<T<T_2$ (only depending on the growth constant $h_1$) such that the function
		\begin{equation}\label{u-conv}
			u(x,t):=\int_{\rn}\G(x,t,\xi,T_1)g(\xi)d\xi - \int_{T_1}^t\int_{\rn}\G(x,t,\xi,\t)f(\xi,\t)d\xi d\t,
		\end{equation}
		is a classical solution to
		\begin{equation}\label{cauchy}
			\begin{cases}
				\H u = f,\qquad &\mathrm{in}\;\rn\times(T_1,T),\\      
				u(x,T_1)=g(x),\qquad &\mathrm{in}\;\rn,
			\end{cases}
		\end{equation}
		as it solves $\H u = f$ according to Definition \ref{Def-classical}, and it takes the initial condition in the following sense
		\begin{equation}\label{datoiniziale}
			\lim\limits_{t\rightarrow T_1}\left( \int_{\rn}\G(x,t,\xi,T_1)g(\xi)d\xi - \int_{T_1}^t\int_{\rn}\G(x,t,\xi,\tau)f(\xi,\tau)d\xi d\tau\right)=g(x).
		\end{equation} 
		\label{item4}
		\item $\G$ is nonnegative, and moreover there exists $S>0$ (depending only on structural assumptions) such that $\sup_{t\geq \t+1} \G(x,t,\x,\t)=S$.\label{item5}
		\item If $c(z)=c$ is constant, then 
		\begin{equation}\label{normalization}
			\int_{\rn}\G(x,t,\x,\t)d\x = \exp\big(c\,(t-\t)\big),\qquad \forall x \in\rn, \t <t.
		\end{equation}
		\label{item6}
		\item The reproduction property holds: for every $x,\x\in\rn$ and $t,\t\in\R$ with $\t<s<t$ we have
		\begin{equation}\label{reproduction}
			\G(x,t,\x,\t) = \int_{\rn} \G(x,t,y,s)\G(y,s,\x,\t)dy.
		\end{equation}
		\label{item7}
		\item The following Gaussian lower bound holds for every $T>0$ and $0<t-\tau\leq T$ 
		\begin{equation}\label{Gaussian-estimates-lower}
			\G(x,t,\x,\t)\geq  \frac{\tilde c_T}{\big (t-\t\big)^{\frac{Q}{2}}} \;\exp{\left(-c_l\frac{d_X(x,\x)^2}{(t-\t)}\right)},
		\end{equation}
		for some $c_l,\tilde c_T>0$ (depending on structural assumptions and $T$).\label{item8}
		\item If the exponential growth of $f$ and $g$ in \eqref{f-growth-exp} and \eqref{g-growth-exp} satisfies
		\begin{equation}\label{h1}
			\min \left(\frac{c_u}{2h_1(2k_1^2-1)},\frac{3c_u}{h_1}\right)>T.
		\end{equation}
		where $T$ is the constant given by \emph{\ref{item4}.}, $c_u$ is the same as in \emph{\ref{item2}.}, and $k_1$ is given by \eqref{k1}, then  there exist $C,M>0$ such that for every $(x,t)\in\rn \times (T_1,T)$ the function $u$ defined in \eqref{u-conv} verifies 
		\begin{equation}\label{egrowth}
			|u(x,t)|\leq C \exp\Big(Md_X(x,0)^2\Big).
		\end{equation}
		Furthermore, there exists at most one classical solution $u$ to \eqref{cauchy} such that 
		\begin{equation}\label{integralestgr}
			\int_{T_1}^{T}\int_{\rn} |u(x,t)|\exp\Big(-k\,d_X(x,0)^2\Big)dxdt<+\infty,
		\end{equation}
		for some $k>0$.\label{item9}
		\item There exists a fundamental solution $\G^*$ to \eqref{HG1star} satisfying the dual properties of the previous statements. Moreover, it holds
		\begin{equation}\label{index}
			\G^*(x,t,\x,\t)=\G(\x,\t,x,t),\qquad \text{for every }\;\;(x,t),(\x,\t)\in\rnn, (x,t)\neq (\x,\t).
		\end{equation} \label{item10}
	\end{enumerate}
\end{theorem}
As mentioned above, the existence of the fundamental solution (and more precisely properties \emph{\ref{item1}.-\ref{item4}.}) follows from the  \emph{parametrix method} introduced by Levi in \cite{levi1907sulle} and \cite{levi1909problemi}, and here adapted to our subelliptic evolutive setting. Let us briefly outline the parametrix technique: for every $\z \in \rnn$,  we evaluate the coefficients of the principal part matrix $A$ at the point $\z$, and we denote by $\G_{A_\z}=\G_{A_\z}(\cdot,\zeta)$ the fundamental solution of the following \textit{frozen} operator
\begin{equation}\label{HGfreeze}
	\h_{A_\z}=\sum_{i,j=1}^{m_1} a_{ij}(\z)X_iX_j-\p_t.
\end{equation}
We call parametrix the function $Z(z,\z):=\G_{A_\z}(z,\z)$. We look for a fundamental solution $\G$ as a solution to the following equation
\begin{equation}\label{G=Z+J}
	\G(z,\z)=Z(z,\z)+J(z,\z),
\end{equation}
where $J$ is a \emph{correction term} that we assume of the following form
\begin{equation}\label{J}
	J(z,\z)=\int_\t^t\int_{\rn} Z(x,t,y,s)G(y,s,\x,\t)dyds,
\end{equation}
for some unknown function $G$. In particular, we construct a function $G$ such that $\G$ is a classical solution to
\begin{equation}\label{Gsol}
	\H\G(z,\z)=\H Z(z,\z)+\H J(z,\z)=0,\qquad \text{for every } z\neq \z.
\end{equation}
Moreover, assuming that we can differentiate $J$ under the integral sign, we get
\begin{equation}\label{LJ}
	\H J(z,\z)=\int_\t^t\int_{\rn} \H Z(x,t,y,s)G(y,s,\x,\t)dyds-G(z,\z).
\end{equation}
Combining \eqref{Gsol} and \eqref{LJ} we get the following equation for $G$
\begin{equation}\label{G}
	G(z,\z)= \H Z(z,\z) + \int_\t^t\int_{\rn} \H Z(x,t,y,s)G(y,s,\x,\t)dyds,
\end{equation}
which is a Fredholm integral equation of the second kind. We solve \eqref{G} by means of successive approximation method, obtaining
\begin{equation}\label{g}
	G(z,\z)=\sum_{k=1}^\infty \big(\H Z\big)_k(z,\z),
\end{equation}
where $\big(\H Z\big)_1:=\H Z$, and 
\begin{equation}\label{hgk}
	\big(\H Z\big)_{k+1}(z,\z):=\int_\t^t\int_{\rn} \big(\H Z\big)_1(x,t,y,s)\big(\H Z\big)_k(y,s,\x,\t)dyds.
\end{equation}
To prove the well-posedness of the parametrix method, we must verify that the integral \eqref{hgk} is properly defined, that the series \eqref{g} converges, and that the identity \eqref{LJ} is satisfied, thereby ensuring the validity of \eqref{Gsol}. As it can be seen from \eqref{Gaussian-estimates-2} and \eqref{hgk}, the H\"older continuity assumption is crucial to address the existence of some singular integrals.

Once the existence of the fundamental solution is established, the proof of properties \emph{\ref{item5}.-\ref{item10}.} follows from classical PDEs techniques. We emphasize that the proof of the lower bound \eqref{Gaussian-estimates-lower} strongly depends on the parametrix  construction of $\G$. Indeed, this method enables us to establish a \emph{local} lower bound (see \eqref{Gaussian-estimates-lower-local} below), that, by means of the technique developed in \cite{jerison-calle}, we are able to turn into a \emph{global} one.

We finally remark that the proof of properties \emph{\ref{item1}.-\ref{item8}.} still goes through if we consider the non-divergence form operator
\begin{equation}\label{non-div}
	\h u = \sum_{i,j=1}^{m_1} \bar a_{ij}X_iX_ju+\sum_{i=1}^{m_1}\bar b_{i}X_iu+\bar c\,u,
\end{equation}
whenever $\bar a$, $\bar b$ and $\bar c$ are continuous and satisfy \eqref{non-divergence-form-condition}.

\medskip\newenvironment{pf-fundamental}{\noindent { \textit{Proof of Theorem} \ref{Th-Fundamental}.}}{\hfill $\square$}
\begin{pf-fundamental}
	We recall for reader's convenience that we aim to prove the existence of the fundamental solution to operators \eqref{HG1} and \eqref{HG1star}, \emph{i.e.}
	\begin{equation}\label{HG2}
		\begin{aligned}
			&\H u =\sum_{i,j=1}^{m_1} a_{ij}X_iX_ju +\sum_{i=1}^{m_1}\left( \sum_{j=1}^{m_1}X_ja_{ij} + b_i\right)X_iu+ cu-\p_t u,\\
			&\HH u =\sum_{i,j=1}^{m_1} a_{ij}X_iX_ju + \sum_{i=1}^{m_1}\left( \sum_{j=1}^{m_1}X_ja_{ij} - b_i\right)X_iu+ \left(c-\sum_{i=1}^{m_1}X_ib_i\right)u+\p_t u. 
		\end{aligned}
	\end{equation}

	\medskip\emph{\ref{item1}.-\ref{item5}.}

	Since the proof of \emph{\ref{item1}.-\ref{item5}.} follows very closely the one presented in \cite{bonfiglioli2004fundamental}, here we only give a sketch of it and we refer the interested reader to \cite{bonfiglioli2004fundamental} for details. In the following we focus our attention on to the fundamental solution $\G$ of $\H$. Clearly, the fundamental solution $\G^*$ of the adjoint equation $\HH$ satisfies the dual statement.

	The starting point is to show the existence of a smooth fundamental solution to \eqref{HGfreeze}, hence the existence of the parametrix $Z$. This is achieved thanks to Theorem 2.5 in \cite{bonfiglioli2002uniform}. To be more precise, Theorem 2.5 in \cite{bonfiglioli2002uniform} states that the fundamental solution of constant coefficient operators (hence the parametrix) exists and satisfies properties \emph{\ref{item2}.-\ref{item10}.} with upper and lower Gaussian bounds
	\begin{equation}\label{gaussian-parametrix}
		\frac{1}{c_\l\big(t-\t\big)^{\frac{Q}{2}}}\exp{\left(-c_\l\frac{d_X(x,\x)^2}{(t-\t)}\right)}\leq \G_{A_\z}(z,\z) \leq  \frac{c_\l}{\big(t-\t\big)^{\frac{Q}{2}}}\exp{\left(-\frac{d_X(x,\x)^2}{c_\l(t-\t)}\right)},
	\end{equation}
	that hold for every $z,\z \in \rnn$. In addition, similar Gaussian bounds hold also for its derivatives. Crucially, $c_\l$ depends only on the ellipticity constant $\lambda$ given by \textbf{(H4)}. In particular, the Gaussian estimates for the parametrix are \emph{uniform}, as they do not depend on the point $\z\in\rnn$ where we compute \eqref{HGfreeze}.

	The core principle underlying Levi's method is the following observation: since $Z(z,\z)$ is the fundamental solution associated to \eqref{HGfreeze}, it holds
	\begin{equation}\label{H-Levi}
		\H Z(z,\z)\!=\!\!\!\sum_{i,j=1}^{m_1}\!\!\big(a_{ij}(z)-a_{ij}(\z)\big)X_iX_jZ(z,\z)+\sum_{i=1}^{m_1}\!\left(\sum_{j=1}^{m_1}X_ja_{ij}(z) + b_i(z)\right)\!\!X_iZ(z,\z)+c(z)Z(z,\z).
	\end{equation}
	Thanks to the H\"older continuity assumption, the Gaussian upper bounds ensure the integrability required to make the integral in \eqref{hgk} well-defined, and to guarantee that the series in \eqref{g} converges and provides a solution to \eqref{G} (see Proposition 2.1 and Corollary 2.2 in \cite{bonfiglioli2004fundamental}). Moreover, the following estimate 
	\begin{equation}\label{G-estimate-1}
		|G(z,\z)|\leq \frac{c_T (t-\tau)^{\frac{\a}{2}-1}}{(t-\t)^{\frac{Q}{2}}}\exp{\left(-\frac{d_X(x,\x)^2}{c(t-\t)}\right)},\\
	\end{equation}
	holds true for some $c>0$ depending only on structural assumptions, while $c_T$ depends also on $T$. By \eqref{gaussian-parametrix} and \eqref{G-estimate-1}, together with the Lagrange mean value theorem on integral paths, we can prove that $\G$ is 'almost' a fundamental solution (see Theorem 2.7 in \cite{bonfiglioli2004fundamental} for further details). More precisely, for every $\z\in \rnn$  the function $\G(\cdot,\z)$  belongs to the weaker regularity space $C_{x,t}^2$, which is the space of continuous functions $w$ such that $w(\cdot,t)$ has continuous Lie derivatives $X_iX_j$ for every fixed $t$, and $w(x, \cdot)$ has continuous time derivative $\p_t$ for every fixed $x$. Moreover, $\H\G(z,\z)=0$ in $\rnn \setminus \{\z\}$, and if $f$ and $g$ satisfy the properties of statement \emph{\ref{item4}.}, it follows that the function $u$ defined in \eqref{u-conv} belongs to $C_{x,t}^2$ and solves the Cauchy problem \eqref{cauchy} (see Theorem 3.1 in \cite{bonfiglioli2004fundamental}). 
	By \eqref{G-estimate-1} and the Gaussian estimates for the parametrix, we also infer that $\G$ satisfies the Gaussian upper bounds \eqref{Gaussian-estimates}, \eqref{Gaussian-estimates-1} and \eqref{Gaussian-estimates-2}.
	
	Hence, to complete the proof of \emph{\ref{item1}.-\ref{item4}.} we need to recover the $C^2_\GG$ regularity for $\G$ (since it implies the same regularity for the function $u$ defined in \eqref{u-conv}). This regularity result follows from a combination of the comparison principle stated in  Proposition \ref{Prop-comparison} (see Appendix \eqref{Comparison} below), together with a regularization argument and the Schauder estimates proved in \cite{bramantibrandolinischauder} (see Corollary 4.8, Corollary 4.9 and the proof of Theorem 1.2 in \cite{bonfiglioli2004fundamental} for more details). In addition, the comparison principle implies also \emph{\ref{item5}.} (see for instance Proposition 4.4 in \cite{bonfiglioli2004fundamental}).

	\medskip\emph{\ref{item6}.}

	We prove \eqref{normalization} when $c=0$, the general case follows by a standard change of variable.

	Firstly, we show that for every $\phi\in C_c^\infty\big(\rnn\big)$ such that $\sup_{(y,s),(\x,\t)\in \mathrm{supp}(\phi)}\ |s-\t|<\e$, and $0\leq \phi \leq 1$, it holds
	\begin{equation}\label{w-la-coop}
		u(x,t):=\int_{\rnn}\G(x,t,\x,\t)\phi(\x,\t)d\x d\t \leq \e,\qquad \forall (x,t)\in\rnn.
	\end{equation}  
	Suppose, without loss of generality, that 
	\begin{equation}\label{gnocchi-di-patate-spesotti}
		0=\min \Big\{s\in\R : (y,s)\in \mathrm{supp}(\phi)\, \text{ for some } y \in\rn\Big\}.
	\end{equation}	
	Since $\phi$ is smooth with compact support, by \emph{\ref{item4}.} we infer that $\H u(x,t) = -\phi(x,t)$, by \emph{\ref{item3}.}  and \eqref{gnocchi-di-patate-spesotti} we deduce that $u(x,t)=0$ for every $t\leq 0$, and by \eqref{Gaussian-estimates} we see that $u$ vanishes as $|x|\to +\infty$, thanks to the continuity of the Euclidean metric w.r.t. the CC-metric. Thus, for every $(x,t)\in\rnn$ we set $v(x,t):=t-u(x,t)$, and we deduce that it solves 
	\begin{equation}
		\begin{cases}
			\H v(x,t) = -1+\phi(x,t)\leq 0,&\qquad \mathrm{in}\;\rn\times(0,\e), \\
			v(x,0) = 0,&\qquad \mathrm{in}\;\rn,\\
			\displaystyle\lim_{|x|\to+\infty}v(x,t) = t\geq 0,&\qquad \mathrm{in}\;(0,\e). 
		\end{cases}
	\end{equation}
	By comparison principle Proposition \ref{Prop-comparison}, we infer that $v(x,t)\geq 0$, thus implying $u(x,t)\leq t$ and therefore the validity of \eqref{w-la-coop}, since $t<\e$. If $t\geq\e$, by \eqref{gnocchi-di-patate-spesotti} and the assumption on the support of $\phi$ we infer $u(x,t)=0$, and then \eqref{w-la-coop} follows.

	Next, we want to show that \eqref{w-la-coop} implies
	\begin{equation}\label{less-then-normal}
		\int_{\rn}\G(x,t,\x,\t)d\x \leq 1,
	\end{equation}
	for every $x\in \rn$ and $t>\t$. Indeed, for every $\e>0$, let $(\phi_n)_{n\in\N}\in C_c^\infty\big(\rn\times(\t,\t+\e)\big)$ be a sequence of function with $0\leq \phi_n\leq 1$, pointwise converging to the characteristic function of $\rn\times(\t,\t+\e)$. By \eqref{w-la-coop}, as $n$ goes to infinity we get
	\begin{equation}
		\frac{1}{\e}\int_\t^{\t+\e}\int_{\rn}\G(x,t,\x,\t)d\x \leq 1,
	\end{equation}
	implying \eqref{less-then-normal} by Lebesgue differentiation theorem.

	Finally, we conclude the proof of \emph{\ref{item6}.} by showing that \eqref{less-then-normal} is indeed an equality. For every $T>0$ we set
	\begin{equation}
		w(x,t):=t-\int_0^{t}\int_{\rn}\G(x,t,\x,s)d\x ds,\qquad (x,t)\in \rn\times(0,T).
	\end{equation}
	By the same arguments previously employed, we see that $w$ is a nonnegative function satisfying $\H w=0$ in $ \rn\times(0,T)$, and $w(x,0)=0$ for every $x\in\rn$. From the comparison principle we deduce that $w=0$, that implies
	\begin{equation}
		\int_0^t \left(1-\int_{\rn}\G(x,t,\x,s)d\x ds\right)=0.
	\end{equation} 
	Finally, this relation, combined with \eqref{less-then-normal}, yields the desired conclusion.

	\medskip\emph{\ref{item7}.}

	Let us fix $\x\in\rn$ and $s,\t\in\R$ with $s>\t$. For every  $x\in\rn$ and $t>s$ we set
	\begin{equation}
		w(x,t):=\int_{\rn}\G(x,t,y,s)\G(y,s,\x,\t)dy,\qquad v(x,t):=\G(x,t,\x,\t).
	\end{equation}
	The reproduction property follows once we prove that $w(x,t)=v(x,t)$ for every $x \in \R^N$ and $t>s$. Since $v$ is a fundamental solution, for every $T>s$ it solves $\H v= 0$ in $\rn\times (s,T)$, moreover $v(x,t)\to \G(x,s,\x,\t)$ as $t\to s$. Let us now turn our attention to $w$. Since $\G\big(\cdot,\x,\t\big)\in C^2_\GG\big(\rnn\setminus\{\x,\t\}\big)$, it is bounded in $\rn\times (s,T)$ (we recall that $\t<s$). Hence, by \emph{\ref{item4}.} we infer that for every $T>s$ the function $w$ solves $\H w= 0$ in $\rn\times (s,T)$, and that $w(x,t)\to \G(x,s,\x,\t)$ as $t\to s$. The thesis follows then by comparison principle, once we prove that 
	\begin{equation}\label{vanishing}
		\sup_{t\in(s,T),\,|x|<R}|v(x,t)-w(x,t)|\underset{R\to \infty}\longrightarrow 0. 
	\end{equation}
	From the Gaussian estimates \eqref{Gaussian-estimates} we easily infer that
	\begin{equation}\label{aranciata}
		\begin{aligned}
			&|v(x,t)-w(x,t)|\\
			&\quad \quad\leq C_T^2\left| \frac{\exp{\left(-c_u\frac{d_X(x,\x)^2}{(t-\t)}\right)}}{\big(t-\t\big)^{\frac{Q}{2}}} + \bigintsss_{\rn}\frac{\exp{\left(-c_u\frac{d_X(x,y)^2}{(t-s)}\right)}}{\big(t-s\big)^{\frac{Q}{2}}} \frac{\exp{\left(-c_u\frac{d_X(y,\x)^2}{(s-\t)}\right)}}{\big(s-\t\big)^{\frac{Q}{2}}} dy \right|.
		\end{aligned}
	\end{equation}
	Let $\G_0$ denote the fundamental solution of $\Delta_\GG-\p_t$. From the Gaussian estimates \eqref{gaussian-parametrix} and the reproduction property of $\Gamma_0$, we infer that there exist four constants $c_1,c_2,c_3,c_4>0$ such that
	\begin{equation}\label{cioccolata-al-latte}
		\begin{aligned}
			&\bigintsss_{\rn}\frac{\exp{\left(-c_u\frac{d_X(x,y)^2}{(t-s)}\right)}}{\big(t-s\big)^{\frac{Q}{2}}}\frac{\exp{\left(-c_u\frac{d_X(y,\x)^2}{(s-\t)}\right)}}{(s-\t)^{\frac{Q}{2}}}dy\\&\leq c_2\bigintsss_{\rn}\G_0(x,c_1 t,y, c_1 s)\G_0(y, c_1 s,\x,c_1 \t)dy\leq \frac{c_4}{(t-\t)^{\frac{Q}{2}}}\exp{\left(-\frac{d_X(x,\x)^2}{c_3(t-\t)}\right)},
		\end{aligned}
	\end{equation} 	
	for every $(x,t),(\x,\t)\in\rnn$. Thus, we can bound \eqref{aranciata} by $C \exp\left(-c\,d_X(x,\x)^2\right)$,	for some $c,C>0$ that depend on $t$ and $\t$. Finally, \eqref{vanishing} follows from the continuity of the Euclidean metric w.r.t. the CC-metric, concluding the proof.

	\medskip\emph{\ref{item8}.}

	We begin by proving a \emph{local} Gaussian lower bound, which follows from the construction \eqref{G=Z+J} provided by the parametrix method. We claim that there exist $C_1,C_2>0$ and $0<\tilde{T}<1$ (depending only on structural assumptions) such that
	\begin{equation}\label{Gaussian-estimates-lower-local}
		\G(x,t,\x,\t) \geq	 \frac{1}{2 c_\l\big(t-\t\big)^{\frac{Q}{2}}}\exp{\left(- c_\l\frac{d_X(x,\x)^2}{(t-\t)}\right)},
	\end{equation}
	for every $(x,t),(\x,\t)\in \rnn$ such that $0<t-\t\leq \tilde T$ and $\frac{d_X(x,\x)^2}{t-\t} \leq C_1 \left|\ln\Big(C_2(t-\t)\Big)\right|$ (where $c_\l$ is the constant appearing in \eqref{gaussian-parametrix}).

	In order to prove claim \eqref{Gaussian-estimates-lower-local}, we first estimate the function $J$ introduced in \eqref{J}. By  \eqref{gaussian-parametrix}, \eqref{G-estimate-1}, and \eqref{cioccolata-al-latte} we infer that for every $z=(x,t),\z=(\x,\t)\in \rnn$ such that $0<t-\t\leq T$, the following relation holds 
	\begin{equation}\label{J-bound}
		\begin{aligned}
			|J(z,\z)|&\leq \int_\t^t\int_{\rn} |Z(x,t,y,s)||G(y,s,\x,\t)|dyds\\
			&\leq \bar C_T \bigintsss_\t^t (s-\tau)^{\frac{\a}{2}-1}\bigintsss_{\rn}\frac{\exp{\left(-\frac{d_X(x,y)^2}{c_\l(t-s)}\right)}}{\big(t-s\big)^{\frac{Q}{2}}}\frac{\exp{\left(-\frac{d_X(y,\x)^2}{c(s-\t)}\right)}}{(s-\t)^{\frac{Q}{2}}}dyds\\
			&\leq  \frac{C_T}{(t-\t)^{\frac{Q}{2}}}\exp{\left(-\frac{d_X(x,\x)^2}{\tilde c(t-\t)}\right)}\int_\t^t (s-\tau)^{\frac{\a}{2}-1}ds\\
			&=\frac{2 C_T}{\a} \frac{(t-\t)^\frac{\a}{2}}{(t-\t)^{\frac{Q}{2}}}\exp{\left(-\frac{d_X(x,\x)^2}{\tilde c(t-\t)}\right)},
		\end{aligned}
	\end{equation}
	for some $\bar C_T,C_T> 0$ and $\tilde c\geq c_\l$, that depend only on structural assumptions and $T$. Hence, the local lower bound \eqref{Gaussian-estimates-lower-local} follows from \eqref{G=Z+J}, \eqref{gaussian-parametrix} and \eqref{J-bound}, with $C_1$ only depending on the constant $\tilde{c}$ and $C_2$ depending on $C_T$ and $c_\lambda$.

	The global lower bound \eqref{Gaussian-estimates-lower} follows from the positivity of $\G$ given by \emph{\ref{item5}.}, combined with the reproduction property \eqref{reproduction}. Indeed, let $C_1,C_2$ and $\tilde{T}$ be the constants given by the local lower bound, and set
	\begin{equation}\label{chinotto-san-pellegrino-zero-zuccheri-aggiunti}
		A:=C_1 \left|\ln\Big(C_2 \tilde T\Big)\right|.
	\end{equation}
	We suppose without loss of generality that $\tilde{T}$ is such that
	\begin{equation}
		\sqrt{A}> \left(\frac{c_\l}{2\omega_Q}\right)^\frac{1}{Q} \Big(k_1+k_1^2\Big),
	\end{equation}	
	where $\omega_Q$ is the measure of the $d_X$-unit ball, $Q$ is the homogeneous dimension of $\GG$, $k_1$ is defined in \eqref{k1}, and $c_\l$ is the constant appearing in \eqref{Gaussian-estimates-lower-local}. Let us fix $x,\x \in \rn$, $0<t-\t<T$, and let $k$ be the smallest integer greater then
	\begin{equation}\label{def-k}
		\max\left(\frac{T}{\tilde T}, c_1 \frac{d_X(x,\x)^2}{T (t-\t)}\right),
	\end{equation}
	with $c_1>0$ to be fixed later. By the segment property and the equivalence of homogeneous distance (see for instance Corollary 5.15.6 and Proposition 5.1.4 	 in \cite{bonfiglioli2007stratified}), there exist $x=x_0	, x_1,\ldots,x_k,x_{k+1}=\x$ and $k_3>0$ (depending only on the group $\GG$) such that 
	\begin{equation}\label{segment}
		d_X(x_j,x_{j+1})\leq k_3 \frac{d_X(x,\x)}{k+1},\qquad \text{for every } j=0,\ldots,k.
	\end{equation}
	We fix
	\begin{equation}\label{VirginiaO}
		r:=\left(\frac{2c_\l}{\omega_Q}\right)^\frac{1}{Q}\sqrt{\frac{(t-\t)}{k+1}},
	\end{equation}
	and pick $t=s_0,s_1,\ldots,s_k,s_{k+1}=\t$ such that 
	\begin{equation}\label{fazio}
		s_j-s_{j+1}= \frac{t-\t}{k+1},\qquad \text{for every } j=0,\ldots,k.
	\end{equation}
	Iterating the reproduction property \eqref{reproduction} $k$ times, by the positivity of $\G$ we get
	\begin{equation}\label{grissini-conad}
		\begin{aligned}
			\G(x,t,\x,\t)&= \int_{\rn}\G(x,t,y_1,s_1)\G(y_1,s_1,\x,\t)dy_1 \\
			&= \bigintsss_{(\rn)^k}\prod_{j=0}^{k}\G(y_j,s_j,y_{j+1}s_{j+1})dy_1\ldots dy_k\\
			&\geq \bigintsss_{\prod_{j=1}^k B_r(x_j)}\prod_{j=0}^{k}\G(y_j,s_j,y_{j+1},s_{j+1})dy_1\ldots dy_k =: I,
		\end{aligned}
	\end{equation}
	where we have set in the last two integrals $y_0=x$ and $y_{j+1}=\x$, (we recall that $B_r(x_j)$ denote the $d_X$-ball of radius $r$ and center $x_j$). We set
	\begin{equation}\label{def-c1}
		c_1:=\left(\frac{k^2_1k_3\sqrt{T}}{\sqrt{A}-\left(\frac{2c_\l}{\omega_Q}\right)^\frac{1}{Q}\Big(k_1+k_1^2\Big)}\right)^2,
	\end{equation}	
	and we claim that
	\begin{equation}
		\frac{d_X(y_j,y_{j+1})^2}{s_j-s_{j+1}} \leq C_1 \left|\ln\Big(C_2(s_j-s_{j+1})\Big)\right|,
	\end{equation}
	for every $y_j\in B_r(x_j)$ and for all $j=1,\ldots,k$. This, in particular, allows us to employ the lower bound \eqref{Gaussian-estimates-lower}. By \eqref{k1},\eqref{VirginiaO}, \eqref{segment}, \eqref{def-k}, and \eqref{fazio} we get 
	\begin{equation}
		d_X(y_j,y_{j+1})\leq r(k_1+k_1^2)+k_1^2k_3 \frac{d_X(x,\x)}{k+1} \leq r(k_1+k_1^2)+\frac{k_1^2k_3}{k+1}\sqrt{\frac{T\,k}{c_1}(t-\t)}\leq \sqrt{A} \sqrt{s_j-s_{j+1}}.
	\end{equation}
	The claim is therefore a consequence of \eqref{chinotto-san-pellegrino-zero-zuccheri-aggiunti}. Hence, we are in position to apply \eqref{Gaussian-estimates-lower-local}, from which it follows that
	\begin{equation}\label{grissini-coop}
		\begin{aligned}
			I&\geq \Bigg(\frac{1}{2c_\l}\Bigg)^k \bigintsss_{\prod_{j=1}^k B_r(x_j)}\prod_{j=0}^{k}\;\frac{\exp{\left(-c_\l\frac{d_X(y_j,y_{j+1})^2}{(s_j-s_{j+1})}\right)}}{\big(s_j-s_{j+1})^{\frac{Q}{2}}}dy_1\ldots dy_k\\
			&\geq \Bigg(\frac{1}{2c_\l}\Bigg)^k \left(\prod_{j=0}^{k} \frac{\exp\big({-c_\l A}\big)}{\big(s_j-s_{j+1})^{\frac{Q}{2}}}\right)\left(\prod_{j=1}^k\omega_Q\left(\frac{2c_\l}{\omega_Q}\right) (s_j-s_{j+1})^\frac{Q}{2}\right)\\
			&\geq\frac{\exp\big({-kc_\l A}\big)}{(t-\t)^\frac{Q}{2}}.
		\end{aligned}
	\end{equation}
	If $c_1d_X(x,\x)^2\geq (t-\t)T$, then $k = c_1d_X(x,\x)^2/ (t-\t)+\y_1$ for some $\y_1\in [0,1)$. This, combined with \eqref{grissini-conad} and \eqref{grissini-coop}, implies
	\begin{equation}
		\G(x,t,\x,\t)\geq \frac{\exp\big({-c_\l A\big)}}{(t-\t)^\frac{Q}{2}}\exp{\left(-c_\l Ac_1\frac{d_X(x,\x)^2}{(t-\t)}\right)}.
	\end{equation}
	Otherwise, if $c_1d_X(x,\x)^2< (t-\t)T$, then $k =\frac{T}{\tilde T}+\y_2$ for some $\y_2\in [0,1)$. Hence, again by \eqref{grissini-conad} and \eqref{grissini-coop}, we deduce that
	\begin{equation}
		\G(x,t,\x,\t)\geq \exp\big({-c_\l A}\big)\,\frac{\exp{\left(-c_\l A\frac{T}{\tilde T}\right)}}{(t-\t)^\frac{Q}{2}}\geq \frac{\exp{\left(-c_\l A \left(\frac{T}{\tilde T}+1\right)\right)}}{(t-\t)^\frac{Q}{2}}\exp{\left(-c_\l Ac_1\frac{d_X(x,\x)^2}{(t-\t)}\right)},
	\end{equation}
	and this concludes the proof.

	\medskip\emph{\ref{item10}.}
	
	The proof of the existence and the properties of $\G^*$ is analogous to that of $\G$ and therefore we are left with proving \eqref{index}. We first observe that by \emph{\ref{item1}.-\ref{item4}.} the function
	\begin{equation}
		u(x,t):=\int_{\rn}\G^*(x,t,\xi,T_2)g(\xi)d\xi - \int_{t}^{T_2}\int_{\rn}\G^*(x,t,\xi,\tau)f(\xi,\tau)d\xi d\tau,
	\end{equation}
	is a classical solution to
	\begin{equation}
		\begin{cases}
			\HH u = f,\qquad &\mathrm{in}\;\rn\times(T,T_2),\\      
			u(x,T_2)=g(x),\qquad &\mathrm{in}\;\rn,
		\end{cases}
	\end{equation}
	as it solves the equation, and the final condition is taken in the following limit sense
	\begin{equation}\label{datofinale}
		\lim\limits_{t\rightarrow T_2}\left(\int_{\rn} \G^*(x,t,\xi,T_2)g(\xi)d\xi - \int_{t}^{T_2}\int_{\rn}\G^*(x,t,\xi,\tau)f(\xi,\tau)d\xi d\tau\right)=g(x).
	\end{equation}
	We now notice that the \emph{Green's identity} holds 
	\begin{equation}\label{green}
		v\H u-u\HH v=\div_\GG\left[\sum_{j=1}^{m_1}\Big(va_{ij}X_ju-ua_{ij}X_jv-uvX_ja_{ij}\Big)+b_i(uv)\right]-\p_t(uv),
	\end{equation}
	for every $u,v \in C_\GG^2$. Let $d_{\GG}$ denote an homogeneous distance on $\GG$ smooth out of the origin, and let $B_r(x)$ be the corresponding $d_\GG$-ball of radius $r>0$ centered at $x\in\rn$. We fix $R>0$, $(x,t),(\x,\t)\in \rnn$, $\e<\t-t$, and we integrate on $\O_T:=B_R(0)\times (\t-\e,t+\e)$  the Green's identity with $u(\cdot)=\G(\cdot,x,t)$ and $v(\cdot)=\G^*(\cdot,\xi,\t)$. Since $\H u =\HH v=0$ in $\O_T$, by applying the divergence theorem on $\GG$ (see \cite{pallara2022mean}, pag.13) we get
	\begin{equation}
		\begin{aligned}
			&\int_{B_R(0)}\G(y,t+\e,x,t)\G^*(y,t+\e,\x,\t)dy-\int_{B_R(0 )}\G(y,\t-\e,x,t)\G^*(y,\t-\e,\x,\t)dy\\	
			&\qquad\qquad=\int_{\t-\e}^{t+\e}\int_{\p B_R(0)}\sum_{j=1}^{m_1}\Big\langle\big(\G^*a_{ij}X_j\G-\G a_{ij}X_j\G^*-\G\G^*X_ja_{ij}\big)+b_i\G\G^*,\tilde{\n} \Big\rangle d \mathcal{H}^{n-1},
		\end{aligned}
	\end{equation}
	where $\tilde{\n}_i=\sum_{j=1}^N \n_j\phi_j^i$(here, $\n$ denotes the Euclidean normal outward vector of $\p B_R(0)$). Since the components $\phi_j^i$ of the vector fields $X_i$ are polynomial, by \eqref{Gaussian-estimates}, and \eqref{Gaussian-estimates-1}  and the equivalence of $d_\GG$ and $d_X$, the integral in the right hand-side vanishes as $R\rightarrow +\infty$. Hence, we conclude that
	\begin{equation}\label{equality}
		\int_{\rn}\G(y,t+\e,x,t)\G^*(y,t+\e,\x,\t)dy=\int_{\rn}\G(y,\t-\e,x,t)\G^*(y,\t-\e,\x,\t)dy.
	\end{equation}
	The proof follows by letting $\e \rightarrow 0$. Indeed, if we set $g_\e(y):=\G^*(y,t+\e,\x,\t)$, then we rewrite the left-hand side of the previous equality as follows
	\begin{equation}\label{dato-iniz-fin}
		\begin{aligned}
			&\int_{\rn}\G(y,t+\e,x,t) g_\e(y)dy\\
			&\quad=\int_{\rn}\G(y,t+\e,x,t) \G^*(y,t,\x,\t)dy + \int_{\rn}\G(y,t+\e,x,t) \Big(g_\e(y)-\G^*(y,t,\x,\t)\Big)dy\\
			&\quad=:I_1+I_2. 
		\end{aligned}
	\end{equation}
	By the regularity of $\G$ outside the pole and \eqref{Gaussian-estimates}, we infer that $I_2$ vanishes as $\e\to 0$. Moreover, by \eqref{datoiniziale} we deduce that $I_1$ converges to $\G^*(x,t,\x,\t)$ as $\e\to 0$. The same arguments apply to the right-hand side of \eqref{equality} (by using \eqref{datofinale} instead of \eqref{datoiniziale}). Hence, the proof is complete.

	\medskip\emph{\ref{item9}.}

	We begin with estimate \eqref{egrowth}. Let $k_1$ as in \eqref{k1}: for every $H_1,H_2>0$ such that 
	\begin{equation}\label{h1h2k1k2}
		\sqrt{\frac{H_1+H_2}{H_1}}>k_1,
	\end{equation}
	we claim that there exists $M>0$ such that
	\begin{equation}\label{inequality-distance}
		H_1\,d_X(x,y)^2-(H_1+H_2)\,d_X(y,0)^2\leq M\,d_X(x,0)^2,\qquad \text{for every}\; x,y\in \rn. 
	\end{equation}
	The claim is clearly true for $y=0$, taking $M\geq H_1$. Then, if $y\neq 0$, we set $r:=d_X(y,0)$ and
	\begin{equation}
		\xi:=\d_{\frac{1}{r}}(x),\qquad w:=\d_{\frac{1}{r}}(y).
	\end{equation}
	Since $d_X$ is $\d_r$-homogeneous of degree one, \eqref{inequality-distance} becomes
	\begin{equation}\label{inequaliy-distance-2}
		H_1\,d_X(\x,w)^2-(H_1+H_2)\leq M\,d_X(\x,0)^2.
	\end{equation}
	If $d_X(\x,w)^2\leq \frac{(H_1+H_2)}{H_1}$, \eqref{inequaliy-distance-2} holds for every $M> 0$. Otherwise, by \eqref{k1} we get
	\begin{equation}
		d_X(\x,0)> \frac{1}{k_1}\sqrt{\frac{H_1+H_2}{H_1}}-1:=\vartheta.
	\end{equation}
	We remark that $\vartheta$ is strictly positive, thanks to \eqref{h1h2k1k2}. Then, by \eqref{k1} inequality \eqref{inequality-distance} follows if
	\begin{equation}
		d_X(\x,0)^2\big(H_1k_1^2-M\big) + 2H_1k_1^2d_X(\x,0) + \big(H_1k_1^2-H_1-H_2\big)<0.
	\end{equation}
	This inequality holds true if we pick $M> 2H_1k_1^2/\vartheta$.

	Let us now take  $H_1=h_1$ and $H_2=\frac{c_u}{4T}-\frac{H_1}{2}$ (where $c_u$ is given by the Gaussian upper bound \eqref{Gaussian-estimates}), and recall that $h_1$ satisfies \eqref{h1}. By \eqref{Gaussian-estimates} and \eqref{h1} we get 
	\begin{equation}\label{gamma-aumentata}
		\begin{aligned}
			|\G(x,t,\x,\t)|&\leq \frac{c_T (t-\tau)^{\frac{\a}{2}}}{(t-\t)^{\frac{Q}{2}}}\exp{\left(-c_u\frac{d_X(x,\x)^2}{(t-\t)}\right)}\\
			&\leq c_T\frac{\exp{\left(-c\frac{d_X(x,\x)^2}{(t-\t)}\right)}}{(t-\t)^{\frac{Q}{2}}}\exp{\left(-(H_1+H_2)d_X(x,\x)^2\right)},
		\end{aligned}
	\end{equation} 
	with $c = \big(c_u-T(H_1+H_2)\big)/2$. From the representation formula \eqref{u-conv}, together with \eqref{f-growth-exp}, \eqref{g-growth-exp}, and  \eqref{gamma-aumentata} we get
	\begin{equation}
		\begin{aligned}
			|u(x,t)|&\leq\int_{\rn}|\G(x,t,\xi,T_1)g(\xi)|d\xi + \int_{T_1}^t\int_{\rn}|\G((x,t,\xi,\tau)f(\xi,\tau)|d\xi d\tau\\
			&\leq h_2c_T \bigintsss_{\rn}\frac{\exp{\left(-c\frac{d_X(x,\x)^2}{(t-T_1)}\right)}}{(t-T_1)^{\frac{Q}{2}}}\exp{\left(-(H_1+H_2)d_X(x,\x)^2+H_1d_X(\x,0)^2\right)}d\xi \\
			&+h_2c_T\bigintsss_{T_1}^t\bigintsss_{\rn}\frac{\exp{\left(-c\frac{d_X(x,\x)^2}{(t-\t)}\right)}}{(t-\t)^{\frac{Q}{2}}}\exp{\left(-(H_1+H_2)d_X(x,\x)^2+H_1d_X(\x,0)^2\right)}d\xi d\t:=I_1.
		\end{aligned}
	\end{equation}
	By setting $\x^{-1}\circ x = y$, from \eqref{inequality-distance} we get
	\begin{equation}
		\begin{aligned}
		I_1&=h_2c_T\bigintsss_{\rn}\frac{\exp{\left(-c\frac{d_X(y,0)^2}{(t-T_1)}\right)}}{(t-T_1)^{\frac{Q}{2}}}\exp{\left(-(H_1+H_2)d_X(y,0)^2+H_1d_X(x,y)^2\right)}dy \\
		&+h_2c_T\bigintsss_{T_1}^t\bigintsss_{\rn}\frac{\exp{\left(-c\frac{d_X(y,0)^2}{(t-\t)}\right)}}{(t-\t)^{\frac{Q}{2}}}\exp{\left(-(H_1+H_2)d_X(y,0)^2+H_1d_X(x,y)^2\right)}dy d\t\\
		&\leq h_2c_T\exp\big({Md_X(x,0)^2}\big)\left(\bigintsss_{\rn}\frac{\exp{\left(-c\frac{d_X(y,0)^2}{(t-T_1)}\right)}}{(t-T_1)^{\frac{Q}{2}}}dy +\bigintsss_{T_1}^t\bigintsss_{\rn}\frac{\exp{\left(-c\frac{d_X(y,0)^2}{(t-\t)}\right)}}{(t-\t)^{\frac{Q}{2}}}dy d\t\right):=I_2.
		\end{aligned}
	\end{equation}
	Finally, employing once again the fundamental solution $\G_0$ of $\Delta_\GG-\p_t$, together with its normalization and the Gaussian upper bound \eqref{gaussian-parametrix}, we conclude
	\begin{equation}
		\begin{aligned}
			I_2&\leq h_2C_T\exp\left({Md_X(x,0)^2}\right)\Bigg(\int_{\rn}\G_0(y,c_1t,0,c_1T_1)dy \\
			&+\int_{T_1}^t\int_{\rn}\G_0(y,c_1t,0,c_1\t)dy d\t\Bigg) \leq C\exp\left({Md_X(x,0)^2}\right),
		\end{aligned}
	\end{equation}
	for some $C_T,c_1>0$ depending only on structural assumptions.

	We conclude by addressing the second part of \emph{\ref{item9}.} Let us suppose that $u,v$ are classical solutions to \eqref{cauchy} satisfying \eqref{integralestgr}. The thesis follows once we show that there exists $\bar\s\in\R$ such that, for every $T_1<\s<\bar \s$, it holds $u(x,t)=v(x,t)$ in $\rn\times(T_1,\s)$.

	We now consider the function $w:=u-v$, which, for every $T_1<\s<T_2$, is a solution to the equation
	\begin{equation}
		\begin{cases}
			\H w =0, \qquad &\mathrm{in}\;\rn\times(T_1,\s),\\      
			w(x,T_1)=0,\qquad &\mathrm{in}\;\rn,
		\end{cases}
	\end{equation}
	and by \eqref{integralestgr} satisfies
	\begin{equation}\label{intestgr}
		\int_{T_1}^{\s}\int_{\rn} |w(x,t)|\exp\left({-kd_X(x,0)^2}\right)dxdt<+\infty,
	\end{equation}
	for some $k>0$. We fix $(x,t)\in \rn\times (T_1,\s)$,  we let again $d_{\GG}$ denote an homogeneous distance smooth out of the origin, and we  denote as $B_R$ the  $d_{\GG}$-ball of radius $R\geq 1$ centered at $x$. Let $0\leq h \leq 1$  be a smooth cut-off function such that $h(\x)=0$ if $\xi \notin B_{R+1}$, and  $h(\x)=1$ if $\xi \in B_{R}$. If we pick $\e>0$ and we integrate over $B_{R+1}\times (T_1,t-\e)$ the Green's identity \eqref{green} with $u:=w$ and $v:=h(\cdot)\G(x,t,\cdot)$, we get
	\begin{equation}
		\begin{aligned}
			&\int_{T_1}^{t-\e}\int_{B_{R+1}}w(\x,\t)\HH\big(h(\x)\G(x,t,\x,\t)\big)d\x d\t=\int_{T_1}^{t-\e}\int_{B_{R+1}}\p_\t\Big(h(\x)\G(x,t,\x,\t) w(\x,\t)\Big)d\x d\t\\
			&\qquad\quad\qquad-\bigintsss_{T_1}^{t-\e}\bigintsss_{B_{R+1}}\div_\GG\left[\sum_{j=1}^m\big(h\G a_{ij}X_jw-w a_{ij}X_j(h\G)-wh\G X_ja_{ij}\big)+b_iwh\G\right]d\x d\t.
		\end{aligned}
	\end{equation} 
	Since $h=0$ on  $\p B_{R+1}$, the divergence theorem on $\GG$ implies that the last integral vanishes. Then $w(x,0)=0$ yields
	\begin{equation}
		\int_{T_1}^{t-\e}\int_{B_{R+1}}w(\x,\t)\HH\big(h(\x)\G(x,t,\x,\t)\big)d\x d\t=\int_{\rn}h(\x)\G(x,t,\x,t-\e) w(\x,t-\e) d\x. 
	\end{equation}
	As shown in \eqref{dato-iniz-fin}, the integral on the right hand-side converges to $h(x)w(x,t)=w(x,t)$ as $\e\rightarrow 0$. Since $h=0$ in $B_R$ and $\HH\G(x,t,\cdot)=0$ in $B_{R+1}\setminus B_R$ (by \eqref{index}), we get 
	\begin{equation}
		w(x,t)=\bigintsss_{T_1}^{t}\bigintsss_{B_{R+1}\setminus B_R}w\left[2\sum_{i,j=1}^{m_1}a_{ij}X_ihX_j\G+\G\sum_{i,j=1}^{m_1}a_{ij}X_iX_jh+\G\sum_{i=1}^{m_1}X_ih\right]d\xi d\t.
	\end{equation}
	Given that $h$ and its derivatives are bounded, and exploiting once again the fact that the vector fields have polynomial components, by \eqref{Gaussian-estimates} and \eqref{Gaussian-estimates-1} and the equivalence of $d_\GG$ and $d_X$, we infer that there exits $\bar \sigma \in \R$ such that, if $T_1<\sigma<\bar \sigma$, the following relation holds for some $C,\bar C_T>0$
	\begin{equation}
		\begin{aligned}
			|w(x,t)|&\leq C_T\bigintsss_{T_1}^t\bigintsss_{B_{R+1}\setminus B_R}|w(\x,\t)|\frac{\exp{\left(-C\frac{d_X(x,\x)^2}{(t-\t)}\right)}}{(t-\tau)^{\frac{Q+1}{2}}} d\x d\t\\
			&\leq \bar C_T\bigintsss_{T_1}^\s\bigintsss_{B_{R+1}\setminus B_R}|w(\x,\t)|\exp{\left(-kk_1^2d_X(x,\x)^2\right)}d\x d\t\\
			&\leq \bar C_T\,\exp{\left((1-kk_1)d_X(x,0)^2\right)}\bigintsss_{T_1}^\s\bigintsss_{B_{R+1}\setminus B_R}|w(\x,\t)|\exp\left({-kd_X(\x,0)^2}\right)d\x d\t.
		\end{aligned}
	\end{equation}
	By \eqref{intestgr}, we infer that the right hand-side vanishes as $R\rightarrow +\infty$, implying $w(x,t)=0$ for every $(x,t)\in \rn\times (T_1,\s)$. Repeating this argument finitely many times yields the desired result.
\end{pf-fundamental}

\section{Mean value formulas}\label{Mean-Value-Formula}
In this section we prove a representation formula that will be exploited in Sections \ref{Maximum} and \ref{Parabolic Harnack inequality} to prove the maximum principle and a parabolic Harnack inequality. This formula is based on the solid mean value formula recently developed in \cite{pallara2023meanvalueformulasclassical}, that we here improve employing Hadamard's descent method, in order to overcome the unboundedness of the kernel.

The solid mean value formula from Theorem 2.2 in \cite{pallara2023meanvalueformulasclassical} holds true once the following conditions are satisfied:

\medskip
\noindent
{{\bf(M1)}} \ {there exists an homogeneous Lie group $\widehat{\GG} = \left(\rnn,\widehat\circ,\widehat\delta_r\right)$ such that $X_1,\ldots,X_{m_1}$ and $X_0 = \p_t$ are left-invariant on $\widehat\GG$ and $\widehat{\d_r}$-homogeneous of degree one;}
\\\\{{\bf(M2)}} \ {there exists a fundamental solution $\G^*$ to the adjoint operator \eqref{HG1star} such that
\begin{itemize}
	\item there exists $r_0>0$ such that for all $\z=(\x,\t) \in \R^{N+1}$ the set 
	\begin{equation}\label{set-mvf}
		\Omega_r(\z):= \biggl\{ z=(x,t) \in \R^N \times (-\infty,\t): \Gamma^*(z,\z)>\frac{1}{r}\biggr\},
	\end{equation}
	is bounded for every $r \leq r_0$;
	\item for all $\z=(\x,\t)\in\rnn$ and $r,\e>0$, the set
	\begin{equation}
		\mathcal{I}_{r,\e}(\z):= \biggl\{ x \in \R^N : \Gamma^*(x,\t-\e,\z)>\frac{1}{r}\biggr\},
	\end{equation}
	satisfies the \emph{pointwise vanishing integral condition}
	\begin{equation}\label{pointwise-vanishing}
		\lim_{\varepsilon \to 0} {\mathscr H}^{N} \big( \mathcal{I}_{r, \varepsilon} (\z) \big) = 0, \qquad
		\lim_{\varepsilon \to 0} \int_{\R^N \backslash \mathcal{I}_{r, \varepsilon} (\z)} \Gamma^*(x,\t - \varepsilon,\z) d \x = 0.
	\end{equation}
\end{itemize}
Before proving our improved mean value formula, we briefly discuss the previous conditions.

We first focus our attention on assumption \textbf{(M1)}. In regularity theory, $\widehat\GG$ is not the usual group where one investigates the properties of $\H$. This is because operator $\H$ is not homogeneous of degree $2$ w.r.t. $\widehat\delta_r$, since the vector field $X_0=\p_t$ is $\widehat{\d_r}$-homogeneous of degree one (namely, $X_0$ belong to the first layer of the Lie algebra of $\widehat{\GG}$). The natural geometry when
studying operator $\H$ is determined by the parabolic  group relevant to $\GG$. More precisely, the natural framework for the regularity theory of operator $\H$ is the group $\GG_p=\Big(\rnn,\bar\circ,\bar \delta_r\Big)$, where
\begin{equation}
		z\,\bar\circ\,\z :=\left(x\circ \xi,t+\t\right),\qquad \bar\d_{r}(x,t)=\Big(\d_r(x),r^2t\Big),\qquad \forall z,\z \in \rnn, r>0,
\end{equation}
with $\circ$ and $\{\delta_r\}_{r>0}$ given by $\GG$. It is straightforward to verify that on  $\GG_p$ the operator $\H$ is left-invariant and homogeneous of degree $2$. However, the divergence theorem on $\GG_p$ would not be enough to prove the mean value formula, since in this setting the surfaces $\{t = \text{constant}\}$ would have codimension $2$, hence being negligible. For further comments on this topic, we refer the reader to the introduction of \cite{pallara2023meanvalueformulasclassical}, or to the  monograph \cite{cassano2016some} about Geometric Measure Theory on Carnot groups.

We now comment on \textbf{(M2)}. The pointwise vanishing integral condition \eqref{pointwise-vanishing}, as emphasized in \cite{CupiniLanconelli2021}, is a fundamental requirement for the proof of the mean value formulas, and is closely related to the local properties of $\G^*$. In the stationary case, this property is straightforward to establish. In the evolutive setting, however, its proof is more delicate and relies on precise estimates, that in our case will be the Gaussian lower and upper bounds, together with the long time bound given by \emph{\ref{item5}.} in Theorem \ref{Th-Fundamental}.

We conclude with a final remark. In Theorem \ref{Th-Fundamental} we established the relation \eqref{index} between $\G$ and $\G^*$, if \textbf{(H1)}, \textbf{(H2)}, \textbf{(H3)} and \textbf{(H4)} hold. As a result, we can express the mean value formula in terms of the fundamental solution $\G$, instead of $\G^*$. In particular, this means that 
\begin{equation}
	\Omega_r(\z)= \biggl\{ z \in \R^N \times (-\infty,\t): \Gamma(\z,z)>\frac{1}{r}\biggr\},\qquad \mathcal{I}_{r,\e}(\z)= \biggl\{ x \in \R^N : \Gamma(\z,x,\t-\e)>\frac{1}{r}\biggr\}.
\end{equation}
Finally, before presenting the representation formula, we introduce the relevant kernel. For every $z,\z\in\rnn$ we set 
\begin{equation}\label{Unbounded-kernel}
	M_\GG(\z,z):=\frac{\langle A(z)\nabla_{\GG} \G(\z,z),\nabla_{\GG}\G(\z,z) \rangle}{{\G(\z,z)}^2}.
\end{equation}
\begin{proposition}[Mean value formula]\label{Pr-meanvalue}
	Let $\H$ be a differential operator as in \eqref{HG1} satisfying assumptions \textbf{(H1)}, \textbf{(H2)}, \textbf{(H3)} and \textbf{(H4)}. Let $\Omega$ be an open subset of $\R^{N+1}$, $f \in C(\Omega)$ and let $u$ be a classical solution to $\H u =f$ in $\Omega$. Then, for every $\z \in \Omega$ and  for every $r>0$ such that ${\Omega_r(\z)} \subset \Omega$, we have 
	\begin{equation}
		\begin{aligned} 
			u(\z) =& \frac{1}{r} \bigintsss_{\Omega_r(\z)} M_\GG (\z, z) u(z) \, dz +\frac{1}{r}  \bigintsss_0^{r} \Biggl(
			\bigintsss_{\Omega_\rho (\z)} f (z) \biggl( \tfrac{1}{\rho} - \Gamma(\z,z) \biggr) dz \Biggr)\, d \rho
			\\+&\frac{1}{r} \bigintsss_0^{r}  \frac{1}{\rho} \Biggl(
			\bigintsss_{\Omega_\rho (\z)} \big( \div_\GG b(z) - c(z) \big) u(z) \, dz \Biggr) \, d \rho. 
		\end{aligned}
	\end{equation} 
\end{proposition}
\begin{proof}
	This representation formula corresponds to the solid mean value formula stated in Theorem 2.2 in \cite{pallara2023meanvalueformulasclassical}. We need to prove that we are in position to apply the aforementioned theorem, namely that \textbf{(M1)} and \textbf{(M2)} hold under assumptions \textbf{(H1)}, \textbf{(H2)}, \textbf{(H3)} and \textbf{(H4)}.
		
	\medskip
	The existence of the Carnot group $\GG = \big(\rn,\circ,\d_r\big)$ introduced in Section \ref{Carnot}, directly implies \textbf{(M1)}. Indeed, we define on $\rnn$ the following composition and dilations law
	\begin{equation}
		z\,\widehat\circ\,\z :=\left(x\circ \xi,t+\t\right),\qquad \widehat\d_{r}(x,t)=\Big(\d_r(x),rt\Big),\qquad \forall z,\z \in \rnn,r>0.
	\end{equation}
	Since $\GG$ is a Carnot group on $\rn$, as an obvious consequence it follows that  $\widehat{\GG} = \left(\rnn,\widehat\circ,\widehat\delta_r\right)$ is  an homogeneous Lie group on $\rnn$. Moreover, the vector fields $X_1,\ldots,X_{m_1}$ and $X_0=\p_t$ are left-invariant  w.r.t. $\widehat \circ$  and homogeneous of degree 1 w.r.t. $\big\{\widehat \delta_r\big\}_{r>0}$, namely
	\begin{equation}
		X_i \Big(f\big(z\,\widehat\circ\,\z\big) \Big)=\big( X_i f\big)\big(z\,\widehat\circ\,\z\big), \qquad X_i \left(f \left(\widehat\d_{r}(z)\right) \right)=r \big(X_i f \big)\left(\widehat\d_{r}(z)\right), \qquad i=0,\ldots, m_1,
	\end{equation}
	for every $f\in C^\infty\big(\R^{N+1}\big)$, thus implying \textbf{(M1)}.

	\medskip
	Let us turn our attention to condition \textbf{(M2)}. By \emph{\ref{item5}.} and \emph{\ref{item8}.} in Theorem \ref{Th-Fundamental}, we infer that there exists $\bar r>0$ such that 
	\begin{equation}\label{portopalo-di-capo-passero}
		\Omega_r(\z)= \biggl\{ z \in \R^N \times (\t-1,\t): \Gamma(\z,z)>\frac{1}{r}\biggr\},
	\end{equation}
	for every $r\leq \bar r$. Let us now consider the Gaussian upper bound \eqref{Gaussian-estimates} with $T=1$. If we set $r_0\leq\min\left(\bar r,c_1^{-1}\right)$, it follows 
	\begin{equation}
		\begin{aligned}
			\Omega_r(\z)&\subset \left\{ z \in \R^N \times (\t-1,\t): d_X(\x,x)^2\leq c_u(\t-t) \ln\left(\frac{c_1\,r }{(\t-t)^\frac{Q}{2}}\right) \right\},\\
			\mathcal{I}_{r,\e}(\z)&\subset \left\{ x \in \R^N : d_X(\x,x)^2\leq c_u\e\ln\Bigg(\frac{c_1\,r }{ \e^\frac{Q}{2}}\Bigg) \right\},
		\end{aligned}
	\end{equation}
	provided $\e<1$. The local boundedness of $\O_r(\z)$ and the first point of \eqref{pointwise-vanishing} follows. We conclude by addressing the second point of \eqref{pointwise-vanishing}. With this aim, we employ the Gaussian lower bound \eqref{Gaussian-estimates-lower}, from which it follows that
	\begin{equation}
		\left\{ x \in \R^N : d_X(\x,x)^2\leq c_l\e \ln\Bigg(\frac{\tilde c_\e r}{\e^\frac{Q}{2}}\Bigg) \right\}\subset \mathcal{I}_{r,\e}(\z),
	\end{equation}
	for every $\e>0$. Hence, since $\G$ is positive, by the Gaussian upper bound \eqref{Gaussian-estimates} we infer that for every $\e<1$ there exists $c_1,c_u>0$ such that
	\begin{equation}
		\begin{aligned}
			0&\leq \int_{\R^N \backslash \mathcal{I}_{r, \varepsilon} (\z)} \Gamma(\z,x,\t - \varepsilon) d x \leq \int_{\left\{d_X(\x,x)^2\geq c_l\e \ln\left(\frac{\tilde c_\e r}{ \e^\frac{Q}{2}}\right) \right\}}  \Gamma(\z,x,\t - \varepsilon) d x\\
			&\leq\frac{c_1}{\e^\frac{Q}{2}} \bigintsss_{\left\{d_X(\x,x)^2\geq c_l\e \ln\left(\frac{\tilde c_\e r}{\e^\frac{Q}{2}}\right) \right\}} \exp{\left(-c_u\frac{d_X(\x,x)^2}{\e}\right)} dx \\
			&\Big(\text{by setting }y= \d_{\frac{1}{\sqrt{\e}}}\big(x^{-1}\circ \x\big)\Big)\\
			&=c_1 \bigintsss_{\left\{d_X(y,0)^2\geq c_l \ln\left(\frac{\tilde c_\e r}{\e^\frac{Q}{2}}\right) \right\}} \exp{\Big(-c_u d_X(y,0)^2\Big)} dy,
		\end{aligned}
	\end{equation}
	and the proof is complete, since the last integral vanishes as $\e\to 0$.
\end{proof}
\begin{remark}
	To be more precise, the kernel appearing in Theorem 2.2 in \cite{pallara2023meanvalueformulasclassical} is actually of the following form
	\begin{equation}\label{A-kernel}
		M_\GG(\z,z):=\frac{\langle \bar A(z)\nabla_{\widehat \GG} \G(\z,z),\nabla_{\widehat \GG}\G(\z,z) \rangle}{{\G(\z,z)}^2},
	\end{equation}
	where $\bar A$ denotes the matrix with entries $a_{ij}$ for $i,j=1,\ldots,m_1$, and $a_{(m_1+1)j}=a_{j(m_1+1)}=0$, while $\nabla_{\widehat \GG}$ denotes the horizontal gradient on $\widehat \GG$, namely
	\begin{equation}
		 \nabla_{\GG}f:=\sum_{i=1}^{m_1} \big(X_i f\big)X_i + (\p_t f)(\p_t).
	\end{equation}	
	However, the previous kernel is formally identical to the one introduced in \eqref{Unbounded-kernel}.	
\end{remark}
As mentioned before, a questionable feature displayed in the previous proposition is the unboundedness of the kernels. This can be easily checked in the case of the classical heat operator, where $M_\GG(\z,z)=\frac{|\x-x|^2}{4(\t-t)^2}$, the so-called Pini-Watson kernel. Moreover, another limitation is that the kernel cannot be bounded from below by a strictly positive constant. These two issues prevent us from taking advantage of the classical elementary argument developed for harmonic functions to establish both the maximum propagation and the Harnack inequality.  Hence, starting from Proposition \ref{Pr-meanvalue}, we aim to prove a representation formula with a kernel that is not only bounded but, also possesses a degree of regularity arbitrarily high. This boundedness will be crucial in particular to prove the parabolic Harnack inequality given in Section \ref{Parabolic Harnack inequality}. Our technique to prove the mean value formula with a more regular kernel is based on Hadamard's descent method, already used by Kuptsov in \cite{Kupcov}, and later developed in \cite{CupiniLanconelli2021,GarofaloLanconelli-1989, GarofaloLanconelli, malagoli, rebucci24}. In order to outline the procedure we follow, we first introduce some further notation. For every $m \in \N$, we introduce the operator
\vspace{-0.1cm}
\begin{equation}\label{Eq:oper-ltilde}
	\tilde{\H}^{(m)}:= \Delta_y^{(m)}+ \H ,
	\vspace{-0.1cm}
\end{equation}
where $\Delta_y^{(m)}$ is the classical Laplace operator in $\R^m$ in the variables $y=(y_1,\ldots,y_m)$. Moreover, we define the function
\begin{equation}\label{Eq:Gammatilde}
	\begin{aligned}
		\tilde{\Gamma}^{(m)}(\y,\x,\t,y,x,t)&=H^{(m)}(\y,\t,y,t)\,\Gamma(\x,\t,x,t)\\
		&=\frac{1}{\big(4\pi(\t-t)\big)^{m/2}}\exp \left(-\frac{|\y-y|^2}{4(\t-t)} \right)\,\Gamma(\x,\t,x,t),
	\end{aligned}
	\vspace{-0.1cm}
\end{equation}
where $\Gamma$ and $H^{(m)}$ denote the fundamental solutions of $\H$ and of the heat equation in $\R^m$, respectively.
	
By standard arguments it follows that $\tilde{\Gamma}^{(m)}$ is a fundamental solution to $\tilde{\H}^{(m)}$ and that, if $u=u(x,t)$ is a solution to $\H u =f$, then the function
\begin{equation}\label{Eq:defutilde}
	\tilde{u}^{(m)}(y,x,t):=u(x,t), \quad y \in \R^m,
\end{equation}
is a solution to $\tilde{\H}^{(m)}\tilde{u}^{(m)}=\tilde{f}^{(m)}$, where
\begin{equation*}
	\tilde{f}^{(m)}(y,x,t)=f(x,t), \quad y \in \R^m.
\end{equation*}
By integrating the mean value formula of Proposition \ref{Pr-meanvalue} with respect to the variable $y$, we then obtain new kernels which are bounded for any $m>2$. {\color{black}In order to state and prove our modified mean value formula, we first introduce some further notation:
\begin{equation}\label{Eq:kernels-m}
	\begin{aligned}
		\Gamma^{(m)}(\z,z)&:=\frac{\Gamma(\z,z)}{\big( 4\pi(\t-t)\big)^{m/2}},\\ 
		\Omega_r^{(m)}(\z)&:=\biggl\{ z \in \R^{N+1}:\Gamma^{(m)}(\z,z)  >\frac{1}{r} \biggr\},\\
		N_r(\z,z)&:=2\sqrt{\t-t}\sqrt{\ln\left( r\, \Gamma^{(m)}(\z,z)\right)},\\ 
		M_r^{(m)}(\z,z)&:=\omega_m \,N_r^m(\z,z) \left( M_\GG(\z,z)+\frac{m}{m+2}\,\frac{N_r^2(\z,z)}{4(\t-t)^2}\right),\\
		W_r^{(m)}(\z,z)&:=m\,\omega_m \left(\frac{N_r^m(\z,z)}{rm} -\Gamma^{(m)}(\z,z)\,\frac{4(\t-t)^{m/2}}{2}\, \gamma\left( \frac{m}{2};\frac{N_r^2(\z,z)}{4(\t-t)}\right)\right),
	\end{aligned}
\end{equation}
where $M_\GG(\z,z)$ is the kernel introduced in \eqref{Unbounded-kernel}, $\omega_m$ denotes the Lebesgue measure of the unit ball in $\R^m$  w.r.t. the Euclidean metric, and $\gamma$ is the lower incomplete gamma function, \emph{i.e.}
\begin{equation}\label{Eq:gammafunction}
	\gamma(s;x):=\int_0^x t^{s-1}\mathrm{e}^{-t}dt.
\end{equation}
We are now in a position to state and prove the following proposition. 
\begin{proposition}[Improved mean value formula]\label{prop-meanvalue}
	Let $\H$ be a differential operator of the form \eqref{HG1} satisfying assumptions \textbf{(H1)}, \textbf{(H2)}, \textbf{(H3)} and \textbf{(H4)}. Let $\Omega$ be an open subset of $\R^{N+1}$, $f \in C(\Omega)$ and let $u$ be a classical solution to $\H u =f$ in $\Omega$. Then, for every $\z \in \Omega$ and for every $r >0$ such that $\Omega^{(m)}_r(\z) \subset \Omega$, we have 
	\begin{equation} \label{e-meanvalue} 
		\begin{aligned}
			u(\z) =& \frac{1}{r} \bigintsss_{\Omega_r^{(m)}(\z)} M_r^{(m)} (\z, z) u(z) \, dz 
			+ \frac{1}{r}  \bigintsss_0^{r} \Biggl(
			\bigintsss_{\Omega^{(m)}_\r (\z)} f (z) \, W_\r^{(m)}(\z,z) \, dz \Biggr)\, d \r \\
			+&\frac{1}{r} \bigintsss_0^{r} \frac{\omega_m}{\r} \Biggl(
			\bigintsss_{\Omega^{(m)}_\r (\z)}N_\r^m(\z,z)\, \big( \div_\GG b(z) - c(z) \big) \,u(z) \, dz \Biggr) \, d \r. 
		\end{aligned}
	\end{equation} 
\end{proposition}
\begin{proof}
	As a first step, we observe that, for every $m \in \N$, the operator $\tilde{\H}^{(m)}$ in \eqref{Eq:oper-ltilde} belongs to the class of operators \eqref{HG1} with the matrix $A$ replaced by the $(m+N) \times (m+N)$ real matrix
	\begin{equation*}
		\tilde{A}(y,z)=\begin{pmatrix}\mathbb{I}_m  &\mathbb{O}\\\mathbb{O} & A(z)\end{pmatrix},\qquad y\in\R^{m},z\in\rnn,
	\end{equation*}
	where $A$ denotes the matrix related to operator $\H$ in \eqref{HG1}. Hence, operator $\tilde{\H}^{(m)}$ clearly satisfies the assumptions of Proposition \ref{Pr-meanvalue}, and we infer 
	\begin{equation}\label{Eq:meanMtilde}
		\begin{aligned}
			\tilde{u}^{(m)}(\y,\z)= &\frac{1}{r} \bigintsss_{\tilde{\Omega}_r^{(m)}(\y,\z)} \tilde{M}_\GG^{(m)} (\y,\z, y,z) \tilde{u}^{(m)}(y,z)\, dy\, dz \\
			+& \frac{1}{r}  \bigintsss_0^{r} \Biggl(
			\bigintsss_{\tilde{\Omega}^{(m)}_\r (\y,\z)} \tilde{f}^{(m)}(y,z) \,\left(\frac{1}{\r}-\tilde{\Gamma}^{(m)}(\y,\z,y,z) \right)dy  \, dz \Biggr)\, d \r 
			\\
			+&\frac{1}{r} \bigintsss_0^{r}  \frac{1}{\r} \Biggl(
			\bigintsss_{\tilde{\Omega}^{(m)}_\r (\y,\z)} \big( \div_\GG b(z) - c(z) \big) \,\tilde{u}^{(m)}(y,z)\,dy \, dz \Biggr) \, d \r,
		\end{aligned}	
	\end{equation}
	where $\tilde{\Omega}^{(m)}_r(\y,\z)$ is the super-level set relevant to $\tilde{\Gamma}^{(m)}$ in \eqref{Eq:Gammatilde}, \emph{i.e.}
	\begin{equation}\label{Eq:omegatildem}
		\tilde{\Omega}^{(m)}_r(\y,\z):=\biggl\{ (y,z) \in \R^{m+N} \times (-\infty,\t): \tilde{\Gamma}^{(m)}(\y,\z,y,z)>\frac{1}{r}\biggr\}.
	\end{equation}
	We first point out that, in virtue of \eqref{divG}, the kernel $\tilde{M}^{(m)}(\y,\z,y,z)$ in \eqref{Eq:meanMtilde} corresponding to operator $\tilde{\H}^{(m)}$ is equal to 
	\begin{equation}
		\begin{aligned}
			\tilde{M}^{(m)}(\y,\z,y,z)=&\\
			&+ \frac{H^{(m)}(\y,\t,y,t)^2\displaystyle\sum\nolimits_{i,j=1}^{m_1}a_{ij}(z)X_i\G(\z,z)X_j\G(\z,z)}{\tilde{\Gamma}^{(m)}(\y,\z,y,z)^2}\\
			&\qquad\qquad\qquad\qquad\qquad\qquad+\frac{\G(z,\z)^2\displaystyle\sum\nolimits_{j=1}^{m}\p_{y_j}^2H^{(m)}(\y,\t,y,t)}{\tilde{\Gamma}^{(m)}(\y,\z,y,z)^2}.
		\end{aligned}
	\end{equation}
	Then, owing to \eqref{divG} and \eqref{Eq:Gammatilde}, the kernel $\tilde{M}^{(m)}(\y,\z,y,z)$ becomes
	\begin{equation}\label{Eq:Mtildesum}
		\tilde{M}^{(m)}(\y,\z,y,z)=M_\GG(\z,z)+\frac{|\y-y|^2}{4(\t-t)^2},
	\end{equation}
	where $M_\GG(\z,z)$ is the kernel defined in \eqref{Unbounded-kernel}, and $\frac{|\y-y|^2}{4(\t-t)^2}$ is the classical Pini-Watson kernel in $\R^{m+1}$.
		
	We observe that, as a consequence of \eqref{Eq:Gammatilde}, we can rewrite the super-level set $\tilde{\Omega}_r^{(m)}(\y,\z)$ in \eqref{Eq:omegatildem} as follows
	\begin{equation}\label{Eq:OmegaTilde}
		\tilde{\Omega}_r^{(m)}(\y,\z)= \Biggl\{ (y,z) \in \R^{m+N+1}: r\, \frac{\Gamma(\z,z)}{\big(4\pi(\t-t)\big)^{m/2}} > \exp\Bigg( \frac{|\y-y|^2}{4(\t-t)}\Bigg) \Biggr\}.
	\end{equation}
	In virtue of \eqref{Eq:defutilde}, we now rewrite equation \eqref{Eq:meanMtilde} as follows
	\begin{equation}\label{Eq:meanMtildeu}
		\begin{aligned}
			u(\z)= &\frac{1}{r} \bigintsss_{\tilde{\Omega}_r^{(m)}(\y,\z)} \tilde{M}^{(m)} (\y,\z, y,z) u(z)\, dy\, dz \\
			\quad+ &\frac{1}{r}  \bigintsss_0^{r} \left(
			\bigintsss_{\tilde{\Omega}^{(m)}_\r (\y,\z)} f(z) \,\left(\frac{1}{\r}-\tilde{\Gamma}^{(m)}(\y,\z,y,z) \right)dy  \, dz \right)\, d \r \\
			\quad+& \frac{1}{r} \bigintsss_0^{r}  \frac{1}{\r} \left(\bigintsss_{\tilde{\Omega}^{(m)}_\r (\y,\z)} \big( \div_{\GG} b(z) - c(z) \big) \,u(z)\,dy \, dz \right) \, d \r\\
			=:&\,I_1+I_2+I_3.
		\end{aligned}	
	\end{equation}
	We first focus on the term $I_1$. Exploiting \eqref{Eq:Mtildesum} and \eqref{Eq:OmegaTilde}, we get rid of the variable $y$ by integrating as follows
	\begin{equation}\label{Eq:I1}
		\begin{aligned}
			I_1&=\frac{1}{r}\bigintsss_{\tilde{\Omega}_r^{(m)}(\y,\z)} \left( M_\GG(\z,z)+\frac{|\y-y|^2}{4(\t-t)^2}\right) u(z)\, dy\, dz \\
			&=\frac{1}{r}\bigintsss_{{\Omega}_r^{(m)}(\z)}\bigintsss_{\lbrace |\y-y|<N_r(\z,z) \rbrace} \left( M_\GG(\z,z)+\frac{|\y-y|^2}{4(\t-t)^2}\right) u(z)\, dy\, dz\\
			&=\frac{1}{r} \bigintsss_{{\Omega}_r^{(m)}(\z)} M_r^{(m)}(\z,z) u(z)\, dz,
		\end{aligned}
	\end{equation}
	where the quantities $\Omega_r^{(m)}(\z)$, $N_r(\z,z)$ and $M_r^{(m)}(\z,z)$ are defined in \eqref{Eq:kernels-m}. In order to take care of $I_2$, we first recall the following equality 
	\begin{equation*}
		\int_0^R 	\mathrm{e}^{-\frac{r^2}{c}}r^{m-1}dr=\frac{c^{\frac{m}{2}}}{2}\int_0^{\frac{R^2}{c}}\mathrm{e}^{-r}r^{\frac{m}{2}-1}dr=\frac{c^{\frac{m}{2}}}{2}\gamma\left(\frac{m}{2};\frac{R^2}{c} \right),
	\end{equation*}
	where $\gamma$ is the lower incomplete gamma function (we recall the definition given in \eqref{Eq:gammafunction}). As a consequence of the previous equality and \eqref{Eq:OmegaTilde}, we obtain
	\begin{equation}\label{Eq:I2}
		\begin{aligned}
			I_2&=\frac{1}{r}  \bigintsss_0^{r} 
			\bigintsss_{{\Omega}^{(m)}_\r (\y,\z)}  \left(\bigintsss_{\lbrace |\y-y|< N_\r(\z,z) \rbrace}\left(\frac{1}{\r}-{\Gamma}(\z,z)\frac{\exp \left(-\frac{|\y-y|^2}{\big(4(\t-t)\big)^{\frac{m}{2}}} \right)}{\big(4\pi(\t-t)\big)^{\frac{m}{2}}} \right)dy  \right) f(z) dz  d \r \\
			&=\frac{1}{r} \bigintsss_0^{r} \left(\bigintsss_{{\Omega}^{(m)}_\r (\z)} W_\r^{(m)}(\z,z)f(z)\,dz\right)\, d\r,
		\end{aligned}
	\end{equation}
	where $W_\r^{(m)}(\z,z)$ was introduced in \eqref{Eq:kernels-m}. Finally, exploiting once again \eqref{Eq:OmegaTilde}, we infer
	\begin{equation}\label{Eq:I3}
		\begin{aligned}
			I_3&=\frac{1}{r} \int_0^{r}  \frac{1}{\r} \left(
			\int_{{\Omega}^{(m)}_\r (\z)}\big( \div_\GG b(z) - c(z) \big) \,u(z) \left(\int_{\lbrace |\y-y|< N_\r(\z,z) \rbrace}dy \right)  dz \right) \, d \r\\
			&=\frac{1}{r} \int_0^{r}  \frac{1}{\r} \left(
			\int_{{\Omega}^{(m)}_\r (\z)}\big( \div_\GG b(z) - c(z) \big) \,u(z) \, N_\r^m(\z,z)	\, \omega_m\,  dz \right) \, d \r.
		\end{aligned}
	\end{equation}
	The thesis follows combining \eqref{Eq:meanMtildeu}, \eqref{Eq:I1}, \eqref{Eq:I2} and \eqref{Eq:I3}.
\end{proof}

\section{Proof of Theorem \ref{Th-maximum}}\label{Maximum}
Let $\z\in\O$ such that $u(\z)=\max_\O u\geq 0$. Since $\H1=c$, from \eqref{e-meanvalue} we get
\begin{equation}
	\begin{aligned}\label{mvfdiff}
		0&=\boxed{\frac{1}{r}\int_{\O^{(m)}_r(\z)}M_r^{(m)}(\z,z)\big(u(z)-u(\z)\big)dz}_{=I_1}\\
		& +\boxed{\frac{1}{r}\int_{0}^r\frac{\omega_m}{\r}\left(\int_{\O^{(m)}_\r(\z)}N_\r^m(\z,z)\big(\div_{\GG} b(z)-c(z)\big)\big(u(z)-u(\z)\big)dz\right)d\r}_{=I_2}\\
		& +\boxed{\frac{1}{r}\int_{0}^r\left(\int_{\O^{(m)}_\r(\z)}\big(f(z)-\frac{\omega_m}{\r}u(\z)c(z)\big) W_\r^{(m)}(\z,z) dz\right)d\r}_{=I_3},\\
	\end{aligned}
\end{equation}
for every $r >0$ such that $\Omega^{(m)}_r(\z) \subset \Omega$. We now claim that for every $\zeta \in \Omega$ such that $u(\zeta)=\max_{\Omega}u$, we have
\begin{equation}\label{max-prog-ball}
	u(z)=u(\z),\qquad \text{for every }\,z\in \O^{(m)}_r(\z).
\end{equation}
In order to prove the claim, we take a closer look to the kernels appearing in \eqref{mvfdiff} (that we recall to be defined in \eqref{Eq:kernels-m}). By definition, we notice that $N_\r^m(\z,z),W_\r^{(m)}(\z,z)>0$ for every $z\in \O^{(m)}_r(\z)$. Moreover, thanks to the ellipticity assumption, the kernel $M_\GG$ defined in \eqref{Unbounded-kernel} is nonnegative, thus implying $M_r^{(m)}(\z,z) >0$ for every $z\in \O^{(m)}_r(\z)$. Since $f\leq 0$, $c\geq 0$, $\div_\GG b-c\geq 0$, and $u(\z)$ is a nonnegative maximum point, we deduce that $I_1,I_2,I_3\leq 0$. Hence, the three integrals in \eqref{mvfdiff} vanish and consequently, by the strictly positivity of the kernel, $u(z)-u(\z)=0$ for $\mathcal{H}^{N+1}$ almost every $z \in \O^{(m)}_r(\z)$. Claim \eqref{max-prog-ball} then follows from the continuity of $u$.

We now take advantage of claim \eqref{max-prog-ball} to conclude the proof of Theorem \ref{Th-maximum}. Let $z\in\A_\z(\O)$, and let $\gamma:[0,1]\to \O$ be an $\H$-admissible path such that $\gamma(0)=\z$ and $\gamma(1)=z$. We want to prove that $u\big(\gamma(s)\big)=u(\z)$ for every $s\in [0,1]$.

We set
\begin{equation}
	I:= \Big\{t\in[0,1] \,:\, u\big(\gamma(s)\big)=u(\z) \,\text{ for every }\, s \in [0,t] \Big\},
\end{equation}
and define $\bar t:=\sup I$. Clearly, $I\neq \emptyset$ since $0\in I$. Moreover, by the continuity of $u$ and $\gamma$, $I$ is closed, hence $\bar t\in I$. Suppose by contradiction that $\bar t < 1$, then we set $\bar z=(\bar x,\bar t):=\gamma\big(\bar t\big)$, and we recall that $\z\in\O$ and $u(\z)=\max_\O u$. We claim that there exists $r_1,s_1>0$ such that $\O^{(m)}_{r_1}(\bar z)\subset\O$, and
\begin{equation}\label{gammain}
	\gamma \big(\bar t + s\big)\in \O^{(m)}_{r_1}(\bar z), \qquad \text{for every }\,s\in[0,s_1).
\end{equation}
Then, as a consequence of \eqref{max-prog-ball}, we get $u\big(\gamma (\bar t + s)\big)= u(\bar z)=u(\z)$ for every $s\in(0,s_1)$, which contradicts $\bar t=\sup I$ and concludes the proof of our statement.

Thus, we are left with proving claim \eqref{gammain}. Since $\bar z\in\O$ by the definition of $\H$-admissible path, by Gaussian estimates we infer that there exists $\bar r$ small enough such that $\O^{(m)}_{\bar r}(\bar z)\subset\O$). Hence, to prove the claim we only need to show that there exist $r_1\leq \bar r$ and $s_1>0$ such that \eqref{portopalo-di-capo-passero} holds. From \emph{\ref{item5}.} and \emph{\ref{item8}.} in Theorem \ref{Th-Fundamental}, we infer the existence of $r_1\leq \bar r,c,\bar c>0$ (depending only on structural assumptions and $m$) such that 
\begin{equation}
	\O_{r_1}^{(m)}(\bar z) \supseteq \O_{r_1}^{*}(\bar z):=\Bigg\{(x,t)\in \rn\ti \big(\bar t-1,\bar t\,\big) : d_X(x,\bar x)^2\leq c \big(\bar t-t\big)\Big(\ln \big(r_1\bar c\big)-\ln\big(\bar t-t\big)\Big)\Bigg\}.
\end{equation}
Therefore, to prove \eqref{gammain} it is enough to show that there exists $s_1$ such that $\gamma \big(\bar t + s\big)\in \O^{*}_{r_1}(\bar z)$ for every $s<s_1$. By definition of $\H$-admissible path we have 
\begin{equation}
	\gamma\big(\bar t+s\big) =  \big(\gamma_*(\bar t + s),\bar t-s\big)
\end{equation}
for some  X-trajectory $\gamma_*$ with control function $\omega=(\omega_1,\ldots, \omega_{m_1})$, namely $\gamma_*\in \mathcal{C}$ such that
\begin{equation}	
	\dot{\gamma_*}(\t)=\sum_{i=1}^{m_1}\omega_i(\t)X_i\big(\gamma^*(\t)\big),\qquad \t\in [0,1].
\end{equation}
We set 
\begin{equation}
	\gamma^*(\t):=\gamma_*\big(\bar t+\t s\big),\qquad \t \in [0,1].
\end{equation}
Clearly, $\gamma^*$ is an X-trajectory with control function
\begin{equation}
	\omega^*(\t)=s\,\omega\big(\bar t+\t s\big),\qquad \t\in[0,1],
\end{equation}
satisfying $\gamma^*(0)=\gamma_*(\bar t)=\bar x$ and $\gamma^*(1)=\gamma_*(\bar t + s)$. Hence, by \eqref{eq:d1-dp} we obtain
\begin{equation}
	d_X\big(\gamma_*(\bar t + s),\bar x\big) =d^{(2)}_X\big(\gamma_*(\bar t + s),\bar x\big)\leq \sqrt{\int_{\bar t}^{\bar t +s}s |\omega(\t)|^2d\t}\leq \sqrt{s}\,\|\omega\|_{L^2(0,1)}.
\end{equation}
Consequently, $\gamma\big(\bar t+s\big)=\big(\gamma_*(\bar t + s), \bar t-s\big)\in \O^{*}_{r_1}(\bar z)$ provided we pick $s_1\leq r_1\bar c\, \exp\Big({-\|\omega\|_{L^2(0,1)}^2/c}\Big)$, and this concludes the proof of claim \eqref{gammain}.

To summarize the above, we proved that  $u(z)=u(\z)$ for every $z\in\A_\z(\O)$. By the continuity of $u$ we conclude that this holds also for every $z\in \overline{\A_\z(\O)}$. Finally, since $u$ is constant in $\overline{\A_\z(\O)}$, by $\H u=f$ we conclude that $u(\z)c(z)= f(z)$ for every $z\in \overline{\A_\z(\O)}$.

\medskip
We finally point out that the condition on the sign of $u(\z)$ was exploited only to ensure that $u(\z)c(z)$ has the desired sign. If $c\equiv 0$, the needed condition is always satisfied, hence we conclude that $f=0$. This concludes the proof of Theorem \ref{Th-maximum}.

\section{Parabolic Harnack inequality}\label{Parabolic Harnack inequality}
In this section we establish a \emph{parabolic-type} Harnack inequality, which will be crucial in deriving the \emph{invariant} Harnack inequality stated in Theorem \ref{Th-Harnack}. To this end, we first recall the notion of $\GG$-ball of  radius $r$ and center $x_0$ and of $\GG$-parabolic cylinder of  radius $r$ and center $z_0=(x_0,t_0)$
\begin{equation}\label{x-ball}
	B_r(x_0)=\Big\{ x \in \R^N \, : \, d_{X}(x,x_0) \leq r \Big\},\qquad 	Q_r(z_0)=B_r(x_0)\times \Big(t_0-r^2,t_0\Big),
\end{equation}
where $d_X$ denotes the Carnot-Carathéodory distance defined in Definition \ref{CC-distance}.
\begin{proposition}\label{Parabolic-Harnack}
	Let $\H$ be a differential operator of the form \eqref{HG1} satisfying assumptions \textbf{(H1)}, \textbf{(H2)}, \textbf{(H3)} and \textbf{(H4)}. Then, there exist four positive constants $C_P,r_1,\e_1,\vartheta_1$ (only depending on structural assumptions), with $\e_1$, $\vartheta_1 \in (0,1)$, such that for every $z_0\in\O$ and every $r>0$ with $r\leq r_1$ and ${Q}_{r}(z_0)\subset \O$, the following estimate holds
	\begin{equation}\label{harnack}
		\sup_{D_r(z_0)} u \leq C_P \, u(z_0),
	\end{equation}
	for every $u\geq0$ classical solution to $\H u =0$ in $\O$, where
	\begin{equation}\label{eq:def-D_r}
		D_r(z_0):=B_{\vartheta_1r}(x_0)\times \Big\{t= t_0 - \e_1 r^2\Big\}.
	\end{equation}	
\end{proposition}
\begin{figure}\label{figure-harnack-sets}
	\centering
	
	\begin{tikzpicture}[scale =1]
		
		\draw [brickred] (-5,0.7) -- (5,0.7);
		\draw [brickred] (-5,-2.5) -- (5,-2.5);
		\draw [brickred] (-5,0.7) -- (-5,-2.5);
		\draw [brickred] (5,0.7) -- (5,-2.5) node[anchor=west, scale=1]
		{$\O$};
		
		\draw [black] (-3,0) -- (3,0);
		\draw [black] (-3,-2) -- (3,-2);
		\draw [black] (-3,0) -- (-3,-2);
		\draw [black] (3,0) -- (3,-2) node[anchor=west, scale=1]
		{$Q_r(z_0)$};
		
		\filldraw (0,0) circle (1pt) node[above=1pt, scale=1] {$z_0$};
		
		\draw [forestgreen(web), thick] (-1,-1) -- (1,-1)
		node [anchor=east, xshift=-52pt, yshift=7pt] {\scalebox{1} {$D_r(z_0)$}};
	\end{tikzpicture}
	\caption{The sets appearing in the statement of Proposition \ref{Parabolic-Harnack} (the time variable is represented vertically, and upwardly increasing).}
\end{figure}
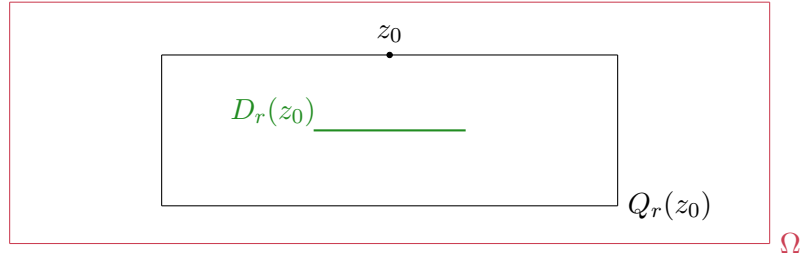
We first describe the strategy of the proof, which relies on the improved mean value formula derived in Proposition \ref{prop-meanvalue}. We begin by discussing a similar strategy employed for uniformly parabolic operators with Hölder continuous coefficients. Then we highlight the main difficulties arising in our framework and outline how we overcome them.

In the uniformly parabolic setting, a key properties of \eqref{e-meanvalue} is the boundedness of the kernel $M_r^{(m)}$: more precisely, there exist $m^-,m^+>0$ and $r_0,\e_0\in(0,1)$ such that
\begin{equation}\label{isola-delle-correnti}
	\begin{aligned}
		M_{r}^{(m)} \big(\x,\t,x,t\big) &\leq m^+r^{\frac{m-2}{m+N}},\qquad   &\text{for every } \, (x,t)\in \O^{(m)}_{r}\big(\x, \t\big),\\ M_{5r}^{(m)} \big(x_0,t_0,\x,\t\big) &\geq m^- r^m,\qquad &\text{for every } \,\big(\x,\t\big)\in  \O^{(m)}_{4r}(x_0, t_0)\cap \Big\{\t\leq t_0 - \e\,r^2\Big\},
	\end{aligned}
\end{equation}
whenever $r \in (0,r_0)$ and $\e \in (0, \e_0)$. In the constant coefficient case, \eqref{isola-delle-correnti} follows from the explicit expression of the corresponding fundamental solution. For instance, in the case of the heat equation in $\mathbb{R} \times (0,T)$, one has, for every $r> 0$ and for all $(\xi, \tau), (x, t) \in \mathbb{R} \times (0,T)$,
\begin{equation}
	\begin{aligned}
		M^{(m)}_r(\x,\t,x,t)&=\frac{2^m m \omega_m}{(m+2)(\t-t)}\Bigg((\t-t)\ln\Bigg(\frac{r}{(4\pi(\t-t))^\frac{m+1}{2}}\Bigg)-\frac{(\x-x)^2}{4}\Bigg)^\frac{m}{2}\cdot\\
		&\qquad \qquad \qquad \qquad\qquad \qquad \qquad \qquad\cdot\Bigg(\ln\Bigg(\frac{r}{(4\pi(\t-t))^\frac{m+1}{2}}\Bigg)-\frac{(\x-x)^2}{2m(\t-t)}\Bigg),
	\end{aligned}
\end{equation}
and
\begin{equation}
	\O^{(m)}_r(\x,\t)=\Bigg\{(\x-x)^2< 4(\t-t)\ln\Bigg(\frac{r}{(4\pi(\t-t))^\frac{m+1}{2}}\Bigg)\Bigg\}.
\end{equation}
It is straightforward to verify the validity of \eqref{isola-delle-correnti} when $m>2$. We now define the set
\begin{equation}
	K_{r,\e	}(x_0,t_0):=\overline{\O^{(m)}_r(x_0,t_0)}\cap \Big\{t\leq t_0 - \e\, r^2\Big\},
\end{equation}
for some $r \in (0,r_0)$ and $\e \in (0, \e_0)$. Then, thanks again to the explicit expression of the fundamental solution, it is easy to show that there exist constants $\vartheta, r_0,\e_0 \in (0,1)$ such that the following assertion holds 
\begin{equation}\label{eq:inclusione-pesante}
\O^{(m)}_{\vartheta r}(z) \subset \O^{(m)}_{4 r}(z_0) \cap \Big\{t\leq t_0 - \e\, r^2\Big\}
\end{equation}
for every $r \in (0,r_0)$, $\e \in (0, \e_0)$ and $z \in K_{r,\e	}(z_0)$. Then, the parabolic Harnack inequality follows from an elementary argument similar to the one employed for harmonic functions: let $\big(\x,\t\big)$ be a point in the set $K_{r,\e	}(x_0,t_0)$, for some $r \in (0,r_0)$ and $\e \in (0, \e_0)$, then there exists $\vartheta\in(0,1)$ such that
\begin{equation}
	\begin{aligned}
		u\big(\x,\t\big)\overset{\eqref{e-meanvalue}}=&\frac{1}{\vartheta r } \int_{\Omega_{\vartheta r}^{(m)}(\x,\t)} \tt M_{\vartheta r}^{(m)} \big(\x,\t,x,t\big) u(x,t) \, dx dt \\
		\overset{\eqref{isola-delle-correnti}}\leq\, &\frac{m^+{(\vartheta r)}^{\frac{m-2}{m+N}}}{\vartheta r} \int_{\Omega_{\vartheta r}^{(m)}(\x,\t)} u(x,t) \, dx dt\\
		\overset{\eqref{eq:inclusione-pesante}}\leq\,&\frac{m^+ {(\vartheta r)}^{\frac{m-2}{m+N}}}{\vartheta r} \int_{\O^{(m)}_{4 r}(x_0, t_0)\cap \left\{\t\leq t_0 - \e\,r^2\right\}} \!\!\!\!\!\!\!\!\!\!\!\!\!\!\!\!\!\!\!\!\!\!\!\!\!\!\!\!\!\! u(x,t) \, dx dt\\
		\overset{\eqref{isola-delle-correnti}}\leq\, &\frac{5m^+{(\vartheta r)}^{\frac{m-2}{m+N}}}{\vartheta m^-r^m }\frac{1}{5 r } \int_{\O^{(m)}_{5r }(x_0, t_0)}\!\!\!\!\!\!\!\!\!\!\!\!\!\!\! M_{5r}^{(m)} (x_0,t_0,x,t) u(x,t) \, dx dt \\
		\overset{\eqref{e-meanvalue}}\leq &\frac{5m^+ r_0^{\frac{m-2}{m+N}-m}}{m^-\vartheta^{\frac{2+N}{m+N}} } u(x_0,t_0).
	\end{aligned}
\end{equation}
The Harnack inequality \eqref{harnack} follows from the inclusion $D_r(z_0)\subseteq K_{5r,\e}^{(m)}(x_0, t_0)\subseteq  Q_{r}(z_0)$.

\medskip

In the case of H\"older continuous coefficients, the explicit expression of the fundamental solution $\G$ is no longer available. As explained above, the existence of $\Gamma$ is proved employing the parametrix method outlined in Section \ref{Fundamental}. A key feature of this method is that $\G$ can be written as the sum of the parametrix $Z$ and a perturbative term (see \eqref{G=Z+J}). Since the parametrix $Z$ is the fundamental solution  of an operator with constant coefficients, we have its explicit expression. Consequently, if one can show that $\Gamma$ is locally equivalent to $Z$, then the desired bounds for the kernel, equation \eqref{isola-delle-correnti}, follow as in the constant coefficients case. The proof of this local equivalence relies on a perturbative argument originally introduced by Polidoro in \cite{polidoro-parametrix} and recently employed in \cite{rebucci24}. This technique relies on a quantitative dependence of the fundamental solution on the principal part matrix $A$, which can be obtained thanks to the explicit form of the parametrix $Z$. More precisely, let $A^-, A^+$ be two symmetric positive definite matrices satisfying
\begin{equation}
	\frac{1}{\Lambda} \,\mathbb{I}_{N}\leq A^-,A^+\leq \Lambda \, \mathbb{I}_{N},\qquad \text{and} \qquad \mathbb{O}<A^+-A^-<\mu \, \mathbb{I}_{N},
\end{equation}
for some $\Lambda,\mu>0$. Here, $\mathbb{I}_{N}$ denotes the $N$-dimensional unit matrix and $\mathbb{O}$ the zero matrix. Denote by $\G_{A^{-}}$ and $\G_{A^{+}}$ the fundamental solutions associated to the constant coefficients operators
\begin{equation}
	\mathscr H^-=\div (A^-\nabla)-\p_t,\qquad \text{and}\qquad \mathscr H^+=\div (A^+\nabla)-\p_t.
\end{equation}
Then, for every $K>0$  the following estimates hold 
\begin{equation}\label{double-estimates}
	\begin{cases}
		\G_{A^-}(z,\z)\leq c_1\G_{A^+}(z,\z),\qquad &\forall z,\z \in \rnn,\\
		\G_{A^+}(z,\z)\leq c_2\G_{A^-}(z,\z)(t-\t)^{-\frac{Q}{2}c_3\mu},\qquad &\forall z,\z : \G_{A^+}(z,\z)\geq K,
	\end{cases}
\end{equation}
where $c_1$ and $c_3$ are structural constants, while $c_2$ depends also on $K$. 

\medskip

In order to apply Polidoro's perturbative method in our subelliptic framework, we need to prove estimates of the form \eqref{double-estimates} for the fundamental solution $\Gamma_A$ of the constant coefficients operator
\begin{equation}\label{boh-non-so-come-chiamarlo}
	{\mathscr{H}}_A=\sum_{i,j=1}^{m_1}a_{ij}X_iX_j-\p_t,
\end{equation}
for $A \in \textbf{M}_\l$. We  prove a version of the estimates in \eqref{double-estimates} by employing the results in \cite{bonfiglioli2004families, bonfiglioli2002uniform}, provided that $\GG$ is free. The key difficulty in the proof of \eqref{double-estimates} lies in the lack of an explicit expression for the fundamental solution of \eqref{boh-non-so-come-chiamarlo}, which prevents a direct derivation of the kernel bounds. However, when $\GG$ is a free Carnot group, this obstacle is partially overcome thanks to a result in \cite{bonfiglioli2004families}, which ensures the existence of a Lie group automorphism $T_A$ that turns $\mathscr H_A$ into the canonical heat operator $\Delta_\GG -\p_t$. More precisely, the fundamental solution $\G_A$ satisfies
\begin{equation}\label{automor}
	\G_A(z,\z)= \G_A\Big(\x^{-1}\circ x,t-\t\Big)=J_A\, \G_0\Big(T_A\big(\x^{-1}\circ x	\big),t-\t\Big),
\end{equation}
where $\G_0$ is the fundamental solution of the canonical heat operator, and $J_A = \big|\det \mathcal{J}_{T_{A}}(x)\big|$, with $\mathcal{J}_{T_A}$ denoting the Jacobian matrix of $T_A$. By exploiting the properties of the automorphism $T_A$, we are able to derive the second estimate in \eqref{double-estimates} when dealing with a free Carnot group. Regarding the first estimate in \eqref{double-estimates}, we are able to prove it in a weaker setting, namely when $A^-$ and $A^+$ are scalar multiples of the identity matrix (see Lemma \ref{sipuofaredipiu} below). However,  combining Lemma \ref{sipuofaredipiu} with the aforementioned automorphism, we are able to apply Polidoro's method and establish the local comparison between $\G$ and its parametrix, Proposition \ref{Prop-closetoparametrix}, in the context of a free Carnot group.

\medskip

To summarize, we have described the strategy to prove the parabolic Harnack inequality whenever $\GG$ is free. To address the general case, \emph{i.e.}, when $\GG$ is not assumed to be free, we adopt the following approach: we \emph{lift} $\GG$ to a free Carnot group $\tilde{\GG}$ (as explained in Section \ref{Carnot}); then, the Harnack inequality is obtained for the corresponding lifted operator $\HG$; finally, since every classical solution to $\H u \geq 0$ solves also $\HG u \geq 0$, the Harnack inequality obtained in the lifted group $\tt\GG$ directly implies the validity of Proposition \ref{Parabolic-Harnack} in the original group $\GG$.

\medskip

Let $\tt \GG$ be the \emph{free} Carnot group introduced in Section \ref{Carnot}, which lifts $\GG$, and let $\big\{\tt X_i\big\}_{i=1}^{m_1}$ be the corresponding set of vector fields. The lifted operator  associated to \eqref{HG1} is
\begin{equation}\label{HGlifted}
	\HG \tilde u(\tt z):=\sum_{i,j=1}^{m_1} \tilde X_i\Big(a_{ij}(z)\tilde X_j\Big)\tilde u(\tt z) + \sum_{i=1}^{m_1} b_i(z)\tilde X_i\tilde u(\tt z)+ c(z)\tilde u(\tt z)-\p_t \tilde u(\tt z),\quad \tt z\in \tilde{\O},
\end{equation}
where $a_{ij},b_i$ and $c$ are the coefficients appearing in \eqref{HG1}, assumed to satisfy \textbf{(H4)}. Here, $\tt z$ denote the point $(\tt x,x,t)\in \R^{\tt N}\ti \rn \ti \R$, and we have set $\tilde \O:= \rnm \ti \O$. Throughout the rest of this section, we will use the letters $\tt z_0$ and $\tt \z$ to denote respectively the points $(\tt x_0,x_0 ,t_0)$ and $ (\tt \x, \x,\t)$ in $\tilde \O$, and $z, z_0$ and $\z$ to denote respectively the points $(x ,t),(x_0 ,t_0)$ and $ (\x,\t)$ in $\O$.

Notice that the coefficients $a_{ij}, b_i$ and $c$ are bounded and $\a$-H\"older continuous with respect to the $\tt{\GG}$-parabolic metric induced by the distance $d_{\tt\GG}$ defined in \eqref{tilde-distance}. Hence, Theorem \ref{Th-Fundamental} applies to $\HG$, ensuing the existence of a fundamental solution $\tt\G$ to \eqref{HGlifted}, and Proposition \ref{prop-meanvalue}  yields the following representation formula for any $\tt u$ classical solution to $\HG \tt u = 0$ 
\begin{equation} \label{tilde-meanvalue} 
	\begin{aligned}
		\tt u(\tt \z) =& \frac{1}{r} \int_{\tt\Omega_r^{(m)}(\tt \z)} \tt M_r^{(m)} (\tt \z,\tt z) \tt u(\tt z) \, d\tt z\\
		+&\frac{1}{r} \int_0^{r} \frac{\omega_m}{\r} \Biggl(
		\int_{\tt \Omega^{(m)}_\r (\tt \z)}\tt N_\r^m(\tt \z,\tt z)\, \big( \div_\GG b(z) - c(z) \big) \,\tt u(\tt z) \, d\tt z   \Biggr) \, d \r.
	\end{aligned}
\end{equation} 
Here, $\tt\Omega_r^{(m)}$, $\tt M_r^{(m)}$, and $\tt N_r^{(m)}$  are the lifted analogues of the quantities introduced in \eqref{Eq:kernels-m}, in the context of $\tt \GG$.

The following observation will be crucial at the end of the section: if $u$ is a classical solution to $\H u =0$ in $\O$, then by the lifting property \eqref{lifting} we infer that the function $(\tt x,x,t)\to \tilde u(\tilde{x},x,t):=u(x,t)$ is a classical solution to $\HG \tilde u = 0$ in $\tilde \O$.

As previously mentioned, our goal is to establish bounds of the kind \eqref{isola-delle-correnti} for the kernel $\tt M_r^{(m)}$ in terms of $r$. To this end, we begin by deriving some estimates for the fundamental solution of the constant coefficients lifted operator
\begin{equation}\label{mazzareddi}
	\tt{\mathscr{H}}_A=\sum_{i,j=1}^{m_1}a_{ij}\tt X_i\tt X_j-\p_t,
\end{equation} 
focusing on its dependence on the coefficients matrix $A \in \textbf{M}_\l$.
\begin{lemma}\label{sipuofaredipiu}
	Assume that \textbf{(H1)},\textbf{(H2)} and \textbf{(H3)} hold, and pick $0<\mu^-< \m^+$. Denote by $\tt\G_{A^{-}}$ and $\tt\G_{A^{+}}$ the fundamental solutions of \eqref{mazzareddi} corresponding to the matrices $A^{-}:=\mu^- \mathbb{I}_{m_1}$ and $A^{+}:=\mu^+ \mathbb{I}_{m_1}$, respectively. Then, there exists $c_1>0$ such that the following inequality holds
	\begin{equation}\label{-minore+}
		\tt\G_{A^{-}}(\tt x, x,t) \leq c_1\, \tt\G_{A^{+}} (\tilde x,x,t),\qquad \forall(\tilde x,x,t) \in \rnm\times \rn \times (0,+\infty).
	\end{equation} 
\end{lemma}
\begin{proof}
	Since $\tt\G_{A^{-}}$ and $\tt\G_{A^{+}}$ are the fundamental solutions of 
	\begin{equation}
		\tt{\mathscr{H}}_{A^-}=\m^-\Delta_{\tt\GG}-\p_t,\qquad \tt{\mathscr{H}}_{A^+}=\m^+\Delta_{\tt\GG}-\p_t,
	\end{equation} 
	the thesis immediately follows from a homogeneity argument.
\end{proof}
\begin{lemma}\label{menominorepiu}
	Assume that \textbf{(H1)},\textbf{(H2)} and \textbf{(H3)} hold, and let $A^-,A^+$ be two positive definite, symmetric matrices satisfying 
	\begin{equation}\label{ellip-Lambda-pachino}
		\frac{1}{\Lambda}\,\mathbb{I}_{m_1}\leq A^-,A^+\leq \Lambda\,\mathbb{I}_{m_1},\qquad \text{and} \qquad \mathbb{O}<A^+-A^-<\mu \, \mathbb{I}_{m_1},
	\end{equation}
	for some $\Lambda,\m>0$. Denote by $\tt\G_{A^{-}}$ and $\tt\G_{A^{+}}$ the fundamental solution of \eqref{mazzareddi} corresponding to the matrices $A^{-}$ and $A^{+}$, respectively. Then, there exists $c_2,c_3>0$ such that for every $K>0$ it holds
	\begin{equation}\label{+minore-}
		\tt\G_{A^{+}}(\tilde x,x,t) \leq c_2 K^{-c_3\mu} t^{-\frac{\tt Q}{2} c_3 \mu}\, \tt\G_{A^{-}} (\tilde x,x,t),
	\end{equation}
	for every $(\tilde x,x,t) \in \rnm\times \rn \times (0,+\infty)$ such that $\tt\G_{A^{+}} (\tilde x,x,t)\geq K$. Here, we recall that $\tt Q$ denote the homogeneous dimension of the lifted group $\tt \GG$. 
\end{lemma}
\begin{proof}
	Since $\tt\GG$ is free, Theorem 1.1 in \cite{bonfiglioli2004families} yields the existence of two Lie group automorphisms $T_{A^{-}}$ and $T_{A^{+}}$ such that 
	\begin{equation}\label{diffeormophic-sublapalcians}
		\begin{aligned}
			\tt\G_{A^{-}}(\tilde x,x,t)= J_{A^-}\,\tt\G_0\big(T_{A^-}(\tt x,x),t)\big) = J_{A^-}\,\frac{\exp\Bigg(-\frac{ d_0 \big(T_{A^-}(\tt x,x)\big)^2}{t}\Bigg)}{t^{\frac{\tilde{Q}}{2}}},\\
			\tt\G_{A^{+}}(\tilde x,x,t)= J_{A^+}\,\tt\G_0\big(T_{A^+}(\tt x,x),t)\big) = J_{A^+}\, \frac{\exp\Bigg(-\frac{ d_0 \big(T_{A^+}(\tt x,x)\big)^2}{t}\Bigg)}{t^{\frac{\tilde{Q}}{2}}},
		\end{aligned}
	\end{equation} 
	where $\tt\G_0$ denotes the fundamental solution of the canonical heat operator $\sum_i\tt X_i^2 -\p_t$, $d_0$ is the corresponding homogeneous distance on $\tt\GG$, and $J_A = \det \mathcal{J}_{T_{A}}(\tt x,x)$, with $\mathcal{J}_{T_A}$ denoting the Jacobian matrix of $T_A$.  By Theorem 2.7 in \cite{bonfiglioli2004families} we infer that $J_A$ is constant, and that there exists $c_\Lambda>0$ (only depending on the ellipticity constant $\Lambda$) such that 
	\begin{equation}\label{two-ineq}
		\begin{aligned}
			&\qquad \qquad \qquad\;\; c_\Lambda^{-1} \leq J_{A^-}, J_{A^+} \leq c_\Lambda, \\
			&c_\Lambda^{-1}  d_0(\tt x,x) \leq d_0\big(T_{A^-}(\tt x,x)\big),d_0\big(T_{A^+}(\tt x,x)\big) \leq c_\Lambda d_0(\tt x,x), \\
			& d_0\big((T_{A^-}(\tt x,x))^{-1} \,\tt\circ\, T_{A^+}(\tt x,x)\big) \leq c_\Lambda \|A^+ - A^-\|^{\frac{1}{\kappa}} d_0(\tt x,x),
		\end{aligned}
	\end{equation}
	for any $(\tt x,x)\in\rnm\ti\rn$, where we recall that $\tt \circ$ is the composition law of $\tt \GG$, and $\kappa$ its step. From the second inequality in \eqref{two-ineq}, together with the assumption $\tt\G_{A^{+}} (\tilde x,x,t)\geq K$, we infer
	\begin{equation}\label{K-implies}
		\frac{ d_0 (\tt x,x)^2}{t}\leq c_\Lambda \ln \left(\frac{1}{t^{\frac{\tt Q}{2}}K}\right).
	\end{equation}
	The desired estimate \eqref{+minore-} then follows from \eqref{diffeormophic-sublapalcians}, \eqref{two-ineq}, and \eqref{K-implies}.
\end{proof}
\begin{remark}
	When dealing with uniformly parabolic operators, i.e. in the Euclidean setting, the estimate given by Lemma \ref{sipuofaredipiu} is known to hold even if $A^-$ and $A^+$ are more general matrices verifying \eqref{ellip-Lambda-pachino}. The same estimate has also been established for a broad class of degenerate Kolmogorov operators (see for instance \cite{polidoro-parametrix,rebucci24}). Moreover, some precise estimates obtained for the heat kernel in the Heisenberg group in \cite{lanconelli1,lanconelli2} seem to suggest the validity of such estimates for Carnot groups in general. A precise characterization of this property and its general validity remains an interesting open problem. 
\end{remark}
\begin{lemma}\label{derivatives-constant-operator}
	Assume that \textbf{(H1)},\textbf{(H2)} and \textbf{(H3)} hold, let $\frac{1}{\Lambda} \, \mathbb{I}_{m_1}< A < \Lambda \, \mathbb{I}_{m_1}$ be a symmetric positive matrix, and let $\tt\G_A$ be the fundamental solution of \eqref{mazzareddi} with matrix $A$. Then, there exists $c>0$ (depending only on structural assumptions and $\Lambda$) such that
	\begin{equation}\label{pachino}
		|X_i \tt\G_A (\tt x,x,t)|\leq c \left(\frac{1}{\sqrt{t}} \frac{d_{\tt \GG}(\tt x,x)}{\sqrt{t}}\right)\tt \G_A(\tt x,x,t),\qquad \forall (\tt x, x,t) \in \rnm\ti \rn\ti(0,+\infty).
	\end{equation}
\end{lemma}
\begin{proof}
	As shown in the proof of Lemma \ref{menominorepiu}, there exists a Lie group automorphism $T_A$ and an homogeneous distance ${d_0}$ (smooth out of the origin and homogeneous of degree one w.r.t. the dilations $\widetilde{\delta}_r$) such that
	\begin{equation}
		\tt\G_{A}(\tilde x,x,t)=J_A\frac{\exp\Bigg(-\frac{ d_0 (T_{A}\big(\tt x,x)\big)^2}{t}\Bigg)}{t^{\frac{\tilde{Q}}{2}}}.
	\end{equation}
	By the chain rule we get
	\begin{equation}\label{mirabella-imbaccari}
		\tt X_i \tt\G_{A}(\tilde x,x,t)= \tt\G_{A}(\tilde x,x,t)\left(-\frac{2}{\sqrt{t}} \frac{ d_0 \big(T_{A}(\tt x,x)\big) }{\sqrt{t}} \tt X_i  \Big(d_0 \big(T_{A}(\tt x,x)\Big) \right).
	\end{equation}
	From the second estimate in \eqref{two-ineq} together with the equivalence of $ d_0$ and $d_{\tt\GG}$ we infer
	\begin{equation}\label{kamarina}
		d_0 \big(T_{A}(\tt x,x)\big)\leq c\, d_{\tt\GG} (\tt x,x),
	\end{equation} 
	for some $c>0$ depending on $\Lambda$. Moreover, from Theorem 2.2 in \cite{bonfiglioli2004families} it follows that
	\begin{equation}
		 X_i  \Big(d_0 \big(T_{A}(\tt x,x)\Big)=\sum_{j=1}^{m_1} \Big(A^{-\frac{1}{2}}\Big)_{ij}X_jd_0\big(T_{A}(\tt x,x)\big).
	\end{equation}
	Since ${d_0}$ is homogeneous of degree one, its Lie derivatives are homogeneous of degree zero and hence bounded away from the origin. Therefore,  $X_i  \Big(d_0 \big(T_{A}(\tt x,x)\Big)$ is uniformly bounded for all $A$ satisfying the ellipticity bounds. 
	
	The lemma is an immediate consequence of \eqref{mirabella-imbaccari} and the above observations.
\end{proof}
We are now in position to establish a local comparison between the fundamental solution $\tt\G$ of the lifted operator \eqref{HGlifted} and its parametrix $\tt Z$. 
\begin{proposition}\label{Prop-closetoparametrix}
	Assume that \textbf{(H1)}, \textbf{(H2)}, \textbf{(H3)} and \textbf{(H4)} hold. Then, for every $\y>0$ there exists $K>0$ (depending only on $\y$ and structural assumptions) such that 
	\begin{equation}\label{close-to-parametrix}
		(1-\y)\tt Z(\tt z,\tt \z)\leq \tt\G(\tt z,\tt \z)\leq (1+\y)\tt Z(\tt z,\tt \z),
	\end{equation}
	for every $\tt z,\tt \z\in O_K:=\Big\{\tt Z(\tt z,\tt \z)\geq K\Big\}$.
\end{proposition}
\begin{proof}
	Without loss of generality we may assume $\tt \z =0$ and adopt the simpler notation $\tilde{\G}(\tilde{z},0)=\tilde{\G}(\tilde{z})$ and $\tilde{Z}(\tilde{z},0)=\tilde{Z}(\tilde{z})$. We also recall that we here employ the notation $\tilde{z}=(\tilde{x},x,t), \tt \z=(\tt \xi,\xi,\tau) \in \R^{\tt N}\ti \rn \ti \R$.

	We first assume that $A(0)=\mathbb{I}_{m_1}$, and then we explain how to remove such assumption. Let $\psi\in C^2(\R)$ be a cut-off function such that $\psi(w)=0$ if $w\geq 1$, and $\psi(w)=1$ if $w\leq 1/2$. For a fixed $r>0$ to be chosen later, we set $A^r(\tt z):= A(0) + \psi\Big(\tt \d_{\frac{1}{r}}d_{\tt \GG}(\tt z)\Big)\big(A(\tt z)-A(0)\big)$,  and consider the perturbed operator 
	\begin{equation}\label{HGperturbed}
		\HG^r \tilde u(\tt z):=\sum_{i,j=1}^{m_1} \tt X_i\Big(a^r_{ij}(\tt z)\tt X_j\Big)\tt u(\tt z) + \sum_{i=1}^{m_1} b_i(z)\tilde X_i\tilde u(\tt z)+ c(z)\tilde u(\tt z)-\p_t \tilde u(\tt z),\qquad \tt z\in \tilde{\O}.
	\end{equation}
	By Theorem \ref{Th-Fundamental}, $\HG^r$ has a fundamental solution $\tt \G^r$ given by 
	\begin{equation}
		\tt \G^r(\tt z) = \tt Z^r (\tt z) + \int_0^t \int_{\rnm\ti \rn} \tt Z^r(\tt z,\tt y, y, s) \tt G^r(\tt y,y, s,0) d\tt y \,dy \,d s,
	\end{equation}
	where $\tt Z^r$ is the parametrix, and $\tt G^r$ is given by
	\begin{equation}\label{gr}
		\tt G^r(\tt z,0)=\sum_{k=1}^\infty \big(\HG^r \tt Z^r\big)_k(\tt z,0),
	\end{equation}
	where $\big(\HG^r \tt Z^r\big)_1:=\HG^r \tt Z^r$ and 
	\begin{equation}\label{hgkr}
		\big(\HG^r \tt Z^r\big)_{k+1}(\tt z,0):=\int_0^t\int_{\rnm\ti \rn}  \big(\HG^r \tt Z^r\big)_1(\tt z, \tt y,y, s)\big(\HG^r \tt Z^r\big)_k(\tt y,y, s,0)d\tt y \,dy \, d s.
	\end{equation}
	Observe that, by construction, $A^r(0) = A(0)$, hence from the definition of parametrix we deduce $\tt Z^r(\tt z) = \tt Z(\tt z)$ for every $\tilde{z} \in \R^{\tt N}\ti \rn \ti \R$. For the same reason, by the properties of $\psi$ we get
	\begin{equation}\label{parametrix}
		\tt Z^r(\tt z,\tt \z) = \tt Z(\tt z,\tt \z),\qquad \text{if }\, d_{\tt \GG}(\tt \z)\leq \frac{r}{2}.
	\end{equation}
	We now prove the comparison \eqref{close-to-parametrix} by expanding \eqref{G=Z+J} as follows
	\begin{equation}\label{aggiungi-togli}
		\begin{aligned}
			\tt \G(\tt z) &= \tt Z (\tt z) + \int_0^t \int_{\rnm\ti \rn} \tt Z^r(\tt z, \tt y,y, s) \tt G^r(\tt y,y, s,0) d\tt y dy d s \\
			&\qquad \qquad+\int_0^t \int_{\rnm\ti \rn} \Big( \tt Z^r(\tt z, \tt y,y, s)-\tt Z(\tt z,\tt y,y, s)\Big)\tt G^r(\tt y,y, s,0) d\tt y dy d s\\
			&\qquad \qquad+\int_0^t \int_{\rnm\ti \rn} \tt Z^r(\tt z,\tt y,y, s) \Big(\tt G (\tt y,y, s,0) -\tt G^r(\tt y,y, s,0) \Big)d\tt y dy d s \\
			&=:  \tt Z (\tt z)+I+II+III.		
		\end{aligned}
	\end{equation}
	The proof of the proposition is completed once we show that $|I|,|II|,|III|\leq\frac{\y}{3}\tt Z (\tt z)$ on the set $\big\lbrace \tt z \in \rnm\ti \rn\ti\R : Z(\tt z)\geq K \big\rbrace$.

	Let us begin by estimating the term $I$. We start by fixing $T>0$ and noticing that, thanks to \eqref{gaussian-parametrix}, we can choose $K=K(T)>0$ such that
	\begin{equation}\label{time-estimate}
		O_K\subset \Big\{(\tt x,x, t)\in \rnm\ti \rn\ti\R : 0\leq t\leq T\Big\}.
	\end{equation}
	We now focus our attention on $\tt Z^r$ and $\tt G^r$. To this end we apply Lemma \ref{sipuofaredipiu} and Lemma \ref{menominorepiu} with matrices $A^-:=\mathbb{I}_{m_1}=A(0)=A^r(0)$ and $A^+:=(1+\mu) \, \mathbb{I}_{m_1}=A(0)+\mu \, \mathbb{I}_{m_1}$, for some $0<\mu<\lambda$ to be chosen later, where $\lambda$ is the ellipticity constant appearing in \eqref{def-ellipticity-class}. By definition $\tt Z(\tt z)=\tt\G_{A^-}(\tt z)$, and the aforementioned lemmas yield 
	\begin{equation}\label{new-estimate-z}
		\tt Z (\tilde x,x,t)\leq c_1 \tt \Gamma_{A^+}(\tilde x,x,t) , \qquad \forall(\tilde x,x,t) \in \rnm\times \rn \times (0,+\infty), 
	\end{equation} 
	and
	\begin{equation}\label{marina-di-ragusa}
		\tt\G_{A^{+}}(\tilde x,x,t) \leq c_2 K^{-c_3\mu} t^{-\frac{\tt Q}{2} c_3 \mu}\, \tt Z(\tilde x,x,t),
	\end{equation}
	in the set $O_K^+:=\Big\{\tt\G_{A^+}(\tt x,x,t) \geq K \Big\} $, for every fixed $K>0$.

	We now choose $r=r(M_1,\mu)$ such that
	\begin{equation}
		r^\a\leq \frac{\mu}{4 M_1},
	\end{equation} 
	where $\a$ and $M_1$ are given by \textbf{(H4)}. Then by the definition of $\psi$ we obtain 
	\begin{equation}
		 \Big(1-\frac{\mu}{2}\Big) \, \mathbb{I}_{m_1}=A(0) - \frac{\mu}{2} \, \mathbb{I}_{m_1}\leq A^r(\tt z)\leq A(0) + \frac{\mu}{2} \, \mathbb{I}_{m_1}= \Big(1+\frac{\mu}{2}\Big) \, \mathbb{I}_{m_1}<A^+, \quad \forall \tt z \in \rnm\ti \rn\ti\R.
	\end{equation}
	From this estimate we deduce that $A^r$ satisfies a uniform ellipticity condition with constant $1+\mu/2$. This fact is crucial, as Theorem B in \cite{bonfiglioli2002uniform} provides the following Gaussian upper bounds for the parametrix.
	\begin{equation}
		\tt Z^r(\tt z,\tt \z)\leq  c_\mu \tt\G_{\mu}(\tt z,\tt \z),
	\end{equation}
	for some $c_\mu>0$ depending solely on $\mu$, with $\tt\G_{\mu}$ denoting the fundamental solution of $\tt{\mathscr{H}}_{\mu}=(1+\m/2)\Delta_{\tt\GG}-\p_t$. Hence, by Lemma \ref{sipuofaredipiu} we deduce the following bound for the parametrix 
	\begin{equation}\label{Z-r-est}
		\tt Z^r(\tt z,\tt \z)\leq c_1 \tt\G_{A^+}(\tt z,\tt \z),
	\end{equation}
	in place of the general Gaussian upper bound \eqref{gaussian-parametrix}. Analogously, the derivatives $X_i \tt Z^r(\tt z,\tt \z)$ and $X_i X_j\tt Z^r(\tt z,\tt \z)$ can also be bounded by $\tt\G_{A^+}(\tt z,\tt \z)$ (as in \eqref{Gaussian-estimates-1} and \eqref{Gaussian-estimates-2}). Hence, by a careful inspection of the computation made in the proof of Theorem \ref{Th-Fundamental}, we deduce the following estimate 
	\begin{equation}\label{G-r-estimate}
		|\tt G^r(\tt z,\tt \z)|\leq \frac{c_T (t-\tau)^{\frac{\a}{2}-1}}{(t-\t)^{\frac{Q}{2}}}\tt\G_{A^+}(\tt z,\tt \z),\\
	\end{equation}
	for some $c_T$ depending only on $T$. Thus, by \eqref{Z-r-est}, \eqref{G-r-estimate}, and the reproduction property we infer that
	\begin{equation}
		|I|\leq \frac{2 c_T}{\a} T^\frac{\a}{2}\tt\G_A^{+}(\tilde x,x,t),\qquad \forall (\tt x,x, t)\in O_K.
	\end{equation}
	Finally, using \eqref{new-estimate-z}, we observe that $O_K \subset O^+_{K/c_4}$. This inclusion, together with \eqref{marina-di-ragusa}, yields
	\begin{equation}\label{playa-grande}
		|I|\leq \frac{2 c_T}{\a} T^\frac{\a}{2}\tt\G_A^{+}(\tilde x,x,t) \leq \frac{C_T}{\a} T^\frac{\a}{2} c_2  K^{-c_3\mu} t^{-\frac{\tt Q}{2} c_3 \mu}\, \tt Z(\tilde x,x,t),\qquad \forall (\tt x, x,t)\in O_K,
	\end{equation}
	implying the estimate 
	\begin{equation}\label{I-est}
		|I|\leq \frac{\y}{3} \tt Z(\tt x,x, t),\qquad \forall (\tt x,x, t)\in O_K,
	\end{equation}	
	provided we pick $\mu = \frac{\a}{\tt Qc_3}$ and $K=K(\H,\l,\eta)$ sufficiently large.

	We now turn to the estimate of term $|II|$ in \eqref{aggiungi-togli}. By \eqref{gaussian-parametrix} and \eqref{parametrix}, we infer that 
	\begin{equation}
		\sup_{(\tt y,y, s)\in\rnm\ti \rn\ti \R} \Big|\tt Z^r(\tt z, \tt y,y, s)-\tt Z(\tt z, \tt y,y, s)\Big|\leq c,
	\end{equation}
	for some $c>0$. Hence, from \eqref{G-r-estimate} and the normalization \eqref{normalization} of $\tt\G_{A^+}$ we deduce $|II|\leq C_{II} T^\frac{\a}{2}$, for some $C_{II}>0$ only depending on $\H$, $\l$, $r$ and $T$. In particular, this implies
	\begin{equation}\label{II-est}
		|II|\leq C_{II}T^\frac{\a}{2}K^{-1}\, \tt Z(\tilde x,x,t),\qquad \forall (\tt x,x, t)\in O_K.
	\end{equation} 
	Finally, we estimate $III$ in an analogous manner. By using the uniform bound
	\begin{equation}
		\big|\tt G (\tt y,y, s,0) -\tt G^r(\tt y,y, s,0)\big|\leq C, 
	\end{equation}
	whose proof follows the same argument as in Lemma 4.7 of \cite{polidoro-parametrix}, we obtain 
	\begin{equation}\label{III-est}
		|III|\leq C_{III} T,\qquad \forall (\tt x,x, t)\in O_K,
	\end{equation}	
	for some constant $C_{III}>0$ only depending on $\H$, $\l$, $r$ and $T$. Eventually combining \eqref{aggiungi-togli}, \eqref{I-est}, \eqref{II-est} and \eqref{III-est}, we infer
	\begin{equation*}
		\tt \G( \tt z)\leq \left(1+\frac{\eta}{3}+ \frac{C_{II}T^{\frac{\alpha}{2}}}{K}+\frac{C_{III}T}{K}\right)\tt Z(\tt z)\leq (1+\eta)\tt Z(\tt z),
	\end{equation*}  
	provided that we choose $K=K(\H,\l,\eta,T,\alpha)$ big enough. A symmetric argument yields the corresponding lower estimate $(1-\eta)\tt Z(\tt z)\leq \tt \G( \tt z)$.

	We now remove the assumption $A(0)=\mathbb{I}_{m_1}$. Let us consider the constant coefficients operator
	\begin{equation}
		\tt{\mathscr{H}}_{A_0}:=\sum_{i,j=1}^{m_1}a_{ij}(0)\tt X_i\tt X_j-\p_t,
	\end{equation} 
	and let $T_{A_0}(\tt x,x)$ denote the automorphism turning $\tt{\mathscr{H}}_{A_0}$ into the canonical heat operator (as follows from \cite{bonfiglioli2004families}). The change of variable $(x,t)\mapsto\big(T_{A_0}(\tt x,x),t\big)=:(\tt y,y,t)$ turns $\H$ into the operator
	\begin{equation}\label{Santa Maria del focallo}
		\overline{\HG} =\sum_{i,j=1}^{m_1} \tt X_i\big(\bar a_{ij}(\tt y,y,t)\tt X_j\big) + \sum_{i=1}^{m_1} \bar b_i(\tt y,y,t)\tt X_i+ c(\tt y,y,t)-\p_t ,
	\end{equation}
	where
	\begin{equation}
		\bar A (\tt y,y,t)=A^{-\frac{1}{2}}(0,0,0)A(\tt y,y,t)A^{-\frac{1}{2}}(0,0,0),\qquad \bar b(\tt y,y,t) = A^{-\frac{1}{2}}(0,0,0) b(\tt y,y,t).
	\end{equation}
	Clearly, $\overline{\HG}$ satisfies \textbf{(H1)}, \textbf{(H2)}, \textbf{(H3)} and \textbf{(H4)}, and since $\bar A (0,0,0)=\mathbb{I}_{m_1}$ the previous result ensures that, for every $\eta >0$, there exists a constant $K>0$ such that
	\begin{equation}
		(1-\y) \bar Z(\tt y,y,t)\leq \bar \G(\tt y,y,t)\leq (1+\y)\bar Z(\tt y,y,t),
	\end{equation}
	for every $(\tt y,y,t)\in \Big\{\bar Z(\tt y,y,t)\geq K\Big\}$. The conclusion then follows from the identities
	\begin{equation}
		\bar Z(\tt y,y,t)=\bar Z(T_{A_0}\big(\tt x,x),t)\big)= Z(\tt x,x,t),\qquad \bar \G(\tt y,y,t)=\bar \G\big(T_{A_0}(\tt x,x),t)\big)= \G(\tt x,x,t).
	\end{equation} 
	This completes the proof of the proposition.
\end{proof}
By a similar argument one can show that, locally, the Lie derivatives of $\tt\G$ can be estimated in terms of $\tt\G$ itself, rather than relying on the Gaussian upper bound \eqref{Gaussian-estimates-1}.
\begin{proposition}\label{Bound-derivative}	
	Assume that \textbf{(H1)},\textbf{(H2)},\textbf{(H3)} and \textbf{(H4)} hold. Then, there exist $c,K>0$ (depending only on structural assumptions) such that 
	\begin{equation}\label{bound-derivative}
		\begin{aligned}
			&|\tt X_i \tt \G(\tt z,\tt \z)|\leq c \left(\frac{1}{\sqrt{t-\t}}d_{\tt\GG}(\tt \x,\x,\tt x,x) + 1 \right) \tt \G(\tt z,\tt \z),\\
			&|\tt X_i \tt \G(\tt \z,\tt z)|\leq c \left(\frac{1}{\sqrt{t-\t}}d_{\tt\GG}(\tt x,x,\tt \x, \x) + 1 \right) \tt \G(\tt z,\tt \z),
		\end{aligned}
	\end{equation}
	for every $\tt z,\tt \z \in \tt O_K:=\big\{\tt\G(\tt z,\tt \z)\geq K\big\}$. 
\end{proposition}
\begin{proof}
	The proof is a consequence of Lemma \ref{derivatives-constant-operator}, and follows the same lines as the proof of Proposition \ref{Prop-closetoparametrix}. We suppose again, without loss of generality, that $\tt \z = 0$,  $A(0)=\mathbb{I}_{m_1}$,  and we define the perturbed operator $\HG^r$ as in \eqref{HGperturbed}. Thanks to \eqref{Gaussian-estimates-1} we can take the Lie derivative $\tt X_i$ under the integral sign in \eqref{aggiungi-togli}, obtaining 
	\begin{equation}\label{aggiungi-togli-derivate}
		\begin{aligned}
			\tt X_i\tt \G(\tt z) &= \tt X_i\tt Z (\tt z) +\! \!\int_0^t\! \!\int_{\rnm\ti \rn} \!\!\!\!\!\!\!\!\tt X_i\tt Z^r(\tt z, \tt y, y, s) \tt G^r(\tt y,y, s,0) d\tt y d y d s \\
			&\qquad \qquad +\!\!\int_0^t\! \!\int_{\rnm\ti \rn}\!\!\!\!\!\!\!\! \tt X_i\Big( \tt Z^r(\tt z, \tt y, y, s) -\tt Z (\tt z, \tt y,y, s) \Big)\tt G^r(\tt y,y, s,0) d\tt y d y d s\\
			&\qquad \qquad +\!\!\int_0^t\! \!\int_{\rnm\ti \rn} \!\!\!\!\!\!\!\!\tt X_i\tt Z^r(\tt z,\tt  y, y, s) \Big(\tt G (\tt y,y, s,0)-\tt G^r(\tt y,y, s,0)\Big)  d\tt y d y d s \\
			&=:\tt X_i\tt Z (\tt z)+ I' +II'+III'.		
		\end{aligned}
	\end{equation}
	We fix $\y=1/2$ and choose $T>0$. By Proposition \ref{Prop-closetoparametrix} we infer that there exists $K>0$ such that $\tt O_K\subset O_{\frac{2}{3}K}\subset  \Big\{(\tt x,x, t)\in \rnm\ti\rn\ti\R : 0\leq t\leq T\Big\}$. Arguing as in the proof of Proposition~\ref{Prop-closetoparametrix} and applying Lemma \ref{derivatives-constant-operator}, we infer that there exists $C>0$ such that
	\begin{equation}
		|I'| \leq \frac{C}{K^{c_3 \mu}}\frac{\tt Z(\tt x,x,t)}{\sqrt{t}},\qquad |II'|\leq C T^\frac{\a-1}{2}, \qquad |III'|\leq C \sqrt{T},\qquad \forall (\tt x,x,t)\in \tt O_{K},
	\end{equation}
	concluding the validity of the first inequality in \eqref{bound-derivative}. In order to prove the second inequality, we first point out that the differentiation is carried out with respect to the adjoint variable $\tt z$. This is then overcome by invoking the property $\tt \G(\tt \x, \x, \t,\tt x, x, t)=\tt \G^*(\tt x, x, t,\tt \x, \x, \t)$ established in Theorem \ref{Th-Fundamental}, and applying the first inequality in \eqref{bound-derivative}.  Finally, assumption $A(0)=\mathbb{I}_{m_1}$ is removed as in the proof of Proposition \eqref{Prop-closetoparametrix}. This concludes the proof of the proposition.	
\end{proof}
We are now in position to state and prove a parabolic Harnack inequality for the lifted operator \eqref{HGlifted}. For every $\tt z_0 = (\tt x_0,x_0, t_0)\in\rnm\ti\rn\ti\R$, $r,\e>0$ and $m\in\N$, we define the truncated kernel set
\begin{equation}\label{ktilde}
	\tt K_{r,\e}^{(m)}(\tt x_0, x_0, t_0):=\overline{\tt \O^{(m)}_r(\tt x_0, x_0, t_0)}\cap \Big\{t\leq t_0 - \e\, r^2\Big\}.
\end{equation}
\begin{proposition}\label{lifted-harnack}
	Let $\HG$ be a differential operator of the form \eqref{HGlifted} satisfying assumptions \textbf{(H1)}, \textbf{(H2)}, \textbf{(H3)} and \textbf{(H4)}. For every $m\in\N$ with $m>2$, there exists three positive constants $r_0,\e_0$ and $C_P$ (only depending on structural assumptions and $m$) such that the following holds: for every $(\tt x_0,x_0, t_0)\in\tt\O$, $0<r\leq r_0$, $0<\e\leq \e_0$ and $\tt\O^{(m)}_{5r}(\tt x_0, x_0, t_0)\subset \tt\O$ we have
	\begin{equation}\label{parabolic-Harnack-lifted}
		\sup_{\tt K_{r,\e}^{(m)}(\tt x_0, x_0, t_0)} \tt u \leq C_P \, \tt u(\tt x_0, x_0, t_0),
	\end{equation}
	for every $\tt u\geq0$ classical solution to $\HG \tt u =0$ in $\tt\O$.
\end{proposition}
\begin{proof}
	Let $m >2$, $(\tt x_0,x_0, t_0)\in\tt\O$ and $r >0$ such that $\tt \O^{(m)}_{4 r}\big(\tt x_0,x_0, t_0\big) \subset \tt \O$. We show that there exist $m^-,m^+>0$ and $r_0,\e_0,\vartheta\in (0,1)$ such that, for every $r\leq r_0$ and $\e\leq \e_0$, the following properties hold
	\begin{enumerate}
		\item $\tt K_{r,\e}^{(m)}(\tt x_0, x_0, t_0)\neq \emptyset$; \label{item1H}
		\item $ \tt M_{\vartheta r}^{(m)} \big(\tt \x,\x,\t,\tt x,x,t\big) \leq m^+r^{\frac{m-2}{m+\tt Q}}$ for every $(\tt x, x,t)\in \tt \O^{(m)}_{\vartheta r}\big(\tt \x, \x, \t\big)$; \label{item2H}
		\item $\tt \O^{(m)}_{\vartheta r}\big(\tt \x,\x, \t\big)\subset \tt \O^{(m)}_{4 r}(\tt x_0, x_0, t_0)\cap \Big\{\t\leq t_0 - \e\,r^2\Big\}$ for every $\big(\tt \x, \x,\t\big)\in K_{r,\e}^{(m)}(\tt x_0, x_0, t_0)$;\label{item3H}
		\item $\tt M_{5r}^{(m)} \big(\tt x_0, x_0,t_0,\tt x, x,t\big) \geq m^- r^m$ for every $\big(\tt x, x,t\big)\in \tt \O^{(m)}_{4r}(\tt x_0, x_0, t_0)\cap \Big\{t\leq t_0 - \e\,r^2\Big\}$.\label{item4H}
	\end{enumerate}
	Once these properties have been established, the proof of \eqref{parabolic-Harnack-lifted} is straightforward. Indeed, if $\div_{\GG} b-c=0$, for every $\big(\tt\x,\x,\t\big)\in\tt K_r^{(m)}(\tt x_0, x_0, t_0)$  we get
	\begin{equation}\label{simple-harnack}
		\begin{aligned}
			\tt u\big(\tt\x,\tt \x,\t\big)\overset{\eqref{tilde-meanvalue}}=&\frac{1}{\vartheta r } \int_{\tt\Omega_{\vartheta r}^{(m)}(\tt\x,\x,\t)} \tt M_{\vartheta r}^{(m)} \big(\tt\x,\x,\t,\tt x, x,t\big) \tt u(\tt x,x,t) \, d\tt x d x dt \\
			\overset{(\ref{item2H}.)}\leq\;& \frac{m^+  r^{\frac{m-2}{m+\tt Q}} }{\vartheta r} \int_{\tt\Omega_{\vartheta r}^{(m)}(\tt \x,\x,\t)} \tt u(\tt x, x,t) \, d\tt x d x dt\\
			\overset{(\ref{item3H}.)}\leq\;& \frac{m^+r^{\frac{m-2}{m+\tt Q}} }{\vartheta r} \int_{\tt \O^{(m)}_{4 r}(\tt x_0, x_0, t_0)\cap \left\{\t\leq t_0 - \e\,r^2\right\}} \tt u(\tt x, x,t) \, d\tt x dx dt\\
			\overset{(\ref{item4H}.)}\leq\;& \frac{5m^+ r^{\frac{m-2}{m+\tt Q}} }{\vartheta m^- r^m }\frac{1}{5 r } \int_{\tt \O^{(m)}_{5r }(\tt x_0, x_0, t_0)} \tt M_{5r}^{(m)}  \big(\tt x_0, x_0,t_0,\tt x, x,t\big) \tt u(\tt x, x,t) \, d\tt x d x dt\\\overset{\eqref{tilde-meanvalue}}\leq&\frac{5m^+  r_0^{\frac{m-2}{m+\tt Q}-m} }{\vartheta m^-} \tt u(\tt x_0, x_0,t_0),
		\end{aligned}
	\end{equation}
	and \eqref{parabolic-Harnack-lifted} follows with $C_P = \frac{5m^+  r_0^{\frac{m-2}{m+\tt Q}-m} }{\vartheta m^-}$.
	
	\medskip
	
	We now remove the assumption $\div_{\GG} b-c=0$. From \textbf{(H4)} we infer that there exists $k>0$ such that $|\div_{\GG} b-c|\leq k:=(m_1+1)M_1$. We introduce the auxiliary functions
	\begin{equation}
		\overline u(\tt x,x,t):=\mathrm{e}^{k(t-t_0)}\tt u(\tt x, x,t),\qquad \underline u(\tt x, x,t):=\mathrm{e}^{-k(t-t_0)}\tt u(\tt x,x,t).
	\end{equation}
	Since $\HG \tt u = 0$, we have that
	\begin{equation}
		\overline{\HG} \overline u: = \HG \overline u + k \overline u = 0,\qquad \underline{\HG} \underline u: = \HG \underline u - k \underline u = 0.
	\end{equation}
	We observe that $\overline{\HG}$ and $\underline{\HG}$ are operators of the form \eqref{HGlifted} with zero-order term $\overline{c}:=c+k$ and $\underline{c}:=c-k$, respectively. It is clear that \emph{\ref{item1H}.}-\emph{\ref{item4H}.} hold with the same constants $r_0,\e_0,\vartheta, m^-,m^+$. Moreover, by Gaussian estimates we infer that there exists $T,r^*>0$ such that $0<\t-t_0<T$ for every $\big(\x,\tt \x,\t\big)\in\tt\Omega_{5 r}^{(m)}(x_0,\tt x_0,t_0)$, provided $r<r^*$. Assuming without loss of generality that $r_0 \leq  r^*$, we infer that there exists $\overline{c}, \underline{c}>0$ such that 
	\begin{equation}\label{soprasottobarrato}
		\overline{c}	\tt u\big(\tt \x, \x,\t\big) \leq \overline{u}\big(\tt\x,\x,\t\big) \leq \,\tt u\big(\tt \x,\x,\t\big),\qquad \,\tt u\big(\tt \x, \x,\t\big) \leq \underline{u}\big(\tt \x, \x,\t\big) \leq \,\underline{c}\tt u\big(\tt \x, \x,\t\big),
	\end{equation} 
	for every $\big(\x,\tt \x, \t\big)\in \tt\Omega_{5 r}^{(m)}(x_0,\tt x_0,t_0)$. Let $\overline M_{r}^{(m)}$ and $\underline M_{r}^{(m)}$ denote the mean value formula kernels relevant to $\overline{\HG}$ and $\underline{\HG}$, respectively, and by $\underline\Omega_{r}^{(m)}$ and $\overline\Omega_{r}^{(m)}$ the corresponding super-level sets. Since $\div_{\GG} b - \overline{c}=\div_{\GG} b -c-k \leq 0$, the first two lines in \eqref{simple-harnack} become
	\begin{equation}
		\begin{aligned}
			\tt u\big(\tt \x, \x,\t\big) &\overset{\eqref{soprasottobarrato}}\leq\frac{1}{\overline{c}}\overline{u} \big(\tt \x,\x, \t\big)\\ 
			&\;\overset{\eqref{tilde-meanvalue}}\leq \frac{1}{\vartheta r \overline{c}} \int_{\overline{\Omega}_{\vartheta r}^{(m)}(\tt\x,\x,\t)} \overline{M}_{\vartheta r}^{(m)} \big(\tt\x,\x,\t,\tt x, x,t\big) \overline{u}(\tt x,x,t) \, d\tt x d x dt \\		
			&\;\;\overset{(\ref{item2H}.)}\leq\frac{m^+ r^{\frac{m-2}{m+\tt Q}} }{\vartheta r \overline{c}} \int_{\overline\Omega_{\vartheta r}^{(m)}(\tt\x,\x,\t)}  \overline u(\tt x, x,t) \, d\tt x  d x dt =: I.
		\end{aligned}
	\end{equation}
	Since \emph{\ref{item3H}.} holds also in the form $\overline\Omega_{\vartheta r}^{(m)}(\tt \x,\x, \t)\subset \underline \O^{(m)}_{4 r}(\tt x_0, x_0, t_0)\cap \big\{\t\leq t_0 \!-\!\e\, r^2\big\}$,  by using $\div_{\GG} b - \underline{c} =\div_{\GG} b-c+k\geq 0$ we conclude
	\begin{equation}
		I\leq \frac{5m^+r^{\frac{m-2}{m+\tt Q}}}{\vartheta m^-r^m }\,\frac{\underline{c}}{\overline{c}}\,\frac{1}{5 r } \int_{\underline \O^{(m)}_{5r }(\tt x_0, x_0, t_0)} \!\!\!\!\!\!\!\!\!\!\!\!\!\!\!\!\!\!\!\!\underline M_{5r}^{(m)}\big(\tt x_0, x_0,t_0,\tt x, x,t\big) \underline u(\tt x, x,t) \, d\tt x d x dt \leq \frac{5m^+}{\vartheta m^-} \,\frac{\underline{c}}{\overline{c}}\, r_0^{\frac{m-2}{m+\tt Q}-m} \,\tt u(\tt x_0, x_0,t_0),
	\end{equation}
	where the last inequality follows again from \eqref{tilde-meanvalue} and \eqref{soprasottobarrato}.  Hence, inequality \eqref{parabolic-Harnack-lifted} follows with $C_P = \frac{5m^+}{\vartheta m^-} \frac{\underline{c}}{\overline{c}}\, r_0^{\frac{m-2}{m+\tt Q}-m}$. 
	
	\medskip
	
	To conclude the proof of Proposition \ref{lifted-harnack}, we are now left with proving claims \emph{\ref{item1H}.}-\emph{\ref{item4H}.}

	\medskip
	\emph{\ref{item1H}.} We prove that the set 
	\begin{equation}
		\tt K_{r,\e}^{(m)}(\tt x_0, x_0, t_0)=\overline{\biggl\{\frac{\tt \Gamma(\tt x_0, x_0, t_0,\tt x, x,t)}{\big( 4\pi(t_0-t)\big)^{m/2}}  >\frac{1}{r} \biggr\}}\cap \Big\{t\leq t_0 - \e\, r^2\Big\},
	\end{equation}
 	is non-empty for every $r,\e\leq 1$. From the Gaussian lower bound \eqref{Gaussian-estimates-lower} we infer that there exists $c,c_l>0$ (depending on structural assumptions) such that $\tt K_{r,\e}^{(m)}(\tt x_0, x_0, t_0)\subset \mathcal K$, where
	\begin{equation}
		\mathcal K:=\overline{\Biggl\{\frac{c}{\big (t_0-t\big)^{\frac{\tt Q}{2}}} \;\exp{\left(-c_l\frac{d_{\tt\GG}(\tt x_0, x_0,\tt x, x)^2}{(t_0-t)}\right)}\frac{1}{\big( 4\pi(t_0-t)\big)^{m/2}}  >\frac{1}{r} \Biggr\}}\cap \Big\{t\leq t_0 - \e\,r^2\Big\}. 
	\end{equation}
	We now show that $\mathcal K$ is non-empty. Let $t:=t_0-\e\,r^2$, and notice that any $(\tt x, x)\in \rnm\ti\rn$ satisfying
	\begin{equation}
		d_{\tt\GG}(\tt x_0, x_0,\tt x, x)\leq \frac{1}{\e\,r^2c_\l}\ln\left(\frac{(4\pi)^\frac{m}{2}}{c}\big(\e\,^2r\big)^{m+\tt Q-1}\right),
	\end{equation}
	belongs to $\mathcal K$. Therefore $\mathcal K$ is non-empty, and so is $\tt K_{r,\e}^{(m)}(\tt x_0, x_0, t_0)$.

	\medskip
	\emph{\ref{item2H}.} We show that there exist $r_0,m^+>0$ such that, for every $r\leq r_0$, it holds 
	\begin{equation}
		\tt M_{ r}^{(m)}\big(\tt \x,\x,\t,\tt x,x,t\big)\leq m^+ r^{\frac{m-2}{m+\tt Q}},\qquad \text{for every }\, (\tt x, x,t)\in \tt \O^{(m)}_{r}\big(\tt \x, \x, \t\big).
	\end{equation}
	Recall that the explicit formula for the kernel $\tt M_{ r}^{(m)}$ is given by 
	\begin{equation}\label{kernel-formula}
		\begin{aligned}
			\tt M_{ r}^{(m)}\big(\tt \x,\x,\t,\tt x, x,t\big)&=\frac{2^{m} m\omega_m}{m+2}(\t-t)^\frac{m-2}{2}\left(\ln\left( r\, \frac{\tt \Gamma\big(\tt \x, \x, \t,\tt x, x,t\big)}{\big( 4\pi(\t-t)\big)^{m/2}} \right)\right)^\frac{m+2}{2}\\
			&\quad+2^m \omega_m \left((\t-t)\ln\left( r\, \frac{\tt \Gamma\big(\tt \x, \x, \t,\tt x, x,t\big)}{\big( 4\pi(\t-t)\big)^{m/2}} \right)\right)^\frac{m}{2}\cdot\\
			&\qquad\qquad \qquad \cdot\frac{\langle A(x,t)\nabla_{\tt \GG} \tt \Gamma\big(\tt \x, \x, \t,\tt x, x,t\big),\nabla_{\tt \GG}\tt \Gamma\big(\tt \x, \x, \t,\tt x, x,t\big) \rangle}{{\tt \Gamma\big(\tt \x, \x, \t,\tt x, x,t\big)}^2}\\
			&=:K_1+K_2.
		\end{aligned}
	\end{equation}
	We first focus on the term $K_2$. To this end, we show that there exist constants $r_0>0$ and $c>0$ such that, for every $r\leq r_0$, the following inequality holds 
	\begin{equation}\label{bound-kernel}
		\frac{\langle A(x,t)\nabla_{\tt \GG} \tt \Gamma\big(\tt \x, \x, \t,\tt x, x,t\big),\nabla_{\tt \GG}\tt \Gamma\big(\tt \x,\x, \t,\tt x, x,t\big) \rangle}{{\tt \Gamma\big(\tt \x, \x, \t,\tt x, x,t\big)}^2}\leq \frac{c}{(\t-t)},\qquad (\tt x, x,t)\in \tt \O^{(m)}_{r}\big(\tt \x, \x, \t\big).
	\end{equation}
	To prove \eqref{bound-kernel} we apply Proposition \ref{Bound-derivative}.  We claim that this proposition remains valid when replacing  $\big\{\tt\G\geq K\big\}$  with the level set $\tt\Omega_{1/K}^{(m)}$. Indeed, Proposition \ref{Bound-derivative} applied to  $\tt \G^{(m)}$ on the set $\Big\{\tt\G^{(m)}(0,\tt \x,\x,\t,\tt x,x,y,t\Big)\geq K\big\}\supseteq\tt\Omega_{1/K}^{(m)}\big(\tt \x,\x,\t\big)$ implies that
	\begin{equation}
		\begin{aligned}
			H^{(m)}(0,\t,0,t)\tt X_i\tt\G\big(\tt \x, \x, \t,\tt x, x,t\big)& = X_i\tt\G^{(m)}(0,\tt \x, \x, \t,0,\tt x, x,y,t)\\
			&\leq  c \left(\frac{1}{\sqrt{t-\t}}d_{\tt\GG}\big(\tt \x,  \x,\tt x,x \big) + 1 \right) \tt \G^{(m)}\big(0,\tt \x,\x ,\t,0,\tt x, x ,t\big)\\
			&= c \left(\frac{1}{\sqrt{t-\t}}d_{\tt\GG}\big(\tt \x,  \x,\tt x, x\big) + 1 \right) H^{(m)}(0,\t,0,t)\tt\G\big(\tt \x, \x, \t,\tt x, x,t\big),				
		\end{aligned}
	\end{equation}
	proving the claim. Let $K$ be the constant given by the aforementioned proposition, and set $r_0:=1/K$. By the ellipticity assumption given by \textbf{(H4)} and Proposition \ref{Bound-derivative}, we finally conclude the validity of estimate \eqref{bound-kernel}.

	We focus now on $K_1$. By Proposition \ref{Prop-closetoparametrix} and \eqref{time-estimate}, we deduce that, if we pick $r_0$ small enough, we can choose $T>0$ such that
	\begin{equation}
		\tt\Omega_{r}^{(m)}(\tt\x,\x,\t)\subseteq \Big\{(\tt x, x, t)\in \rnm\ti\rn\ti\R : 0\leq \t-t\leq T\Big\}, \qquad \text{for every }\,r\leq r_0.
	\end{equation}
	Then, by \eqref{Gaussian-estimates} we infer that there exists $c=c(r_0)>0$ such that
	\begin{equation}\label{Pozzallo}
		K_1\leq \frac{2^m m \omega_m}{m+2}(\t-t)^\frac{m}{2}\left(\ln\Bigg(\frac{c\,r}{(\t-t)^\frac{m+\tt Q}{2}}\Bigg) \right)^\frac{m+2}{2}.
	\end{equation}
	Combining \eqref{kernel-formula}, \eqref{bound-kernel}, and \eqref{Pozzallo}, the validity of \emph{\ref{item2H}.} follows.

	\medskip
	\emph{\ref{item3H}.} The key ingredient to prove this inclusion is the parabolic dilation relevant to $\tt\GG$. Indeed, by Proposition \ref{Prop-closetoparametrix} we infer that there exist $r_0,\e_0>0$ such that for every $r\leq r_0$ and $\e\leq \e_0$ it holds
	\begin{equation}\label{casinalbo}
		\tt K_{r,\e}^{(m)}(\tt x_0, x_0, t_0)\subset \overline {\tt \O^{(m)^*}_{3 r}(\tt x_0, x_0, t_0)}\cap \Big\{t\leq t_0 -\e\, r^2\Big\},
	\end{equation} 
	where
	\begin{equation}
		\tt \O^{(m)^*}_{r}(\tt x_0,x_0, t_0):=\Bigg\{(\tt x, x ,t) \in \rnm\ti\rn\ti \R :\frac{\tt Z(\tt x_0, x_0, t_0,\tt x, x ,t)}{\big( 4\pi(t_0-t)\big)^{m/2}}>\frac{2}{r} \Bigg\},
	\end{equation}
	with $\tt Z$ being the parametrix of $\tt \G$. We claim that there exists $\vartheta \in (0,1)$ such that 
	\begin{equation}\label{parabolic-inclusion}
		\tt \O^{(m)^*}_{3\vartheta r}(\tt x, x, t)\subset \tt \O^{(m)^*}_{4 r}(\tt x_0, x_0, t_0)\cap \Big\{\t\leq t_0 - \e\,r^2\Big\},
	\end{equation}
	for every $(\tt x,x, t)\in  \overline {\tt \O^{(m)^*}_{3 r}(\tt x_0, x_0, t_0)}\cap \left\{t\leq t_0 - \e\, r^2\right\}$. The validity of \emph{\ref{item3H}.} follows then by Proposition \ref{Prop-closetoparametrix}, that implies $\tt \O^{(m)^*}_{4r}(\tt x_0,x_0, t_0)\subset\tt \O^{(m)}_{4r}(\tt x_0, x_0, t_0)$. To prove claim \eqref{parabolic-inclusion}, we employ the parabolic dilation relevant to $\tt\GG=\Big(\rnm\times \rn,\tt \circ,\tt \d_r\Big)$: by the homogeneity of the parametrix we infer that for every positive $k$ it holds 
	\begin{equation}\label{scaling}
		\Big((\tt x_0,x_0)\;\tt\circ \;\tt{\d_{r}}(\tt x,x),t_0+r^2t\Big)\in \tt \O^{(m)^*}_{k r}(\tt x_0,x_0, t_0)\iff \Big((\tt x_0,x_0)\;\tt\circ \;(\tt x,\tt x),t_0+t\Big)\in \tt \O^{(m)^*}_{k}(\tt x_0, x_0, t_0).
	\end{equation}
	Hence, it is sufficient to show \eqref{parabolic-inclusion} whenever $r=1$. Let $d(\tt x_0, x_0,t_0)$ denote the distance between the compact set $\overline {\tt \O^{(m)^*}_{3}(\tt x_0, x_0, t_0)}\cap \left\{t\leq t_0 - \e\right\}$ and the boundary of $\tt \O^{(m)^*}_{4}(\tt x_0, x_0, t_0)$. Clearly, this defines a positive function continuously depending on $(\tt x_0, x_0,t_0)$. In particular, it depends on the  automorphism $T_{A(\tt x_0, x_0,t_0)}$, associated with the matrix $A$ evaluated at the point $(\tt x_0, x_0,t_0)$ (see \eqref{automor}). By \eqref{two-ineq}, we infer that there exists $d_0,\e_0>0$ (depending only on structural assumption) such that $d(\tt x_0, x_0,t_0)\geq d_0$ for every $(\tt x_0, x_0,t_0)\in \tt\O$ and $\e\leq\e_0$. Since the diameter of $\tt\O_1^{(m)^*}(\tt x, x,t)$ is uniformly bounded, by \eqref{scaling} we deduce that it is possible to find $\vartheta \in (0,1)$ such that the diameter of $\tt\O_\vartheta^{(m)^*}(\tt x,x,t)$ is not greater than $d_0$, and then we conclude the proof of the claim.

	\medskip
	\emph{\ref{item4H}.} We recall again that the explicit expression of the kernel is given by \eqref{kernel-formula}. By the ellipticity assumption \textbf{(H4)} we infer $K_2\geq0$. To prove the estimate is therefore sufficient to show that $K_1\geq m^-r^m$. This follows once we set  
	\begin{equation}
		m^-:=\frac{2^mm\omega_m}{4(m+2)}\left(\ln(5/4) \right)^\frac{m+2}{2}.
	\end{equation}
\end{proof}
Taking advantage of the previous proposition, we are eventually able to prove Proposition \ref{Parabolic-Harnack}.
\\\\\newenvironment{pf-parabolic-harnack}{\noindent {\sc Proof of Proposition \ref{Parabolic-Harnack}.}}{\hfill $\square$}
\begin{pf-parabolic-harnack}
	Let $u\geq 0$ be a classical solution to $\H u = 0$ in $\O\subset\rn\ti\R$, and for every $(x,t)\in \O$ and $\tt x \in\rnm$ set $\tt u(\tt x,x,t):= u(x,t)$.  We notice that, by \eqref{lifting}, $\tt u$ solves $\tt \H \tt u = 0$ in $\tt \O:=\rnm \ti \O $.

	Choose $m=4$, and let $r_0,\e_0$ and $C_P$ be the corresponding constant given by Proposition \ref{lifted-harnack}. We claim that there exists $0<r_1\leq r_0<1$, $0<\e_1\leq \e_0<1$, $0<\vartheta_1<1$, and $C>0$ such that
	\begin{equation}\label{first-inclusion}
		\{0\}\ti D_r(z_0)=\{0\}\ti B_{\vartheta_1r}(x_0)\times \Big\{t-t_0=-\e_1 r^2\Big\}\subseteq \tt K_{r,\vartheta_1\e}^{(4)}(0,x_0, t_0),
	\end{equation}
	\begin{equation}\label{second-inclusion}
		\tt K_{5r, \vartheta_1\e}^{(4)}(0,x_0, t_0)\subseteq \tt Q_{r}(0,x_0,t_0)=\tt B_{r}(0,x_0)\times \Big\{0\leq t_0-t\leq r^2\Big\},
	\end{equation}
	\begin{equation}\label{third-inclusion}\
		\tt Q_{r}(0,x_0,t_0)\subseteq  \rnm\ti Q_{r}(x_0,t_0) \subseteq \rnm\ti \O=\tt \O,
	\end{equation}
	for every $(x_0,t_0)\in \O$, $\e\leq \e_1$ and $r\leq r_1$ such that $Q_{r}(x_0,t_0) \subset  \O$. Once the claim is proved, we are in position to apply Proposition \ref{lifted-harnack} to the lifted function $\tt u$, obtaining
	\begin{equation}
		\sup_{D_r(x_0, t_0)} u = \sup_{\{0\}\ti D_r(x_0, t_0)} \tt u \,\leq \sup_{\tt K_{r,\vartheta_1\e}^{(4)}(0,x_0,t_0)} \tt u \leq C_P \, \tt u(0,x_0,t_0) = C_P\, u(x_0, t_0),
	\end{equation}
	for every $(x_0,t_0)\in \O$. This proves the desired parabolic Harnack inequality.

	We prove the inclusions in the claim by employing Gaussian estimates and the boundedness of the set $\tt \O^{(m)}$. By Theorem \ref{Th-Fundamental} we infer that there exists two positive constants $c_\l$ and $\tt c$ such that 
	\begin{equation}
		\tt \G(0,x_0,t_0,\tt x, x,t)\geq \frac{\tilde c}{\big (t_0-t\big)^{\frac{\tt Q}{2}}} \;\exp{\left(-c_l\frac{d_{\tt\GG}(0,x_0,\tt x, x)^2}{(t_0-t)}\right)},
	\end{equation} 
	for every $(0,x_0,t_0),(\tt x,x,t)\in \rnm\ti\rn\ti \R$ such that $0\leq t_0-t\leq 1$. By the definition \eqref{tilde-distance}  of the distance $d_{\tt \GG}$ we deduce 
	\begin{equation}
		\tt \G(0,x_0,t_0,\tt x,x,t)\geq \frac{\tilde c}{\big (t_0-t\big)^{\frac{\tt Q}{2}}} \;\exp{\left(-c_\l\frac{d_X(x_0,x)^2}{(t_0-t)}\right)}\exp{\left(-c_\l\frac{\sum_{i=1}^\kappa \sum_{j=1}^{\tt m_i} |\tt x_j^{(i)}|^{\frac{2}{i}}}{(t_0-t)}\right)}=:I.
	\end{equation}
	If $(\tt x, x, t)\in \{0\}\ti D_r(z_0)$, then $t_0-t=\e_1r^2\leq 1$, $d_{X}(x_0,x)^2\leq \vartheta_1 r$ and $\tt x=0$, this implying
	\begin{equation}
		I\geq \frac{\tt c}{\e_1^{\tt Q/2} r^{\tt Q}}\exp{\Big(-c_\l\vartheta_1^2/\e_1\Big)},
	\end{equation}
	concluding that for every fixed $\vartheta_1\in(0,1)$ \eqref{first-inclusion} holds for a suitable choice of $\e_1$ and for every $r\leq r_*$, for some $r_*\in(0,1)$.

	The proof of \eqref{second-inclusion} follows by a similar argument: combining the Gaussian upper bound \eqref{Gaussian-estimates} with \eqref{ktilde}, \eqref{casinalbo}, and the scaling property \eqref{scaling}
	, we infer that there exist $r^*,\vartheta_1\in(0,1)$ such that \eqref{second-inclusion} holds for every $r\leq r^*$. Inclusion  \eqref{third-inclusion} is a direct consequence of the definition of the lifted distance \eqref{tilde-distance}, which completes the proof of the claim by picking $r_1\leq \min\big(r_0,r^*,r_*\big)$.
\end{pf-parabolic-harnack}

\section{Proof of Theorem \ref{Th-Harnack}}\label{Harnack}
We can first observe that it is not restrictive to assume $z_0=0$ and $r=1$. We let $r_1$, $\vartheta_1 $, $\e_1$ and $C_P$ be the constants given by Proposition \ref{Parabolic-Harnack}, and we set
\begin{equation}\label{thetanot}
	\vartheta_0:=\min\Bigg\{\vartheta_1,\frac{1}{12 k^3_1}\Bigg\},
\end{equation}
where $k_1$ is the constant given by the pseudo-triangular inequality \eqref{k1}. Moreover, we choose four positive constants $\nu, \eta, \mu, \vartheta$ with $0 < \nu < \eta < \mu < 1$ and $0 < \vartheta \leq \vartheta_0$, and consider the cylinders $Q^+$ and $Q^-$ introduced in \eqref{unit-box+-}. Finally, we set
\begin{equation}\label{rnot}
	r_0:= \min \Bigg\{ r_1,\frac{1}{2k_1},\sqrt{1-\gamma}\Bigg\}.
\end{equation} 
We next consider $z^-:=\big(x^-,t^-\big) \in Q^-= B_{\vartheta}(0) \times (-\nu,0)$, and $z^+:=\big(x^+,t^+\big) \in Q^+=B_{\vartheta}(0) \times (-\mu,-\eta)$, and exploit Proposition \ref{Parabolic-Harnack} to construct a \textit{Harnack chain}, i.e a finite set of points $\lbrace z_0,z_1,\ldots,z_m\rbrace \in Q_1(0)$ such that
\begin{equation}\label{eq:def-chain}
	z_0=z^+, \qquad z_m=z^-, \qquad Q_r(z_j)\subset Q_1(0)\subset \O,\qquad z_j \in D_r(z_{j-1}),  
\end{equation}
for every $ j=1,\ldots,m$ and $r \leq r_0$. If \eqref{eq:def-chain} holds, by iterating Proposition \ref{Parabolic-Harnack} we conclude that
\begin{equation}
	u\big(z^-\big)\leq C_P^m u\big(z^+\big).
\end{equation}
Once we prove that $m$ is uniformly bounded w.r.t. $z^+ \in Q^+$ and $z^- \in Q^-$, we conclude the proof of the theorem.

In order to prove \eqref{eq:def-chain}, we need to make a suitable choice of $z_j$'s. To this end, we consider a $X$-trajectory $\underline{\gamma}:[0,1] \rightarrow \Omega$ in $\mathcal{C}$ steering $x^+$ to $x^-$, i.e. $\underline{\gamma}$ is a solution to the following system 
\begin{equation}
	\begin{cases}
		\displaystyle\dot{\underline{\gamma}}(\t)=\sum_{i=1}^{m_1}\underline\alpha_i(\t)X_i\Big(\underline{\gamma}(\t)\Big), \\
		\underline{\gamma}(0)=x^+, \, \underline{\gamma}(1)=x^-,
	\end{cases}
\end{equation}
where $\underline\alpha=(\underline{\a}_1,\ldotp,\underline{\a}_{m_1})$ is the control function satisfying condition \eqref{condition-alphaj}. We further require $\underline{\gamma}$ to be optimal in the $d^\infty_X$ metric, meaning that
\begin{equation}\label{gamma-opt}
	d^\infty_X\big(x^+,x^-\big) = \sup_{\t\in [0,1]} |\underline\a(\t) |.
\end{equation}
Since $\underline{\gamma}$ connects only the space variables $x^+$ and $x^-$, we need to suitably incorporate the time variables $t^+$ and $t^-$ into the construction. To this end, we reparametrize this curve by stretching it over the time interval $[0,t^+-t^-]$, \emph{i.e.} we define the X-curve $\tt\gamma : \big[0,t^+-t^-\big]\to \O$ as 
\begin{equation}
	\tt \gamma(\t) := \underline{\gamma}\Bigg(\frac{\t}{t^+-t-}\Bigg),\qquad \t\in \big[0,t^+-t^-\big],
\end{equation}
whose control function $\tt \a$ is given by 
\begin{equation}\label{riparametrix}
	\tt \a (\t)=\frac{1}{t^+-t^-}\,\underline{\a}\Bigg(\frac{\t}{t^+-t^-}\Bigg),\qquad \t\in [0,t^+-t^-]. 
\end{equation}
The following observation will be crucial in the sequel. For every $s \in \big[0,t^+-t^-\big]$, let $\widehat \gamma:[0,1]\to \O$ denote the X-trajectory $\widehat\gamma(\t)=\tt \gamma (\t s)$. We denote its control function as $\widehat \a (\t)=s\,\tt\a(\t s)$, and we notice that $\widehat \gamma (0)=x^+$ and $\widehat \gamma (1)=\tt\gamma(s)$. Then, the following inequality holds true
\begin{equation}\label{despar}
	d_X\big(\tt\gamma(s),x^+\big)\leq\int_0^1|\widehat \a(\t)|d\t =  \int_0^s|\tt\a(\t)|d\t.
\end{equation}
We employ $\tt \gamma$ to choose the $z_j$'s: for an arbitrary $m\in\N$, we choose a positive $r$ such that
\begin{equation}\label{eq:def-r}
	\varepsilon_1 r^2=\frac{t^+-t^-}{m},
\end{equation}
and for every $j=0,\ldots,m$ we define
\begin{equation*}
	z_j=\big(x_j,t_j\big):=\Big(\tt{\gamma}\big(j\e_1r^2\big),t^+-j\e_1r^2\Big),
\end{equation*}
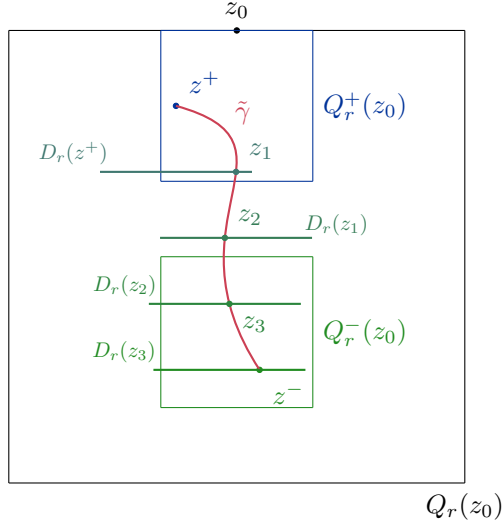
\begin{figure}\label{figure-harnack-chain}
	\centering
	\begin{tikzpicture}[scale=1, every node/.style={font=\small}]
		
		\draw [black] (-3,0) -- (3,0);
		\draw [black] (-3,-6) -- (3,-6);
		\draw [black] (-3,0) -- (-3,-6);
		\draw [black] (3,0) -- (3,-6) node[anchor=north, scale=1] {$Q_r(z_0)$};

		\filldraw (0,0) circle (1pt) node[above=1pt, scale=1] {$z_0$};
		
		\draw [darkpowderblue] (-1,0) -- (1, 0);
		\draw [darkpowderblue] (-1,-2) -- (1, -2);
		\draw [darkpowderblue] (-1,0) -- (-1, -2);
		\draw [darkpowderblue] (1,0) -- node[anchor=west, scale=1] {$Q^+_r(z_0)$} (1, -2);
		\filldraw [darkpowderblue] (-0.8,-1) circle (1pt) node[above right=1pt, scale=1] {$z^+$};
		
		\draw [forestgreen(web)] (-1,-3) -- (1, -3);
		\draw [forestgreen(web)] (-1,-5) -- (1, -5);
		\draw [forestgreen(web)] (-1,-3) -- (-1, -5);
		\draw [forestgreen(web)] (1,-3)  -- node[anchor=west, scale=1] {$Q^-_r(z_0)$} (1, -5);
		\filldraw [forestgreen(web)] (0.3,-4.5) circle (1pt) node[below right=1pt, scale=1] {$z^-$};
		
		\draw[brickred, thick] (-0.8,-1) .. controls (1,-1.5) and (-1,-2.5) .. (0.3,-4.5); 
		\node[brickred, anchor=west, xshift=-13pt, yshift=97pt] at (0.3,-4.5) {$\tt\gamma$};
		
		\filldraw [intermedio3] (-0.0099,-1.875) circle (1pt) node[above right=1pt] {$z_1$};
		\draw [intermedio3, thick] (-1.8,-1.875) -- (0.2,-1.875) 
		node [anchor=east, xshift=-50pt, yshift=7pt] {\scalebox{0.8} {$D_r(z^+)$}};
		\filldraw [intermedio2] (-0.1571,-2.75)  circle (1pt) node[above right=1pt] {$z_2$};
		\draw [intermedio2, thick] (-1.0099,-2.75) -- (0.9901,-2.75) 
		node [anchor=east, xshift=25pt, yshift=5pt] {\scalebox{0.8} {$D_r(z_1)$}};
		\filldraw [intermedio1] (-0.0970,-3.625) circle (1pt) node[below right=1pt] {$z_3$};
		\draw [intermedio1, thick] (-1.1571,-3.625) -- (0.8429,-3.625)
		node [anchor=east, xshift=-50pt, yshift=7pt] {\scalebox{0.8} {$D_r(z_2)$}};
		\draw [forestgreen(web), thick] (-1.0970,-4.5) -- (0.9030,-4.5)
		node [anchor=east, xshift=-52pt, yshift=7pt] {\scalebox{0.8} {$D_r(z_3)$}};
	\end{tikzpicture}
	\caption{Harnack chain (the time variable is represented vertically, and upwardly increasing).}
	\label{harnack-chain}
\end{figure}
see Figure \ref{harnack-chain}. We first observe that from the previous definition it follows
\begin{equation}\label{eq:harnack-points}
z_j=\Big(\tt{\gamma}_{x_{j-1}}\big(\e_1 r^2\big),t_{j-1}-\e_1 r^2\Big),
\end{equation}
where $\tt{\gamma}_{x_{j-1}}$ is the section of the curve $\tt{\gamma}$ steering $x_{j-1}$ to $x_j$, i.e. $\tt{\gamma}_{x_{j-1}}:\big[0,\e_1 r^2\big] \rightarrow \Omega$ with $\tt{\gamma}_{x_{j-1}}(0)=x_{j-1}$, $\tt{\gamma}_{x_{j-1}}\big(\e_1 r^2\big)=x_j$ for every $j=1,\ldots,m$.

It is clear that $z_0=z^+$, $z_m=z^-$. Hence, to prove \eqref{eq:def-chain} we begin by showing that for every $s\in \big[0,t^+-t^-\big]$, $\vartheta\leq \vartheta_0$ and $r\leq r_0$ it holds $B_r\big(\tt{\gamma}(s)\big)\subset B_1(0)$. This implies $Q_r(z_j) \subset Q_1(0)$ if $r\leq r_0$, in virtue of the choice of $\gamma$ and $t_j$'s. Let $y\in B_r\big(\tt{\gamma}(s)\big)$, then for every $r\leq r_0$ it holds
\begin{equation}
	\begin{aligned}
		d_X(y,0)\overset{\eqref{k1}}\leq& k_1 d_X\big(y,\tt{\gamma}(s)\big)+ k_1^2 d_X\big(x^+,0\big) + k_1^2d_X\big(\tt{\gamma}(s),x^+\big)\\
		\overset{\eqref{despar}}\leq & k_1 r+ k_1^2 \vartheta + k_1^2\int_{0}^{s}|\tt\a(\t)| d\t\\
		\leq\,\;&k_1 r+ k_1^2 \vartheta+ k^2_1\big(t^+-t^-\big) \sup_{\t\in[0,t^+-t^-]}|\tt\a(\t)|,
		\end{aligned}
\end{equation}
where in the last inequality we simply used the fact that $s\in \big[0,t^+-t^-\big]$. We then compute
\begin{equation}
	\begin{aligned}		
		d_X(y,0)\overset{\eqref{riparametrix}}\leq&k_1 r+ k_1^2 \vartheta+ k^2_1 \sup_{\t\in[0,1]}|\underline\a(\t)|
		\\
		\overset{\eqref{gamma-opt}}=&k_1 r+ k_1^2 \vartheta+ k^2_1 d_X^{(\infty)}(x^+,x^-)\\
		\overset{\eqref{eq:d1-dp}}=&k_1 r+ k_1^2 \vartheta+ k^2_1d_X(x^+,x^-)\\
		\overset{\eqref{rnot}}\leq &\frac{1}{2} +3k_1^3\vartheta\overset{\eqref{thetanot}}<1,
	\end{aligned}
\end{equation}
and therefore the trajectory $\tt \gamma$ we constructed does not exit $B_1(0)$. We conclude by showing the last inclusion in \eqref{eq:def-chain}: in virtue of \eqref{eq:def-D_r} and \eqref{eq:harnack-points}, $z_j \in D_r(z_{j-1})$ if $d_X(x_j,x_{j-1})\leq \vartheta_1 r$. Arguing as in \eqref{despar} we get
\begin{equation}
	\begin{aligned}
		d_X(x_j,x_{j-1})&\leq\bigintsss_{{(j-1)}\e_	1r^2}^{j\e_1r^2}|\tt \a(\t)| d\t\\
		&\!\!\!\!\;\overset{\eqref{eq:def-r}}\leq  \frac{t^+-t^-}{m} \sup_{\t\in[0,t^+-t^-]}|\tt \a(\t)| \\&\!\!\!\!\;\overset{\eqref{riparametrix}}=\frac{1}{m} \sup_{\t\in[0,1]}| \underline\a(\t) |		\\
		&\!\!\!\!\; \overset{\eqref{gamma-opt}}= \frac{1}{m} d^{(\infty)}_X(x^+,x^-)\\
		&\!\!\!\!\;\overset{\eqref{eq:d1-dp}}=\frac{1}{m}d_X(x^+,x^-).
	\end{aligned}
\end{equation}
Hence, $z_j \in D_r(z_{j-1})$ if we choose $m\in\N$ such that
\begin{equation}\label{primam}
	\frac{1}{m} d_X(x^+,x^-) \leq \vartheta_1 r\overset{\eqref{eq:def-r}} = \frac{\vartheta_1 \sqrt{t^+-t^-}}{\sqrt{m}\sqrt{\e_1}} \, \iff \, m \geq \frac{ d_X(x^+,x^-)^2\e_1}{\vartheta_1^2(t^+-t^-)}.
\end{equation}
The proof of the theorem follows once we obtain a bound for $m$ that is uniform w.r.t. $z^+ \in Q^+$ and $z^- \in Q^-$. In order for the previous construction to hold, we need to satisfy \eqref{primam} and $r\leq r_0$. Since $t^+-t^- \geq \eta-\nu$, we need to require $m \geq \frac{4k^2_1\e_1}{(\eta-\nu)}$ so that \eqref{primam} holds. Moreover, by \eqref{eq:def-r} condition $r \leq r_0$ holds if $m \geq \frac{t^+-t^-}{\e_1 r^2_0}$, and therefore it is sufficient to require that $m \geq \frac{\mu}{\e_1 r_0^2}$ since $t^+-t^- \leq \mu$. To summarize the above, we have proved that inequality \eqref{eq:thm-harnack} holds with $C_H=C_P^{\bar m}$, provided that we choose $\bar m \in\N$ such that $\bar m \geq \max \Big\lbrace \frac{4^2k_1\e_1}{(\eta-\nu)},\frac{\mu}{\e_1 r_0^2} \Big\rbrace$.

\medskip\medskip\medskip\medskip
{\bf Acknowledgments.} The authors would like to thank Prof. Sergio Polidoro for suggesting the problem and for many useful discussions, and Prof. Ermanno Lanconelli for pointing out the references \cite{lanconelli1,lanconelli2}.

Giulio Pecorella is member of the research group “Gruppo Nazionale per l’Analisi Matematica, la Probabilità e le loro Applicazioni” of the Italian “Istituto Nazionale di Alta Matematica”.

\appendix
\section{Comparison principle}\label{Comparison}
In this appendix we provide the proof of the following comparison principle.
\begin{proposition}\label{Prop-comparison}
	Let $\O_T=\O\times (T_1,T)$ be an open bounded subset of $\rnn$, and let $w\in C^2_{x,t}(\O_T)$ be a solution to
	\begin{equation}	
		\begin{cases}
			\H w\geq 0,\qquad \;\;&\mathrm{in}\;\O_T,\\
			\limsup w \leq 0, \qquad &\mathrm{in}\;\p\O\times(T_1,T)\cup \O \times\{T_1\},
		\end{cases}
	\end{equation}
	then $w\leq 0$ in $\O_T$. The same result holds in $\O=\rn$ if 	$\limsup w \leq 0$ in $\rn \times\{T_1\}$ and at infinity.
\end{proposition}
\begin{proof}
	The scheme of the proof is classical, once we follow the $C^2_{x,t}$ regularity of $w$. For every $k\in \N$ we set $\O_{T_k} := \O_T\cap\left\{t\leq \big(T-\frac{1}{k}\big)\right\}$, and we prove the result by showing that $w\leq 0$ in $\O_{T_k}$ for every $k \in \N$. The result on $\rn$ straightforwardly follows.

	Let us pick some $M>M_1$ (where $M_1$ is given by \textbf{(H4)}), and set $v(x,t):=w(x,t)\exp({-M t})$. By a direct computation we see that $v$ solves $	\H v - M v \geq 0$ in $\O_{T_k}$, and that $\limsup v \leq 0$ in $ \p\O\times(T_1,T)\cup \O \times\{T_1\}$. Since the function $v(\cdot,t)$ has continuous first and second order Lie derivatives w.r.t. the generating vector fields $X_i$, then it has continuous Lie derivative w.r.t. every $X\in \span\{X_i\}_{i=1}^{m_1}$(see Lemma 4.2 in \cite{bonfiglioli2004fundamental}). Therefore, if $v\in C^2_{x,t}(\O_T)$ has a positive maximum point in $(x_0,t_0)\in \O\times \left\{t = T_k\right\}$, we claim that
	\begin{equation}\label{maximum-point}
		X_iv(x_0,t_0) = 0,\qquad  \p_t 	v(x_0,t_0)\leq 0,\qquad \sum_{i,j=1}^{m_1}a_{ij}(x_0,t_0)X_iX_j v(x_0,t_0)\leq 0,
	\end{equation}
	whereas when $(x_0,t_0)\in \O_{T_k}$ the second inequality becomes an equality. From this claim it follows $\big(M-c(x_0,t_0)\big)v(x_0,t_0)\leq 0$, hence $v\leq 0$ in $\O_{T_k}$, whereby the thesis follows.

	We prove the claim by proving the last inequality in \eqref{maximum-point}, since the others are straightforward: by contradiction, assume there exists $d \in \R^{m_1}$ such that $\sum_{i,j=1}^{m_1}d_{i}d_jX_iX_j v(x_0,t_0)> 0$. Then, let $\exp(sZ)(0)$ be the integral curve  of $Z=\sum_{i=1}^{m_1}d_iX_i$ starting from the origin, and for every $h\in\R$ set $g(h):= v\big(x_0 \circ \exp(hZ)(0),t_0\big)$. Since $v(\cdot,t_0)$ has continuous second order Lie derivatives, it follows $g \in C^2(-\sigma,\sigma)$ for a sufficiently small $\sigma>0$. Therefore, recalling that $(x_0,t_0)$ is a maximum point, by a Taylor expansion we get
	\begin{equation}
		\begin{aligned}
			v(x_0,t_0)\geq g(h)=& g(0) + \frac{h^2}{2}g''(0) + o_{h\rightarrow 0}\big(h^2\big)\\
			=&\, v(x_0,t_0) + \frac{h^2}{2}\sum_{i,j=1}^{m_1}d_{i}d_jX_iX_j v(x_0,t_0)+o_{h\rightarrow 0}\big(h^2\big) >  v(x_0,t_0),
		\end{aligned}
	\end{equation}
	hence the contradiction.   
\end{proof}

\bibliographystyle{plain}
\bibliography{Bibtex}

\end{document}